\documentclass[final,leqno,onefignum,onetabnum]{siamltex1213}

\usepackage{amssymb}
\usepackage{amsmath}
\usepackage{epsfig}
\usepackage{psfrag}
\usepackage{pstricks}
\usepackage{algorithm}
\usepackage{graphicx}

\usepackage{subfigure}
\allowdisplaybreaks

\newtheorem{eg}{Example}[section]

\title{The Continuity of Images by Transmission Imaging Revisited
\thanks{This work was supported by the National Natural Science Foundation of China (NSFC) No. 11301289.}}

\author{Zhitao Fan\thanks{Department of Mathematics, National University of Singapore, Singapore.}
\and Feng Guan\thanks{Department of Mathematics, University of California, Los Angeles, USA} \and
Chunlin Wu\thanks{School of Mathematical Sciences, Nankai University, PRC} \and
Ming Yan\thanks{Department of Mathematics, University of California, Los Angeles, USA} }

\begin{document}
\maketitle
\slugger{sima}{2013}{xx}{x}{x--x}

\begin{abstract}
\textit{Transmission} imaging, as an important imaging technique widely used in astronomy, medical diagnosis, and biology science, has been shown in~\cite{Wu13} quite different from \textit{reflection} imaging used in our everyday life. Understanding the structures of images (the prior information) is important for designing, testing, and choosing image processing methods, and good image processing methods are helpful for further uses of the image data, e.g., increasing the accuracy of the object reconstruction methods in transmission imaging applications. In reflection imaging, the images are usually modeled as discontinuous functions and even piecewise constant functions. In transmission imaging, it was shown very recently in~\cite{Wu13} that almost all images are continuous functions. However, the author in~\cite{Wu13} considered only the case of parallel beam geometry and used some too strong assumptions in the proof, which exclude some common cases such as cylindrical objects. In this paper, we consider more general beam geometries and simplify the assumptions by using totally different techniques. In particular, we will prove that almost all images in transmission imaging with both parallel and divergent beam geometries (two most typical beam geometries) are continuous functions, under much weaker assumptions than those in~\cite{Wu13}, which admit almost all practical cases. Besides, taking into accounts our analysis, we compare two image processing methods for Poisson noise (which is the most significant noise in transmission imaging) removal. Numerical experiments will be provided to demonstrate our analysis.
\end{abstract}

\begin{keywords}
transmission imaging, reflection imaging, Radon transform, parallel beam geometry, divergent beam geometry, continuity, measure zero, Poisson noise removal
\end{keywords}

\begin{AMS}
92C55, 90C90, 68U10
\end{AMS}

\pagestyle{myheadings}
\thispagestyle{plain}
\markboth{Z. Fan, F. Guan, C. Wu, and M. Yan}{The Continuity of Images by Transmission Imaging Revisited}

\section{Introduction}
Imaging is an important technique which translates a physical scene to lower dimensional (typically 2D) data for convenient observation and record. It has been applied to many fields, including our everyday life, medical diagnosis, exploring the universe, and biological structure analysis. Many imaging systems and instruments, such as various digital cameras, X-ray computed tomography (CT), telescopes, and microscopes, have been developed. Different imaging systems are based on different physical principles. Digital cameras used in our everyday life record the reflection part of the incoming light~\cite{Mumford89}, whereas transmission electron microscopes generate images by counting the electrons having transmitted the scene \cite{Kak88,Frank96,Natterer01}. We refer to these two kinds of imaging techniques by \textit{reflection} imaging and \textit{transmission} imaging in this paper for clarity. See Fig.~\ref{fig-RefTranImaging}. We will consider in this paper the most common case that objects are in $\mathbf{R}^3$ and images are 2D data. The reflection imaging is meaningful by itself, while the transmission imaging is not and the final objective of transmission imaging is to reconstruct the 3D object (density function) from many 2D images.

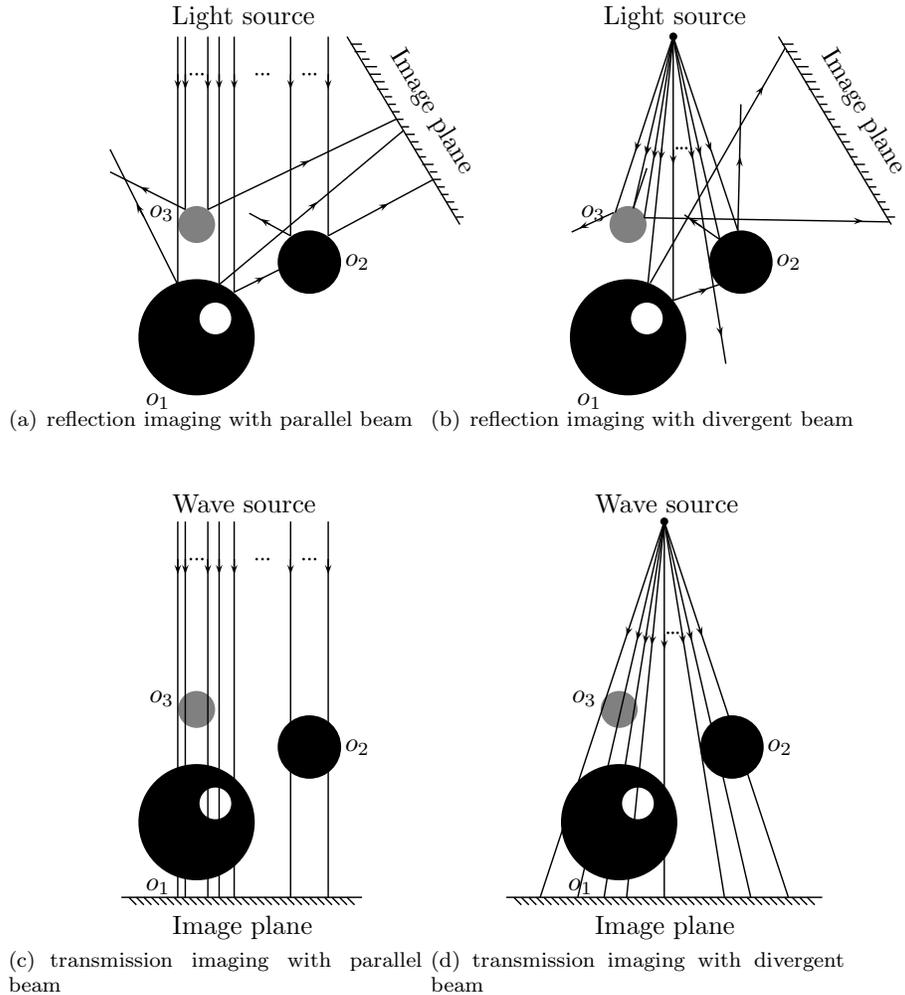
\begin{figure}[htb]
\begin{center}
\subfigure[reflection imaging with parallel beam]{
\psset{xunit=0.5cm,yunit=0.5cm,linewidth=0.19mm}
\begin{pspicture}(-5,-1.5)(6,10)
  \psset{linecolor=black}
  \qdisk(0,0){22pt}
  \psset{linecolor=white}
  \qdisk(0.5,0.5){6pt}
  \uput[0](-1.7,-1.7){$o_1$}
  \psset{linecolor=black}
  \qdisk(3,2){12pt}
  \uput[0](3.6,2){$o_2$}
  \psset{linecolor=gray}
  \qdisk(0,3){7pt}
  \uput[0](-1.6,3.3){$o_3$}
  \psset{linecolor=black}
  \psline(7,3)(4,8)
  \psline(4.1,7.8)(4.4,7.8)\psline(4.2,7.6)(4.5,7.6)\psline(4.4,7.4)(4.6,7.4)\psline(4.5,7.2)(4.8,7.2)\psline(4.6,7.0)(4.9,7.0)\psline(4.7,6.8)(5,6.8)
  \psline(4.8,6.6)(5.1,6.6)\psline(5,6.4)(5.2,6.4)\psline(5.1,6.2)(5.3,6.2)\psline(5.2,6)(5.4,6)\psline(5.3,5.8)(5.5,5.8)\psline(5.45,5.6)(5.65,5.6)
  \psline(5.55,5.4)(5.8,5.4)\psline(5.7,5.2)(5.9,5.2)\psline(5.8,5)(6,5)\psline(5.9,4.8)(6.1,4.8)\psline(6,4.6)(6.3,4.6)\psline(6.2,4.4)(6.4,4.4)
  \psline(6.3,4.2)(6.5,4.2)\psline(6.4,4)(6.6,4)\psline(6.5,3.8)(6.7,3.8)\psline(6.65,3.6)(6.85,3.6)\psline(6.75,3.4)(6.95,3.4)\psline(6.9,3.2)(7.1,3.2)
  \uput[0]{-59}(5,6.8){Image plane}
  \psline(-0.5,8)(-0.5,1.4)\psline{->}(-0.5,7)(-0.5,6.6)\psline(-0.5,1.4)(-2.3,5)\psline{->}(-1.4,3.2)(-1.6,3.55)
  \psline(-0.3,8)(-0.3,3.4)\psline{->}(-0.3,7)(-0.3,6.6)\psline(-0.3,3.4)(-2.3,4.4)\psline{->}(-1.3,3.9)(-1.5,4.0)
  \psline(0.3,8)(0.3,3.4)\psline{->}(0.3,7)(0.3,6.6)\psline(0.3,3.4)(5.3,5.8)\psline{->}(2.8,4.6)(3,4.7)
  \psline(0.6,8)(0.6,1.4)\psline{->}(0.6,7)(0.6,6.6)\psline(0.6,1.4)(5.5,5.5)\psline{->}(3.05,3.45)(3.2,3.6)
  \psline(1,8)(1,1.2)\psline{->}(1,7)(1,6.6)\psline(1,1.2)(2.2,1.8)\psline{->}(1.6,1.5)(1.8,1.6)
  \psline(2.5,8)(2.5,2.7)\psline{->}(2.5,7)(2.5,6.6)\psline(2.5,2.7)(1.4,3.3)\psline{->}(1.95,3)(1.7,3.15)
  \psline(3.5,8)(3.5,2.7)\psline{->}(3.5,7)(3.5,6.6)\psline(3.5,2.7)(6.3,4.2)\psline{->}(4.9,3.45)(5.1,3.55)
  \psset{dotstyle=*,dotangle=0,dotsize=1pt}
  \psdots(-0.15,7)(0,7)(0.15,7)
  \psdots(1.6,7)(1.75,7)(1.9,7)
  \psdots(2.85,7)(3,7)(3.15,7)
  \uput[0](-1,8.5){Light source}
\end{pspicture}}
\subfigure[reflection imaging with divergent beam]{
\psset{xunit=0.5cm,yunit=0.5cm,linewidth=0.19mm}
\begin{pspicture}(-5,-1.5)(6,10)
  \psset{linecolor=black}
  \qdisk(0,0){22pt}
  \psset{linecolor=white}
  \qdisk(0.5,0.5){6pt}
  \uput[0](-1.7,-1.7){$o_1$}
  \psset{linecolor=black}
  \qdisk(3,2){12pt}
  \uput[0](3.6,2){$o_2$}
  \psset{linecolor=gray}
  \qdisk(0,3){7pt}
  \uput[0](-1.6,3.3){$o_3$}
  \psset{linecolor=black}
  \psline(7,3)(4,8)
  \psline(4.1,7.8)(4.4,7.8)\psline(4.2,7.6)(4.5,7.6)\psline(4.4,7.4)(4.6,7.4)\psline(4.5,7.2)(4.8,7.2)\psline(4.6,7.0)(4.9,7.0)\psline(4.7,6.8)(5,6.8)
  \psline(4.8,6.6)(5.1,6.6)\psline(5,6.4)(5.2,6.4)\psline(5.1,6.2)(5.3,6.2)\psline(5.2,6)(5.4,6)\psline(5.3,5.8)(5.5,5.8)\psline(5.45,5.6)(5.65,5.6)
  \psline(5.55,5.4)(5.8,5.4)\psline(5.7,5.2)(5.9,5.2)\psline(5.8,5)(6,5)\psline(5.9,4.8)(6.1,4.8)\psline(6,4.6)(6.3,4.6)\psline(6.2,4.4)(6.4,4.4)
  \psline(6.3,4.2)(6.5,4.2)\psline(6.4,4)(6.6,4)\psline(6.5,3.8)(6.7,3.8)\psline(6.65,3.6)(6.85,3.6)\psline(6.75,3.4)(6.95,3.4)\psline(6.9,3.2)(7.1,3.2)
  \uput[0]{-59}(5,6.8){Image plane}
  \psline(1.2,8)(1.2,0.98)\psline{->}(1.2,5)(1.2,4.6)\psline(1.2,0.98)(2.5,1.43)\psline{->}(1.9,1.2)(2.125,1.3)
  \psline(1.2,8)(2.6,-0.705)\psline{->}(1.674,5.04)(1.71,4.74)\psline{->}(2.466,0.1)(2.48,0.03)
  \psline(1.2,8)(2.44,2.6)\psline{->}(1.87,5.075)(1.925,4.84)\psline(2.44,2.6)(1.5,3.258)\psline{->}(1.95,2.95)(1.62,3.16)
  \psline(1.2,8)(2.92,2.84)\psline{->}(2.14,5.15)(2.206,4.95)\psline(2.92,2.84)(3.0,6.2)\psline{->}(2.96,4.52)(2.98,4.8)
  \psline(1.2,8)(0.53,1.44)\psline{->}(0.902,5.01)(0.865,4.71)\psline(0.588,1.44)(4.18,7.70)\psline{->}(3.5,6.53)(3.6,6.68)
  \psline(1.2,8)(0.43,3.18)\psline{->}(0.726,5.04)(0.69,4.74)\psline(0.43,3.18)(6.95,3.07)\psline{->}(6.2,3.091)(6.25,3.09)
  \psline(1.2,8)(0.14,3.43)\psline{->}(0.53,5.075)(0.475,4.84)\psline(0.14,3.43)(0.5,4.5)\psline{->}(0.53,5.075)(0.475,4.84)
  \psline(1.2,8)(-0.34,3.32)\psline{->}(0.26,5.15)(0.194,4.95)\psline(-0.38,3.32)(-1.5,2.8)\psline{->}(-1.23,2.93)(-1.3,2.90)
  \psset{dotstyle=*,dotangle=0,dotsize=1pt}
  \psdots(1.3,5.036)(1.425,5.036)(1.55,5.036)
  \psset{dotstyle=*,dotangle=0,dotsize=3pt}
  \psdots(1.2,8)
  \uput[0](-1.0,8.5){Light source}
\end{pspicture}}
\subfigure[transmission imaging with parallel
beam]{\psset{xunit=0.5cm,yunit=0.5cm,linewidth=0.19mm}
\begin{pspicture}(-5,-3)(6,10)
  \psset{linecolor=black}
  \qdisk(0,0){22pt}
  \psset{linecolor=white}
  \qdisk(0.5,0.5){6pt}
  \uput[0](-1.7,-1.7){$o_1$}
  \psset{linecolor=black}
  \qdisk(3,2){12pt}
  \uput[0](3.6,2){$o_2$}
  \psset{linecolor=gray}
  \qdisk(0,3){7pt}
  \uput[0](-1.6,3.3){$o_3$}
  \psset{linecolor=black}
  \psline(-2,-2)(4.4,-2)
  \psline(-1.8,-2)(-1.6,-2.2)\psline(-1.6,-2)(-1.4,-2.2)\psline(-1.4,-2)(-1.2,-2.2)\psline(-1.2,-2)(-1,-2.2)\psline(-1,-2)(-0.8,-2.2)
  \psline(-0.8,-2)(-0.6,-2.2)\psline(-0.6,-2)(-0.4,-2.2)\psline(-0.4,-2)(-0.2,-2.2)\psline(-0.2,-2)(0,-2.2)\psline(0,-2)(0.2,-2.2)
  \psline(0.2,-2)(0.4,-2.2)\psline(0.4,-2)(0.6,-2.2)\psline(0.6,-2)(0.8,-2.2)\psline(0.8,-2)(1,-2.2)\psline(1,-2)(1.2,-2.2)\psline(1.2,-2)(1.4,-2.2)
  \psline(1.4,-2)(1.6,-2.2)\psline(1.6,-2)(1.8,-2.2)\psline(1.8,-2)(2,-2.2)\psline(2,-2)(2.2,-2.2)\psline(2.2,-2)(2.4,-2.2)\psline(2.4,-2)(2.6,-2.2)
  \psline(2.6,-2)(2.8,-2.2)\psline(2.8,-2)(3,-2.2)\psline(3,-2)(3.2,-2.2)\psline(3.2,-2)(3.4,-2.2)\psline(3.4,-2)(3.6,-2.2)\psline(3.6,-2)(3.8,-2.2)
  \psline(3.8,-2)(4,-2.2)\psline(4,-2)(4.2,-2.2)
  \uput[0](-1,-2.8){Image plane}
  \psline(-0.5,8)(-0.5,-2)\psline{->}(-0.5,7)(-0.5,6.6)
  \psline(-0.3,8)(-0.3,-2)\psline{->}(-0.3,7)(-0.3,6.6)
  \psline(0.3,8)(0.3,-2)\psline{->}(0.3,7)(0.3,6.6)
  \psline(0.6,8)(0.6,-2)\psline{->}(0.6,7)(0.6,6.6)
  \psline(1,8)(1,-2)\psline{->}(1,7)(1,6.6)
  \psline(2.5,8)(2.5,-2)\psline{->}(2.5,7)(2.5,6.6)
  \psline(3.5,8)(3.5,-2)\psline{->}(3.5,7)(3.5,6.6)
  \psset{dotstyle=*,dotangle=0,dotsize=1pt}
  \psdots(-0.15,7)(0,7)(0.15,7)
  \psdots(1.6,7)(1.75,7)(1.9,7)
  \psdots(2.85,7)(3,7)(3.15,7)
  \uput[0](-1.0,8.5){Wave source}
\end{pspicture}}
\subfigure[transmission imaging with divergent
beam]{\psset{xunit=0.5cm,yunit=0.5cm,linewidth=0.19mm}
\begin{pspicture}(-5,-3)(6,10)
  \psset{linecolor=black}
  \qdisk(0,0){22pt}
  \psset{linecolor=white}
  \qdisk(0.5,0.5){6pt}
  \uput[0](-1.7,-1.7){$o_1$}
  \psset{linecolor=black}
  \qdisk(3,2){12pt}
  \uput[0](3.6,2){$o_2$}
  \psset{linecolor=gray}
  \qdisk(0,3){7pt}
  \uput[0](-1.6,3.3){$o_3$}
  \psset{linecolor=black}
  \psline(-3,-2)(5.4,-2)
  \psline(-2.8,-2)(-2.6,-2.2)\psline(-2.6,-2)(-2.4,-2.2)\psline(-2.4,-2)(-2.2,-2.2)\psline(-2.2,-2)(-2.0,-2.2)\psline(-2.0,-2)(-1.8,-2.2)
  \psline(-1.8,-2)(-1.6,-2.2)\psline(-1.6,-2)(-1.4,-2.2)\psline(-1.4,-2)(-1.2,-2.2)\psline(-1.2,-2)(-1,-2.2)\psline(-1,-2)(-0.8,-2.2)
  \psline(-0.8,-2)(-0.6,-2.2)\psline(-0.6,-2)(-0.4,-2.2)\psline(-0.4,-2)(-0.2,-2.2)\psline(-0.2,-2)(0,-2.2)\psline(0,-2)(0.2,-2.2)
  \psline(0.2,-2)(0.4,-2.2)\psline(0.4,-2)(0.6,-2.2)\psline(0.6,-2)(0.8,-2.2)\psline(0.8,-2)(1,-2.2)\psline(1,-2)(1.2,-2.2)\psline(1.2,-2)(1.4,-2.2)
  \psline(1.4,-2)(1.6,-2.2)\psline(1.6,-2)(1.8,-2.2)\psline(1.8,-2)(2,-2.2)\psline(2,-2)(2.2,-2.2)\psline(2.2,-2)(2.4,-2.2)\psline(2.4,-2)(2.6,-2.2)
  \psline(2.6,-2)(2.8,-2.2)\psline(2.8,-2)(3,-2.2)\psline(3,-2)(3.2,-2.2)\psline(3.2,-2)(3.4,-2.2)\psline(3.4,-2)(3.6,-2.2)\psline(3.6,-2)(3.8,-2.2)
  \psline(3.8,-2)(4,-2.2)\psline(4,-2)(4.2,-2.2)\psline(4.2,-2)(4.4,-2.2)\psline(4.4,-2)(4.6,-2.2)\psline(4.6,-2)(4.8,-2.2)\psline(4.8,-2)(5.0,-2.2)
  \psline(5.0,-2)(5.2,-2.2)
  \uput[0](-1,-2.8){Image plane}
  \psline(1.2,8)(1.2,-2)\psline{->}(1.2,5)(1.2,4.6)
  \psline(1.2,8)(2.8,-2)\psline{->}(1.674,5.04)(1.71,4.74)
  \psline(1.2,8)(3.5,-2)\psline{->}(1.87,5.075)(1.925,4.84)
  \psline(1.2,8)(4.5,-2)\psline{->}(2.14,5.15)(2.206,4.95)
  \psline(1.2,8)(0.2,-2)\psline{->}(0.902,5.01)(0.865,4.71)
  \psline(1.2,8)(-0.4,-2)\psline{->}(0.726,5.04)(0.69,4.74)
  \psline(1.2,8)(-1.1,-2)\psline{->}(0.53,5.075)(0.475,4.84)
  \psline(1.2,8)(-2.1,-2)\psline{->}(0.26,5.15)(0.194,4.95)
  \psset{dotstyle=*,dotangle=0,dotsize=1pt}
  \psdots(1.3,5.036)(1.425,5.036)(1.55,5.036)
  \psset{dotstyle=*,dotangle=0,dotsize=3pt}
  \psdots(1.2,8)
  \uput[0](-1.0,8.5){Wave source}
\end{pspicture}}
\caption{\label{fig-RefTranImaging} A simple illustration of
reflection and transmission imaging}
\end{center}
\end{figure}

Images usually contain various degradations such as noises due to some reasons such as the non-perfectness of the imaging procedure and network transmission. For instances, images by reflection imaging often contain some Gaussian noises and blur effects, while images by transmission imaging are often contaminated by Poisson noises. As mentioned above, for transmission imaging, we have to reconstruct the objects from the 2D images, and the noise in 2D projection images will affect the accuracy of the reconstruction methods.  

The strategies for noise reduction in transmission imaging can be divided into three groups: pre-reconstruction denoising, post-reconstruction denoising, and regularized iterative reconstruction methods with many forward and backward projection steps. There are many pre-reconstruction denoising methods developed which operate on the raw projection data (transmission imaging) before image reconstruction~\cite{Schaap08,Borsdorf08,Wang2006,Riviere06,Shtok11}.  Post-reconstruction processing includes the methods for improving the image quality without affecting spatial resolution. However the artifacts in the reconstructed images are always recognized as structures in the scanned object and will be enhanced. Regularized iterative reconstruction methods have demonstrated superior performance in undersampled tomography imaging. These methods have demonstrated their tremendous power in image reconstruction with only a few projections~\cite{Brune09,Sidky06,Sidky12,Yan13}. However, these methods are seldom used for commercial purpose. For example, the traditional filtered back-projection is still mainly used for image reconstruction by commercial CT scanners because of  several reasons including the speed and image quality~\cite{Pan09}. Many new iterative algorithms for CT reconstruction introduced by major CT manufacturers are still using pre-reconstruction processing, combing with post-reconstruction processing with only one step of backward projection, e.g. iterative reconstruction in image space (IRIS) and sinogram affirmed iterative reconstruction (SAFIRE) by Siemens Medical Solutions,  adaptive iterative dose reduction (AIDR 3D) by Toshiba Medical Systems, iDose by Philips Healthcare~\cite{Willemink13}. In all these algorithms, the pre-reconstruction processing is very important in noise reduction, and understanding the properties of transmission imaging is helpful to noise reduction in projection data. 

Image processing methods often assume some prior knowledge about the image data. These prior knowledge describes the features of the images without any degradation, e.g., how to model the clean images. As well known, images by reflection imaging can be usually modeled as discontinuous functions and even piecewise constant functions in most cases. See \cite{Wu13} for an analysis from physical principle for the case of parallel light imaging (see Fig.~\ref{fig-RefTranImaging} (a)), which can also be applied to the case of divergent light imaging (see Fig. \ref{fig-RefTranImaging}(b)) and other reflection imaging cases. Consequently most of images by reflection imaging have sparse gradients. This is a very important property, based on which many image processing and segmentation techniques, models, and algorithms have been proposed in the literature, such as the popular total variation (TV) regularization \cite{Rudin92}.

In the following we will focus on transmission imaging to understand the properties of the transmission imaging. There are typically two types of wave beam geometries in transmission imaging due to the application backgrounds and the ability of wave generators. See Fig.~\ref{fig-RefTranImaging}(c)(d). Transmission imaging with parallel beam geometry, such as the cryo-EM technique \cite{Frank96}, is widely applied in biological and medicine sciences to detect molecular structures. Transmission imaging with divergent beam geometry is extensively applied in medical diagnosis, such as the X-ray CT technique \cite{Kak88,Natterer01}.

For the transmission imaging with parallel beam geometry, it has been shown in \cite{Wu13} that almost all images can be modeled as continuous functions. Let us explain this in more details. In transmission imaging with parallel beam geometry, people take images (also called projections in the literature) from many different projection directions in order to reconstruct the density functions of the imaged objects. Each projection direction corresponds to one image. See Fig. \ref{fig-RefTranImaging}(c). For a fixed projection direction, the source radial a parallel wave beam and the wave beam transmits the 3D objects (such as some biological specimen) with a portion of the wave arriving at the image plane. The information recorded in the image plane is then used to infer the structure of the scene. The interaction between the wave and the objects in the scene depends on some certain density function of the objects. The image plane records the line integrals of the density function along the lines of the wave beam. In our case of $\mathbf{R}^3$, the projection directions can be regarded as points on the 2 dimensional unit sphere $S^2$. It has been proved that for almost every projection direction, the generated image is a continuous function, even if the density function of the imaged objects are discontinuous (discontinuous density functions are very common). The set of projection directions generating discontinuous images has measure zero on the unit sphere.

In the transmission imaging with divergent beam geometry as shown in Fig. \ref{fig-RefTranImaging}(d), people take images from many points in $\mathbf{R}^3$. Each point in $\mathbf{R}^3$ outside of the support of the 3D objects (such as human body to be scanned) corresponds an image. For a fixed point, the source radial a divergent wave beam and the wave beam transmits the 3D objects with a portion of the wave arriving at the image plane. Again, the information recorded in the image plane is then used to infer the structure of the scene. The interaction between the wave and the objects in the scene depends on some certain kind of density function of the objects. The image plane records the line integrals of the density function along the lines of the wave beam. Totally we can take as many images as points in $\mathbf{R}^3$ outside of the support of the 3D objects. One of the purpose of this paper is to show that almost all images are continuous functions, even if the density function of the imaged objects are discontinuous (discontinuous density functions are very common). The set of points generating discontinuous images has measure zero in $\mathbf{R}^3$.

An essential mathematical tool to describe transmission imaging is Radon transform~\cite{Radon86}. As far as we know, theoretical results on Radon Transform in the literature focus on the analysis of the imaging procedure as a mapping operator~\cite{Helgason99,Kak88,Natterer01}, e.g., the invertability of the operator. In addition, most analysis assume that the density function of the object to be imaged is a continuous or even Schwartz function all over the Euclidean space~\cite{Helgason99,Natterer01,Radon86}. Our analysis in this paper are quite different from those in the literature in two ways. Firstly, we consider the features of images (2D projections) instead of the imaging procedure. Although the central topic in transmission imaging is the reconstruction of the 3D objects from their 2D projection images, restoration of these image data (before 3D reconstruction) is also important in improving the accuracy of the reconstruction methods due to the involvement of noise and other degradations during the imaging procedure. In addition, it is easier to model the noise in projection images before reconstruction. A typical problem is how to remove the Poisson noise (which is the dominant noise in transmission imaging); see~\cite{Brune09,Chan07,Le07,Figueiredo09,Panin99,Setzer09b,Wu11,Zanella09}, etc. Studying the image features helps to choose and design better image processing models and algorithms, as mentioned before. Secondly, our analysis assume discontinuous density functions of objects, which are very common in practical applications.

In particular, we contribute the following in this paper. Firstly, we prove the continuity property of images by transmission imaging with parallel beam geometry under much weaker assumptions than those in \cite{Wu13}. The analysis in \cite{Wu13} excludes the case of cylindrical 3D objects. Our analysis here includes almost all types of 3D objects. Secondly, we prove similar continuity property of images by transmission imaging with divergent beam geometry, thus completing this kind of analysis for two typical beam geometries in transmission imaging. Thirdly, we compare two current popular image regularization techniques for images by transmission imaging, verifying our analysis and providing some information for choosing wavelet frame methods instead of the TV method for processing this kind of images.

The paper is organized as follows. In section 2, we will present our main result, i.e., the continuity analysis of images by transmission imaging. Section 3 provides some numerical results to verify our theoretical analysis. Section 4 concludes the paper.


\section{Theoretical Analysis}

\subsection{Notations and the main theorems}\label{Notation}
Let $p$ be a point in $\mathbb{R}^3$, and $W$ be a vector in $\mathbb{R}^3$. We denote by $L_{_{(p,W)}}$ the half line in the direction $W$ starting from $p$, more precisely,
\begin{align*}
L_{_{(p,W)}}= \{p+tW|~ t\geq 0 \}.
\end{align*}
These straight lines will be the light beams we are going to study. 

Let $\Sigma$ be a bounded smooth surface in $\mathbb{R}^3$, especially $\Sigma\subset D$, where $D$ is a bounded convex domain in $\mathbb{R}^3$. Let $B$ be the boundary of $D$, in this paper we also assume $B$ to be a smooth surface. We are going to study the three type of beams in this paper,
\begin{enumerate}
\item The diverse beams with source located on $B$;
\item The diverse beams with source located in $\mathbb{R}^3$;
\item The parallel beams with directions on the unit sphere $S^2$.
\end{enumerate}

When we are using these beams to scan the surface $\Sigma$, it is possible to have ``singular sources". In fact if for a source point $q$, there exists a direction $W$, such that the intersubsection of the half line $L_{_{(q, W)}}$ and the surface $\Sigma$ includes a interval on $L_{_{(q, W)}}$, then we call this point $q$ to be a singular source. More precisely we give the following definition. 
\begin{definition}
Let $q$ be a point in $\mathbb{R}^3$ and $W$ be a direction in $\mathbb{R}^3$. If there is point $p$ in $\Sigma$ such that $p\in L_{_{(q, W)}}$ and $p+tW\in \Sigma$ if $t\in (0,\epsilon)$ for some positive $\epsilon$. Then $q$ is a singular source, $W$ is as singular direction, $L_{_{(q, W)}}$ is a singular beam and $p$ is a singular point of the beam $L_{_{(q, W)}}$. We also say the beam $L_{_{(q, W)}}$ is singular at point $p$.
\end{definition}
Especially, we specify three singular sets according to the three types of beams.
\begin{enumerate}
\item All the singular sources on $B$, 

$\mathcal{Z}=\{q\in B:\, \exists W\in \mathbb{R}^3$ such that $L_{_{(q, W)}}$ is a singular beam $\}$. 

This is the singular sources set for the diverse beams with source located on $B$.
\item All the singular sources in $\mathbb{R}^{3}$, 

$\tilde{\mathcal{Z}}=\{q\in \mathbb{R}^3:\, \exists W\in \mathbb{R}^3$ such that $L_{_{(q, W)}}$ is a singular beam $\}$. 

This is the singular sources set for the diverse beams with source located in $\mathbb{R}^3$. Especially $\mathcal{Z}=\tilde{\mathcal{Z}}\cap B$.
\item All the singular directions, 

$\mathcal{D}=\{W\in S^2:\, \exists q\in\mathbb{R}^3$ such that $L_{_{(q, W)}}$ is a singular beam $\}$. 

This is the singular directions set for the parallel beams with directions on $S^2$.
\end{enumerate}

These singular sources sets and singular directions set are determined by the surface $\Sigma$. We define $\mathcal{X}$ to be all the singular points on $\Sigma$. Precisely speaking, that is
\begin{align*}
\mathcal{X}=\{p\in \Sigma|~ \exists W\in T_p\Sigma ~\text{such}~\text{that}~ L_{_{(p, W)}}~\text{is}~\text{singular}~\text{at}~p\}.
\end{align*} 
One may notice that a singular point is also a singular source. The property of $\mathcal{X}$ is very important in the study of $\mathcal{Z}$, $\tilde{\mathcal{Z}}$ and $\mathcal{D}$. 

\begin{lemma}\label{singular points}
For any point $q\in \tilde{\mathcal{Z}}$, there is $p\in\mathcal{X}$, $W\in T_p\Sigma$, such that $L_{_{(p, W)}}$ is singular at point $p$ and $q\in L_{_{(p, W)}}$. And For any unit direction $W\in \mathcal{D}$, there is a point $p\in \mathcal{X}$, such that $L_{_{(p, W)}}$ is singular at point $p$.
\end{lemma}
\begin{proof}
The proof is very strait forward. For the first part, if $q\in \tilde{\mathcal{Z}}$ and $L_{_{(q, V)}}$ is singular at point $p_1\in \Sigma$, then there exists $\epsilon>0$ such that $p_1+tV\in\Sigma$ for any $t\in (0, \epsilon)$. Let $p=p_1+\epsilon/2 V$ and $W=-V$, then $L_{_{(p, W)}}$ is singular at point $p$ and $q\in L_{_{(p, W)}}$. For the second part, if $W\in \mathcal{D}$, then there is a $p\in\Sigma$, such that $L_{_{(p, W)}}$ is singular at $p$. Therefore $p\in \mathcal{X}$.
\end{proof}

We note that the singular points set $\mathcal{X}$ can be open on $\Sigma$. For example if $\Sigma$ is a ruled surface, then every point on $\Sigma$ is a singular point, i.e. $\mathcal{X}=\Sigma$. But the singular sources sets $\mathcal{Z}$, $\tilde{\mathcal{Z}}$ and the singular directions set $\mathcal{D}$ are more restrictive. More precisely, we state the main theorem of this paper.
 
\begin{theorem}\label{main}
For each type of beams, the singular sources set or singular directions set have measure zero as a subset of the corresponding ambient set. More precisely
\begin{enumerate}
\item $\mathcal{Z}$ has measure zero in $B$;
\item $\tilde{\mathcal{Z}}$ has measure zero in $\mathbb{R}^3$;
\item $\mathcal{D}$ has measure zero in $S^2$.
\end{enumerate}
\end{theorem}

The proof of our main theorem in subsection \ref{Proof} will be based on the Sard's theorem on page 16 of \cite{milnor},
\begin{theorem}[Sard's theorem]\label{Sard}
Let $f$ be a smooth map between smooth manifolds $M$ and $N$, and $X$ be the set of critical points of $f$. Then the image of $X$ under $f$ is a set of measure zero in $N$.
\end{theorem}

To explain the critical point of a smooth map between two manifolds. We first recall that the tangent bundle $TM$ of a manifold $M$ consists the pairs $(p, W)$, where $p$ is in $M$ and $W$ is a tangent vector of $M$ at point $p$. Then $TM$ is a $2m$ dimensional manifold, where $m=\dim M$. For example if $M=\Sigma$, and $(U,\{u, v\})$ is a local coordinate chart of $\Sigma$ then the local coordinate chart for the tangent bundle $TU$ is
\begin{align*}
(u,v,a,b)\mapsto \left(p(u,v), a\frac{\partial}{\partial u}+b\frac{\partial}{\partial v}\right)
\end{align*}
where $p(u,v)$ is the point of coordinate $(u,v)$, and $a\frac{\partial}{\partial u}+b\frac{\partial}{\partial v}$ is a tangent vector in $T_p\Sigma$. In this paper, because all the surfaces are embedded in $\mathbb{R}^3$, we can also consider the tangent vectors as directions in $\mathbb{R}^3$.\\ 

A point $p\in M$ is a critical point of map $f:\,M\to N$ if the tangent map at point $p$, $f_*|_p: T_{p}M \to T_{f(p)}N$ is not onto. The image of the critical points are called critical values. In the next subsection, we are going to construct several smooth maps, such that the singular sources and singular directions are critical values of the corresponding maps.

\subsection{Construction of smooth maps for Sard's Theorem}\label{construction}

In this subsection, we characterize the singular points set $\mathcal{X}$ using the second fundamental form. Then we construct several smooth maps, of which the singular sources and singular directions are critical values. We start with a simple observation,
\begin{lemma}\label{Second form}
For any point $p\in \mathcal{X}$, there is at least one vector $W\in T_p\Sigma$ such that the directional curvature of $\Sigma$ at point $p$ on direction $W$ is zero.
\end{lemma}
\begin{proof}
For any singular point $p\in \mathcal{X}$, there exists $\epsilon>0$ and a direction $W$ such that $p+tW\in \Sigma$ for $t\in (0,\epsilon)$. Then we have that the directional curvature of $\Sigma$ at $p$ on direction $W$ is the same as the curvature of the straight line, which is zero.
\end{proof}

The second fundamental form II is a symmetric quadratic form on the surface $\Sigma$. If $(u, v)$ is a local coordinate of $\Sigma$, then we can express the second fundamental form as
\begin{align*}
\text{II}=Ldu^2+2Mdudv+Ndv^2,
\end{align*}
or, equivalently, in a matrix form
\begin{align*}
\text{II}=\left[\begin{array}[c]{cc}
L & M\\
M & N
\end{array}
\right].
\end{align*}
Also for each point $p\in \Sigma$ and direction $W\in T_p\Sigma$, we have that the directional curvature of $\Sigma$ at $p$ on $W$ is 
$\text{II}(W, W)/ \langle W, W\rangle$. In particular the two real eigenvalues $k_1$ and $k_2$ of $\text{II}$ are the principle curvatures, $K=k_1k_2$ is the Gauss curvature and $H=k_1+k_2$ is the mean curvature. Gauss curvature and mean curvature are smooth functions on the surface $\Sigma$.

According to the second fundamental form, we divide the surface $\Sigma$ in to $\Sigma=\Sigma_+\cup \Sigma_0\cup \Sigma_1 \cup \Sigma_-$, in which
\begin{enumerate}
\item The points with positive Guass curvature $\Sigma_+=\{p\in \Sigma|~ \det\text{II}(p)>0\}$. For each point $p\in \Sigma_+$ and any $W\neq 0\in T_p\Sigma$, the directional curvature $\text{II}(W,W)/\langle W, W\rangle  \neq 0$;
\item The points with zero Guass curvature and zero mean curvature, $\Sigma_0=\{p\in \Sigma|~ \text{II}(p)=0\}$. For each point $p\in \Sigma_0$ and any vector $W\neq 0\in T_p\Sigma$, the directional curvature $\text{II}(W,W)/\langle W, W\rangle = 0$;
\item The point with zero Guass curvature and non-zero mean curvature $\Sigma_1=\{p\in \Sigma|~ \text{II}(p)\ \text{has}\ \text{rank}\ 1\}$. For each point $p\in \Sigma_1$ there is a unique unit vector $W\in T_p\Sigma$, up to sign, such that the directional curvature $\text{II}(W,W)=\text{II}(W,W)/\langle W, W\rangle  = 0$;
\item The points with negative Guass curvature, $\Sigma_-=\{p\in \Sigma|~ \det\text{II}(p)<0\}$. For each point $p\in \Sigma_1$ there are two linearly independent unit vectors $W_1, W_2\in T_p\Sigma$, such that the directional curvature $\text{II}(W_1,W_1)=\text{II}(W_1,W_1)/\langle W_1, W_1\rangle=\text{II}(W_2,W_2)/\langle W_2, W_2\rangle=\text{II}(W_2,W_2)= 0$.
\end{enumerate}

Lemma \ref{Second form} implies that the singular points set has no intersubsection with $\Sigma_+$, on each point of which the direction curvature is always positive. So we can divide $\mathcal{X}$ to a union of three parts $\mathcal{X}=(\mathcal{X}\cap\Sigma_0)\cup(\mathcal{X}\cap(\Sigma_1\cup\Sigma_-))$. Now we are ready to define two different types of smooth maps according to these two parts of $\mathcal{X}$.

\subsubsection{Smooth maps for zero principle curvatures}
Let $T\Sigma$ be the tangent bundle of $\Sigma$. We recall that $T\Sigma$ is a $4$-dimensional manifold. We define the following maps
\begin{enumerate}
\item $g:\,T\Sigma\to B$ such that $g(p,W)=L_{_{(p,W)}}\cap B$;
\item $\tilde{g}:\,T\Sigma\to \mathbb{R}^3$ such that $g(p,W)=p+W$;
\item $g^D:\,T\Sigma\to S^2$ such that $g(p, W)= W/|W|$.
\end{enumerate}

Then we have
\begin{proposition}\label{g}
The maps $g$, $\tilde{g}$ and $g^D$ are smooth maps. For any $p\in \Sigma_0\cap\mathcal{X}$ and $W\in T_p\Sigma$, the point $(p, W)\in T\Sigma$ is a critical point of all the three maps $g$, $\tilde{g}$ and $g^D$.
\end{proposition}
\begin{proof}
We fix one pint $p_0\in \Sigma_0\cap\mathcal{X}$. For the computation of the tangent map, let $\{U, (u, v)\}$ be a local coordinate chart around $p_0$, and $(u,v,a,b)$ gives a local coordinate system of $TU$ by
\begin{align*}
(u,v,a,b)\mapsto \left(p(u,v), a\frac{\partial p(u,v)}{\partial u}+b\frac{\partial p(u,v)}{\partial v}\right),
\end{align*}
in which $p(u,v)$ is a local parametrizaion of $\Sigma$ with $p_0=p(0,0)$ and $a\frac{\partial p}{\partial u}+b\frac{\partial p}{\partial v}$ is a vector in $T_p\Sigma$. 

We first look at $g:\, T\Sigma\to B$. According to the local coordinate on $TU$, we have
\begin{align*}
g(u,v,a,b)=p(u,v)+r(u,v,a,b)\cdot\left(a\frac{\partial p(u,v)}{\partial u}+b\frac{\partial p(u,v)}{\partial v}\right)
\end{align*}
for a function $r$. In order to prove the smoothness of $g$, we only need to show that $r$ is a smooth function. Actually the smooth surface $B$ is locally defined as zero set of a smooth function $F(x,y,z)=0$. So we have the function $r$ is actually the solution of the following equation,
$$F\left(p(u,v)+r\cdot(a\frac{\partial p(u,v)}{\partial u}+b\frac{\partial p(u,v)}{\partial v})\right)=0.$$
From our assumption that the domain $D$ is convex, we know that the vector $a\frac{\partial p(u,v)}{\partial u}+b\frac{\partial p(u,v)}{\partial v}$ is not tangent to $B$ at point $g(u,v,a,b)$, therefore the partial derivative $\frac{\partial F}{\partial r}\neq 0$. Then using the inverse function theorem, we have that the solution $r(u,v,a,b)$ is a smooth function on $T\Sigma$. 

We are ready to compute the tangent map of $g$ at $(u,v,a,b)$, in fact
\begin{align*}
g_u&=p_u+r_u(ap_u+bp_v)+r(ap_{uu}+bp_{uv});\\
g_v&=p_v+r_v(ap_u+bp_v)+r(ap_{vu}+bp_{vv});\\
g_a&=r_a(ap_u+bp_v)+rp_{u};\\
g_b&=r_b(ap_u+bp_v)+rp_{v};\\
\end{align*}
Let $n_p$ be the normal vector of $\Sigma$ at point $p(u,v)$, then 
\begin{align*}
\langle n_p,g_u\rangle |_{p_0}&=r(a\langle n_p,p_{uu}\rangle+b\langle n_p,p_{uv}\rangle)|_{p_0}=r(aL+bM)|_{p_0}=0;\\
\langle n_p,g_v\rangle |_{p_0}&=r(a\langle n_p,p_{vu}\rangle+b\langle n_p,p_{vv}\rangle)|_{p_0}=r(aM+bN)|_{p_0}=0;\\
\langle n_p,g_a\rangle |_{p_0}&=0;\\
\langle n_p,g_b\rangle |_{p_0}&=0.
\end{align*}

This implies that the image of tangent map $g_*$ at point $(p_0, W)$ is perpendicular to $n_{p_0}$. We claim that $n_{p_0}$ is not parallel to the normal vector $n_{g(p_0,W)}$ of $B$ at point $g(p_0,W)$. Then there exists a vector $V$ in $T_{g(p_0,W)}B$ such that $V$ is not perpendicular to $n_{p_0}$, therefore the tangent map $g_*$ at the point $(p_0, W)$ is not onto. We prove the claim by contradiction, suppose $n_{p_0}$ is parallel to $n_{g(p_{0},W)}$, then the half line $L_{_{(p_0,W)}}$ will be tangent to the surface $B$ at point $g(p_0, W)$, but this implies that $L_{_{(p_0,W)}}$ stays outside of the convex domain $D$, which contradicts with the fact that $p_0\in \Sigma\subset D$. We finished the proof of that $\Sigma_0\cap\mathcal{X}$ are critical points of the map $g$.

Secondly, we consider the map $\tilde{g}:\,T\Sigma\to \mathbb{R}^3$. Using the above notion of the local coordinate system of $T\Sigma$, we have 
\begin{align*}
\tilde{g}(u,v,a,b)=p(u,v)+a\frac{\partial p}{\partial u}+b\frac{\partial p}{\partial v},
\end{align*}
and the tangent maps are
\begin{align*}
\tilde{g}_u&=p_u+ap_{uu}+bp_{uv};\\
\tilde{g}_v&=p_v+ap_{vu}+bp_{vv};\\
\tilde{g}_a&=p_{u};\\
\tilde{g}_b&=p_{v}.
\end{align*}
Then again we get 
\begin{align*}
\langle n_p,\tilde{g}_u\rangle|_{p_0}=\langle n_p,\tilde{g}_v\rangle|_{p_0}=\langle n_p,\tilde{g}_a\rangle|_{p_0}=\langle n_p,\tilde{g}_b\rangle|_{p_0}=0.
\end{align*}
This directly implies that the tangent map $\tilde{g}_*$ to  $T\mathbb{R}^3$ is not onto.

At last, let us look at the map $g^D:\, T\Sigma\to S^2$. Using the above notion of the local coordinate system of $T\Sigma$, we have 
\begin{align*}
g^D(u,v,a,b)=W/|W|,
\end{align*}
where $W=a\frac{\partial p}{\partial u}+b\frac{\partial p}{\partial v}$. Then we can compute for the tangent map
\begin{align*}
{g^D}_u&=\frac{|W|\cdot W_u-|W|_u\cdot W}{|W|^2};\\
{g^D}_v&=\frac{|W|\cdot W_v-|W|_v\cdot W}{|W|^2};\\
{g^D}_a&=\frac{|W|\cdot W_a-|W|_a\cdot W}{|W|^2};\\
{g^D}_b&=\frac{|W|\cdot W_b-|W|_b\cdot W}{|W|^2}.
\end{align*}
Then again the vanishing second fundamental form implies that
\begin{align*}
\langle n_p,{g^D}_u\rangle |_{p_0}=\langle n_p,{g^D}_v\rangle |_{p_0}=\langle n_p,{g^D}_a\rangle |_{p_0}=\langle n_p,{g^D}_b\rangle |_{p_0}=0.
\end{align*}
And it is clear this time that $n_p$ is not parallel to $W$, therefore there is a vector $V\in T_{W/|W|}S^2$ such that, $\langle V, n_p\rangle\neq 0$. This implies that the tangent map $(g^D)_*$ at point (p,W) is not onto.
\end{proof}

\subsubsection{Smooth maps for different principle curvatures}
Let us consider the part of $\Sigma$ with different principle curvatures, precisely speaking
\begin{align*}
\Sigma_n=\{ p\in \Sigma:\,H_p^2-4K_p\neq 0\}.
\end{align*} 
In particular $\Sigma_-\cup \Sigma_1\subset \Sigma_n$. Because the Gauss curvature $K$ and the mean curvature $H$ are both smooth functions on $\Sigma$, we have that $\Sigma_n$ is an open subset of $\Sigma$. Therefore we have a countable open cover $\Sigma_n=\cup_{i\in \mathbb{N}} U_i$ such that on each $U_i$, there are two unit vector fields $V_1$ and $V_2$ corresponding to $k_1$ and $k_2$, i.e. $|V_1|=|V_2|=1$, $\text{II}(V_1, V_1)=k_1$, $\text{II}(V_2, V_2)=k_2$ and $\text{II}(V_1, V_2)=0$. Then we define on $U_i$ that $W_{\alpha\beta}=(-1)^\alpha\sqrt{|k_2|}V_1+(-1)^\beta\sqrt{|k_1|}V_2$ for $\alpha, \beta\in \{0,1\}$. 

\begin{lemma}
For $p\in (\Sigma_-\cup \Sigma_1)\cap U_i$, directional curvature of $\Sigma$ at point $p$ on direction $W_{\alpha\beta}$ equals zero, for each $\alpha, \beta\in \{0, 1\}$.
\end{lemma}
\begin{proof}
If $p\in \Sigma_-\cap U_i$, then $K=k_1k_2<0$. We have 
$$\text{II}(W_{\alpha\beta}, W_{\alpha\beta})=|k_2|\text{II}(V_1, V_1)+|k_1|\text{II}(V_2, V_2)=|k_1|k_2+k_1|k_2|=0.$$ 
So the directional curvature on $W_{\alpha\beta}$ is $\text{II}(W_{\alpha\beta}, W_{\alpha\beta})/|W_{\alpha\beta}|^2=0$. In this case $W_{00}=-W_{11}$ and $W_{01}=-W_{10}$ but $W_{00}$ and $W_{01}$ are linearly independent.

If $p\in \Sigma_1\cap U_i$, then $K=k_1k_2=0$. With out loss of generality we assume $k_1\neq 0$ and $k_2=0$, then $W_{\alpha 0}=\sqrt{|k_1|}V_2$ and $W_{\alpha 1}=-\sqrt{|k_1|}V_2$ for each $\alpha\in\{0, 1\}$. Then clearly the directional curvature on $W_{\alpha\beta}$ equals zero. In this case $W_{\alpha 0}=-W_{\alpha 1}$ for $\alpha= 1,2$.
\end{proof}

Now we are ready to construct the following maps 
\begin{enumerate}
\item On $U_i$, let $f_{\alpha\beta}:\, U_i\to B$ such that $f_{\alpha\beta}(p)=L_{_{(p,W_{\alpha\beta})}}\cap B$.
\item On $U_i\times \mathbb{R}$, let $\tilde{f}_{\alpha\beta}:\, U_i\times \mathbb{R}\to \mathbb{R}^3$ such that $\tilde{f}_{\alpha\beta}(p, t)=p+t W_{\alpha\beta}$.
\item On $U_i$, let $f^D_{\alpha\beta}:\, U_i\to S^2$ such that $f^D_{\alpha\beta}(p)=W_{\alpha\beta}/|W_{\alpha\beta}|$.
\end{enumerate}

We note that these maps are not globally defined on $\Sigma_n$, since the principle curvatures $k_1$ and $k_2$ are not globally defined functions on $\Sigma$ and the choice of $V_1$ and $V_2$ are not unique either. Topologically speaking the collection of the pairs $(p, W_{\alpha\beta})$ gives a $4$-sheets cover space $\tilde{\Sigma_n}$ of the manifold $\Sigma_n$. Then the $4$ locally defined maps $f_{\alpha\beta}$ can be realized as one globally defined map on $\tilde{\Sigma_n}$. For the reader's convenience we decide to avoid the using of too much algebraic topology, and choose the local definition as above. 

For these maps we have the following proposition
\begin{proposition}\label{f}
If $p\in \mathcal{X}\cap U_i$, then there exist $\alpha$ and $\beta$ in $\{0, 1\}$, such that 
\begin{enumerate}
\item The point $f_{\alpha\beta}(p)$ is a critical value of $f_{\alpha\beta}$;
\item The point $\tilde{f}_{\alpha\beta}(p, t)$ is a critical value of $\tilde{f}_{\alpha\beta}$;
\item The point $f_{\alpha\beta}^D(p)$ is a critical value of $f^D_{\alpha\beta}$.
\end{enumerate}
\end{proposition}
\begin{proof}
Because $p$ is a singular point, there exists a direction $W\in T_p\Sigma$ and $\epsilon>0$, such that $p+sW\in\Sigma$ for $s\in (0,\epsilon)$. Therefore the directional curvature of $\Sigma$ at point $p$ on the direction $W$ is zero, then there exist $\alpha, \beta \in \{0, 1\}$ such that $W_{\alpha\beta}/|W_{\alpha\beta}|= W/|W|$. Without loss of generality we assume that $W_{00}/|W_{00}|= W/|W|$. 

On the other hand, because $\Sigma_0\cap \Sigma_n=\emptyset$ and $\Sigma_+\cap\mathcal{X}=\emptyset$, the singular point $p\in \mathcal{X}\cap(\Sigma_1\cup \Sigma_-)$. We divide the problem into two cases, 
\begin{enumerate}
\item[Case 1] If $p\in \Sigma_-$, then $W_{00}=-W_{11}$ is not parallel to $W_{10}$ and $W_{01}$, which implies that $W_{\alpha\beta}/|W_{\alpha\beta}|\neq W/|W|$ in a neighbourhood of point $p$ for $(\alpha, \beta)\neq (0, 0)$. But on any point $p+sW$ with $s\in (0, \epsilon)$, we have that the directional curvature at point $p+sW$ on direction $W$ is still zero, i.e. $\text{II}_{p+sW}(W,W)=0$. Therefore $W_{00}/|W_{00}|=W/|W|$ at point $p'=p+sW$ for sufficient small $s>0$. Then we get that
\begin{align*}
f_{00}(p+sW)&=L_{_{(p+sW, W_{00})}}\cap B=L_{_{(p+sW, W)}}\cap B\\
&=L_{_{(p, W)}}\cap B=L_{_{(p, W_{00})}}\cap B= f_{00}(p),
\end{align*}
for sufficiently small $s>0$. Then we get that $\frac{\partial}{\partial s}f_{00}(p+sW)=0$, which means that the tangent map $(f_{00})_*(p)=0$. But $\dim \Sigma=\dim B=2$, so $(f_{00})_*$ is not onto at point $p$, which means $p$ is a critical point of $f_{00}$ and $f_{00}(p)$ is a critical value. 

Similarly we have that 
\begin{align*}
f^D_{00}(p+sW)=W_{00}/|W_{00}|=f^D_{00}(p),
\end{align*}
for sufficiently small $s$. Again we get $(f_{00}^D)_*(p)=0$, therefore $p$ is a critical point of $f_{00}^D$ and $f_{00}^D(p)$ is a critical value of $f_{00}^D$. 

For the map $\tilde{f}_{00}$, we have
\begin{align*}
\tilde{f}_{00}(p+sW, t)=p+sW+tW_{00}=p+sW+t\lambda(s)W,
\end{align*}
where $\lambda(s)$ is a positive function of $s$. The derivatives are $\frac{\partial}{\partial s}\tilde{f}_{00}(p+sW, t)=(1+t\lambda'(s))W$ and $\frac{\partial}{\partial t}\tilde{f}_{00}(p+sW, t)=\lambda(s)W$. Therefore we get, at point $(p, t)$ where $s=0$, the tangent map,
\begin{align*}
&(\tilde{f}_{{00}})_* (\lambda(0)W-(1+t\lambda'(0))\frac{\partial}{\partial t})\\
=&\lambda(0)(1+t\lambda'(0))W-(1+t\lambda'(0))\lambda(0)W=0.
\end{align*}
But clearly $\lambda(0)W\neq 0\in T_p\Sigma$ and $(1+t\lambda'(0))\frac{\partial}{\partial t}\in T\mathbb{R}$ are not equal to each other, so we get $(\tilde{f}_{{00}})_*$ at point $(p, t)$ is not injective. Again the fact $\dim \Sigma\times \mathbb{R}=\dim \mathbb{R}^3=3$ implies that $(p, t)$ is a critical pint of $\tilde{f}_{00}$ and $\tilde{f}_{00}(p, t)$ is a critical value for any $t\in \mathbb{R}$.

\item[Case 2] If $p\in \Sigma_1$, then $p+sW\in\Sigma$ and $p+sW$ is also a singular point for each $s\in (0,\epsilon)$. If there is a $s\in (0, \epsilon)$ such that $p'=p+sW\in \Sigma_-$, then we get back to the previous case. If $p+sW\in \Sigma_1$ for all $s\in (0, \epsilon)$, then we can assume that $W_{00}=W_{10}$ are in the same direction of $W$. Then we get back to the same computation as the previous case.
\end{enumerate}
The proof is completed.
\end{proof} 

\subsection{Proof of the main theorems}\label{Proof}
In this subsection we give the proof of the main theorem. 

\begin{theorem}\label{mainproof}
For each type of beams, the singular sources set or singular directions set has measure zero as a subset of the corresponding ambient set. More precisely
\begin{enumerate}
\item $\mathcal{Z}$ has measure zero in $B$;
\item $\tilde{\mathcal{Z}}$ has measure zero in $\mathbb{R}^3$;
\item $\mathcal{D}$ has measure zero in $S^2$.
\end{enumerate}
\end{theorem}
\begin{proof}
For each point $q\in \mathcal{Z}$, there is a point $p\in \mathcal{X}$ and direction $W\in T_p\Sigma$ such that $L_{_{(p, W)}}$ is singular at $p$ and $q=L_{_{(p, W)}}\cap B$.
\begin{enumerate}
\item If $p\in \Sigma_0\cap\mathcal{X}$, then Proposition \ref{g} implies that $(p, W)$ is a critical point of $g$. Therefore $q$ is a critical value of $g$.
\item If $p\in (\Sigma_1\cup\Sigma_-)\cap \mathcal{X}$, then Proposition \ref{f} implies that there is a $i\in \mathbb{N}$ and $\alpha, \beta\in \{0, 1\}$ such that $p\in U_i$ and $q$ is a critical value of $f_{\alpha\beta}:\, U_i\to B$.
\end{enumerate}
Then we have that $\mathcal{Z}$ is covered by the union of critical values of the countable many maps $g:\, \Sigma\to B$ and $f_{\alpha\beta}:\, U_i\to B$ for each $i\in \mathbb{N}$ and $\alpha, \beta\in \{0, 1\}$. Then using Sard's theorem, we get $\mathcal{Z}$ is covered by a countable union of zero measure sets, which is also a zero measure set.\\

For each point $q\in \tilde{\mathcal{Z}}$, there is a point $p\in \mathcal{X}$, $t\in \mathbb{R}$ and direction $W\in T_p\Sigma$ such that $L_{_{(p, W)}}$ is singular at $p$ and $q=p+W$.
\begin{enumerate}
\item If $p\in \Sigma_0\cap\mathcal{X}$, then Proposition \ref{g} implies that $(p, W)$ is a critical point of $g$. Therefore $q$ is a critical value of $\tilde{g}$.
\item If $p\in (\Sigma_1\cup\Sigma_-)\cap \mathcal{X}$, then Proposition \ref{f} implies that there is a $i\in \mathbb{N}$ and $\alpha, \beta\in \{0, 1\}$ such that $p\in U_i$ and $q$ is a critical value of $\tilde{f}_{\alpha\beta}:\, U_i\times \mathbb{R}\to \mathbb{R}^3$.
\end{enumerate}
Then we have that $\tilde{\mathcal{Z}}$ is covered by the union of critical values of the countable many maps $\tilde{g}:\, \Sigma\to \mathbb{R}^3$ and $\tilde{f}_{\alpha\beta}:\, U_i\times \mathbb{R}\to \mathbb{R}^3$ for each $i\in \mathbb{N}$ and $\alpha, \beta\in \{0, 1\}$. Then using Sard's theorem, we get $\tilde{\mathcal{Z}}$ is covered by a countable union of zero measure sets, which is also a zero measure set.\\

For each point $W\in \mathcal{D}$, there is a point $p\in \mathcal{X}$ such that $L_{_{(p, W)}}$ is singular at $p$.
\begin{enumerate}
\item If $p\in \Sigma_0\cap\mathcal{X}$, then Proposition \ref{g} implies that $(p, W)$ is a critical point of $g^D$. Therefore $W$ is a critical value of $g^D$.
\item If $p\in (\Sigma_1\cup\Sigma_-)\cap \mathcal{X}$, then Proposition \ref{f} implies that there is a $i\in \mathbb{N}$ and $\alpha, \beta\in \{0, 1\}$ such that $p\in U_i$ and $W$ is a critical value of $f^D_{\alpha\beta}:\, U_i\to S^2$.
\end{enumerate}
Then we have that $\mathcal{D}$ is covered by the union of critical values of the countable many maps $g^D:\, \Sigma\to B$ and $f^D_{\alpha\beta}:\, U_i\to B$ for each $i\in \mathbb{N}$ and $\alpha, \beta\in \{0, 1\}$. Then using Sard's theorem, we get $\mathcal{D}$ is covered by a countable union of zero measure sets, which is also a zero measure set.
\end{proof}

\subsection{Continuity of the image}

In this subsection we give the mathematical definition of image function, and prove that the image functions are continuous almost surely. Let $\Omega$ be an open domain in $\mathbb{R}^3$ while $\Sigma$ be the boundary of $\Omega$ and $\rho:\,\Omega\cup \Sigma\to \mathbb{R}_+$ be a continuous density function on the closeusre $\bar{\Omega}=\Omega\cup \Sigma$ . We further assume that for each straight line $L$ in $\mathbb{R}^3$, the intersubsection $L\cap \Sigma$ is a countable union of points and intervals on $L$.

\begin{definition}
For a parallel beam with direction $W$, let $H$ be the plane perpendicular to $W$ which passes through origin. Then the image function $\mathcal{I}_W:\, H\rightarrow \mathbb{R}$ is defined as 
\begin{align}\label{paraimage}
\mathcal{I}_W(q)=m_\rho(L_{(q,W)}\cap \Omega)
\end{align}
for each point $q\in H$, where $m(\cdot)$ is the Lebesgue integral of $\rho$ along $L_{(q,W)}\cap \Omega$, i.e.
\begin{align}
m_\rho(L_{(q,W)}\cap \Omega)=\int\limits_{L_{(q,W)}\cap \Omega} \rho\, d\mu.
\end{align}
\end{definition}

For a diverse beam, we will define two different image functions, one is the spherical image function, and the other is the image function on flat image plane. 
 
\begin{definition}
For a diverse beam starting from source point $q$, the spherical image function $\mathcal{SI}_q:\, S^2\rightarrow \mathbb{R}$ is defined as 
\begin{align}\label{paraimage2}
\mathcal{SI}_q(W)=m_\rho(L_{(q,W)}\cap \Omega)
\end{align}
for each $W\in S^2$. If there is a flat image plane $H$, which intersects all the light beams from $q$ through $\Omega$ and $q\notin H$, then image function $\mathcal{HI}_q:\, H\rightarrow \mathbb{R}$ is defined as 
\begin{align}\label{paraimage3}
\mathcal{HI}_q(p)=m_\rho(L_{(q,\, p-q)}\cap\Omega)
\end{align}
for each $p\in H$.
\end{definition}

The main theorem in this subsection is
\begin{theorem}\label{continuousproperty}
For almost every point $q$ in $\mathbb{R}^3$, and almost every direction $W\in S^2$, the image functions $\mathcal{SI}_q$, $\mathcal{HI}_q$ and $\mathcal{I}_W$ are continuous functions. 
\end{theorem}

In fact we only need to show that if a point $q\in \mathbb{R}^3$ is not a singular source then the image functions $\mathcal{SI}_q$ and $\mathcal{HI}_q$ of the diverse beam from $q$ is continuous. Similarly we will also show that if a direction $W\in S^2$ is not a singular direction then the image function $\mathcal{I}_W$ of the parallel beam with direction $W$ is continuous. The proof of the theorem \ref{continuousproperty} is based on the following lemma in real analysis.

\begin{lemma}\label{integrationlemma}
If $f(x,y):\, A\times A'\to \mathbb{R}$ is a compactly supported (lower, upper) continuous function, and $f(x,y)$ is bounded, then the integration
\begin{align}
g(y)=\int_{A} f(x,y) dx
\end{align}
is a (lower, upper) continuous function from $A'$ to $\mathbb{R}$.
\end{lemma}

\begin{proof}[Proof of Theorem \ref{continuousproperty}]
Firstly we define the following functions
\begin{itemize}
\item $\Phi :\, \mathbb{R}^3\to \mathbb{R}$ is the function which equals $\rho$ on $\Omega$ and $0$ elsewhere; 
\item $\overline{\Phi}:\, \mathbb{R}^3\to \mathbb{R}$ is the function which equals $\rho$ on $\bar{\Omega}$ and $0$ elsewhere;
\item $\Phi^\Sigma:\, \mathbb{R}^3\to \mathbb{R}$ is the function which equals $\rho$ on $\Sigma$ and $0$ elsewhere. 
\end{itemize} 
Clearly they satisfy $\overline{\Phi}=\Phi+\Phi^\Sigma$. Because $\Omega$ is an open domain, we have that 
\begin{itemize}
\item $\Phi$ is a compact supported boudned and lower continuous function;
\item $\overline{\Phi}$ is a compact supported boudned and upper continuous function.
\end{itemize}

On the other hand, for any direction $W\in S^2$, there is a coordinate system $(x,y,z)$ for $\mathbb{R}^3$ such that $W=(0,0,1)$ and the perpendicular plane $H=W^{\perp}$  is the $x-O-y$ plane. Therefore the image function is
\begin{align}
\mathcal{I}_W(x,y)=\int_{\mathbb{R}}\Phi(x,y,z)dz,
\end{align}
and the correspondence functions are
\begin{align}
\mathcal{I}^\Sigma_W(x,y)=\int_{\mathbb{R}}\Phi^\Sigma(x,y,z)dz,
\end{align}
\begin{align}
\overline{\mathcal{I}}_W(x,y)=\int_{\mathbb{R}}\overline{\Phi}(x,y,z)dz.
\end{align}
Then Lemma \ref{integrationlemma} implies that $\mathcal{I}_W(x,y)$ is a lower continuous function and $\overline{\mathcal{I}}_W(x,y)$ is an upper continuous function. 

If $W$ is not a singular direction, then $\mathcal{I}_W^\Sigma(x,y)=0$ and that $\mathcal{I}_W(x,y)=\bar{\mathcal{I}}_W(x,y)$ is both lower and upper continuous. In this case, the image function $\mathcal{I}_W(x,y)$ is a continuous function. 

For any point $q\in \mathbb{R}^3$, let $(r, \theta, \phi)$  be the polar coordinate system centered at point $q$. Then the image function is
\begin{align}
\mathcal{SI}_q(\theta,\phi)=\int_{0}^\infty\Phi(r,\theta,\phi)dr,
\end{align}
and the correspondence functions are
\begin{align}
\mathcal{SI}^\Sigma_q(\theta,\phi)=\int_{0}^{\infty}\Phi^\Sigma(r, \theta,\phi)dr,
\end{align}
\begin{align}
\overline{\mathcal{SI}}_q(\theta,\phi)=\int_{0}^{\infty}\overline{\Phi}(r, \theta,\phi)dr.
\end{align}
Then Lemma \ref{integrationlemma} implies that $\mathcal{SI}_q(\theta,\phi)$ is a lower continuous function and $\overline{\mathcal{SI}}_q(\theta,\phi)$ is an upper continuous function. 

If $q$ is not a singular source, then $\mathcal{SI}_q^\Sigma(\theta,\phi)=0$ and that $\mathcal{SI}_q(\theta,\phi)=\bar{\mathcal{SI}}_q(\theta,\phi)$ is both lower and upper continuous. In this case, the image function $\mathcal{SI}_q(\theta,\phi)$ is a continuous function. 

For the image function $\mathcal{HI}_q:\,H\rightarrow \mathbb{R}$, consider the map $T:\, H\rightarrow S^2$ with $T(p)=L(o,\,p-q)\cap S^2$ for any $p\in H$, and $o$ as the center of $S^2$. Because $q\notin H$, the transformation $T:\,H\rightarrow S^2$ is continuous and one-to-one. Therefore the image function 
\begin{align*}
\mathcal{HI}_q=\mathcal{SI}_q\circ T
\end{align*}
is also continuous as a composition of two continuous maps.

At last, Theorem \ref{mainproof} shows that the singular directions and singular sources are zero measure sets in $S^2$ and $\mathbb{R}^3$ respectively. Therefore, the image functions are almost surely continuous. 
\end{proof}

\section{Numerical Experiments}
In this section we compare several image denoising methods for Poisson noise removal to verify our analysis above. Currently there are many successful image denoising methods, such as the total variation (TV) model~\cite{Rudin92,Brune09,Chan07,Le07,Figueiredo09,Panin99,Setzer09b,Wu11,Zanella09}, the wavelet frame thresholding and regularization methods~\cite{Donoho95,Cai09,JS1,CDOS12}, anisotropic diffusion~\cite{Perona90,Weickert98}, and nonlocal mean method~\cite{Buades05}. We choose only to compare the TV model and wavelet frame regularization method to demonstrate our analysis in previous section, since these two methods use $\ell_1$ minimization of different regularizers and the regularizers reveal the smoothness of the
underlying data.


Compared to the TV model, which favors only piecewise constant images, wavelet frame regularization is more flexible and can model smooth images very well.
As shown in \cite{CDOS12}, different wavelet frame based approaches can be used to approximate the TV method with different orders of differential operators, which also justifies that wavelet frame based approach can be applied more successfully in this smooth image situation.
Actually in one-dimensional case, the TV model is equivalent to the simplest frame based model, i.e., the Haar framelet model, which approximates the first order differential operators \cite{SWBMW}.

We now compare these two methods for images generated by transmission imaging. In practice an image function is recorded by sensing elements on the imaging plane. The sensing elements sample the image function. However, the samples can approximate the image function very well as long as the resolution is high enough. Without loss of generality, we assume that the resolution is $M\times N$ and thus an image is an $M\times N$ real valued 2-dimensional array.

Suppose $u=\{u_{i,j}\}_{1\leq i\leq M,1\leq j\leq N}$ is an image. Its discrete gradient is given by
$$(\nabla u)_{i,j} =((\mathring{D}_x^+u)_{i,j},(\mathring{D}_y^+u)_{i,j}),$$
where $\mathring{D}_x^+$ and $\mathring{D}_y^+$ forward difference
operators with periodic boundary condition ($u$ is periodically
extended); see, e.g., \cite{Wu10}. Consequently fast Fourier
transform can be adopted in our algorithm. The TV model for Poisson
noise removal is as follows:
\begin{equation}
\label{eq-TVKL-model} \min\limits_{u}
E_{\mathrm{TVKL}}(u)\equiv\alpha\mathrm{TV}(u)+\sum\limits_{i,j,(Ku)_{i,j}>0}\left((Ku)_{i,j}-f_{i,j}\log (Ku)_{i,j}\right),
\end{equation}
where $\mathrm{TV}(u)=\sum\limits_{i,j}|(\nabla u)_{i,j}|$.

To present wavelet frame regularization model, we need to first introduce the discrete framelet transform. A discrete framelet transform applies some discrete convolutions to a signal. The convolution kernels are actually some filters including a low pass filter and some high pass filters. The low pass filter is the refinement mask of a refinable function, while the high pass filters are determined by the masks of framelets represented by the refinable function. Starting from a refinable function, one may construct an MRA-based tight frame system using the unitary extension principle (UEP) \cite{RS1}. Also, a usual way to construct multivariate framelets is using tensor-products of univariate framelets. For detailed and comprehensive introduction to the MRA-based wavelet frame theory, fast wavelet frame transform, please refer to \cite{Dong10}. In the following we list three simple but very useful univariate framelets which are used in the experiments.
\begin{itemize}
\item Haar framelets. The function $h(x)=1$ for $x\in[0,1]$ and $0$ otherwise has a refinement mask $a_0=\frac{1}{2}[1,1]$. The high pass filter is defined as $a_1=\frac{1}{2}[1,-1]$.
\item Piecewise linear framelets. The piecewise linear B-spline
$B_2(x)=\max(1-|x|,0)$ has a refinement mask
$a_0=\frac{1}{4}[1,2,1]$. The high pass filters are defined as
$a_1=\frac{\sqrt{2}}{4}[1,0,-1]$ and $a_2=\frac{1}{4}[-1,2,-1]$.
\item Piecewise cubic framelets. The centered piecewise cubic B-spline $B_4$ has a refinement mask $a_0=\frac{1}{16}[1,4,6,4,1]$. The high pass filters $a_1,a_2,a_3,a_4$ are defined as follows
$$
\begin{array}{ll}
a_1=\frac{1}{8}[-1,-2,0,2,1], &
a_2=\frac{\sqrt{6}}{16}[-1,0,-2,0,1],\\
a_3=\frac{1}{8}[-1,2,0,-2,1], & a_4=\frac{1}{16}[1,-4,6,-4,1].
\end{array}
$$
\end{itemize}

The model with wavelet frame regularization for Poisson noise removal is as follows
\begin{eqnarray}
\label{eq-WaveletFrameKL-model} && \min\limits_{u}
 E_{\mathrm{WFKL}}(u)\equiv\|\diag(\lambda)
\mathrm{W}u\|_1+\sum\limits_{i,j,(Ku)_{i,j}>0}\left((Ku)_{i,j}-f_{i,j}\log (Ku)_{i,j}\right),
\end{eqnarray}
where $\mathrm{W}u$ is the wavelet frame transform of $u$ and $\lambda$ is a positive vector.

Problems \eqref{eq-TVKL-model} and \eqref{eq-WaveletFrameKL-model} are both $\ell_1$ minimization problems. Recently many efficient methods have been developed to solve this kind of problems; see, e.g.,~\cite{Chambolle04,Chan99,Carter01,Cai09,Esser09,Goldstein08,Li09,Osher10,Tai09,Setzer09,Wang08,Wu10,Wu11,Yin08,Yin09,Zhu08} and references therein. In our implementation we applied operator splitting and augmented Lagrangian method with single inner iteration \cite{Glowinski89} to solve these problems.

We now provide some numerical experiments for verifying the analysis in previous section. The experiments were performed under Ubuntu and Matlab R2011b (version 7.13.0) on a workstation with Intel Xeon (Core 6) E5645 2.40GHz and 50Gb memory. We used $\frac{\|u^k - u^{k-1}\|_2}{\|u^{k}\|_2} \leq 5 \times 10^{-5}$  as the stopping criteria. The levels for all the three wavelet transforms were all set to be $1$. The parameters in the model $\alpha, \lambda$ were rigorously tuned up for all images to achieve the optimal performance. In each example, the density functions were reconstructed using the same method, respectively. In the experiments, the results from TV regularization, Haar wavelet system, piecewise linear B-Spline wavelet system and piecewise cubic B-spline wavelet system were compared in signal noise ratio (SNR) values and Frobenius norms.

The SNR value is defined as
\[
\mbox{SNR}: = 10 \log_{10} \frac{\|u_0-\bar{u_0}\|^2}{\|u-u_0\|^2}
\]
where $u_0$ is the original (clean) signal; $\bar{u_0}$ is the mean value of $u_0$; and $u$ is the noisy or restored signal. The errors are computed as the Frobenius norms on the differences between reconstructed density functions and the ground truth density function.

\begin{eg}
{This is an example for the 2D Shepp-Logan Phantom density function of size $256 \times 256$  in Matlab. The projection and reconstruction are done by Matlab built-in function ``fanbeam" and ``ifanbeam". There are $360$ projections with $509$ detector values for each projection. For each projection, we add Poisson noise by using Matlab ``imnoise" with a scaling to the projection image by a factor $10^{-12}$, e.g., ``imnoise($f*$factor, `Poisson')/factor'' where $f$ is a clear projection image. These projections corrupted with Poisson noises are then processed by the TV and framelet based denoising models $(\ref{eq-TVKL-model})$ and $(\ref{eq-WaveletFrameKL-model})$. We mention that here we did not compute the results of the Haar framelet based model, because the Haar framelet based model with level 1 is equivalent to the TV model in 1D case, as mentioned before. Fig.~\ref{figure1} shows the denoised results of two projections. It is clear that the wavelet frame systems with piecewise linear B-spline and cubic B-spline framelets generate better results with higher SNR values than the TV model. Also, the frame system with piecewise cubic B-spline framelets yields better results with higher SNR values than the frame system with piecewise linear B-spline framelets.  The clean projections, the noisy projections, and the various denoised projections are then used to reconstruct the 2D Phantom; See Fig. \ref{figure2}. It can be seen that better density functions with small errors have been reconstructed from the denoised projections by the frame based model. The frame based denoising model outperforms greatly the TV model for this kind of images and their denoised projection data generate better density function reconstructions. The smoother the framelet is, the better the results are.}
\begin{figure}[ht]
\begin{center}
\subfigure[Projection 100]{ \includegraphics[width=0.6\textwidth]{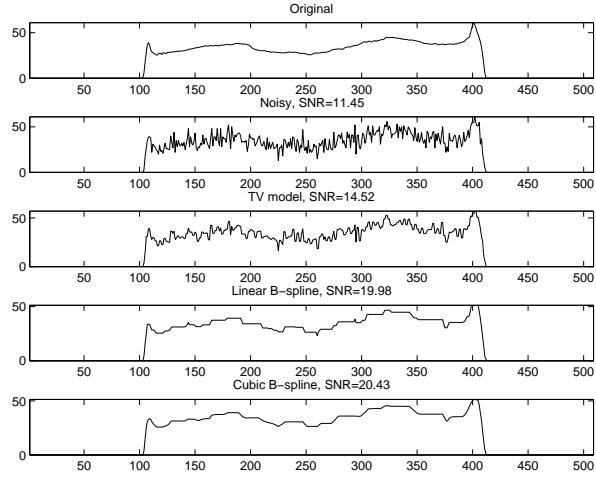} }
\subfigure[Projection 200 ]{ \includegraphics[width=0.6\textwidth]{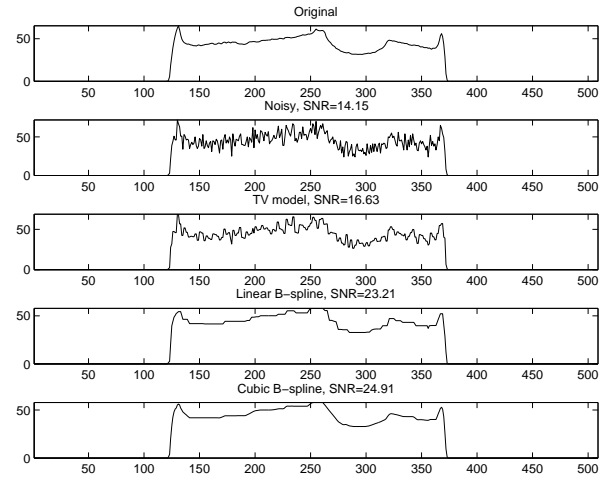}}
\caption{\small The comparisons between the TV and wavelet frame based models for denoising two projections of the 2D Phantom object. (a) The 100'th projection: the original projection; the corrupted projection with Poisson noise; the denoised results by the TV model, the piecewise linear B-spline framelet and piecewise cubic B-spline framelet models. (b) The 200'th projection: the original projection; the corrupted projection with Poisson noise; the denoised results  by the TV model, the piecewise linear B-spline framelet and piecewise cubic B-spline framelet models.}
\label{figure1}
\end{center}
\end{figure}
\vspace{.2cm}
\begin{figure}
\centering
\subfigure[Original]{ \includegraphics[width=0.3\textwidth]{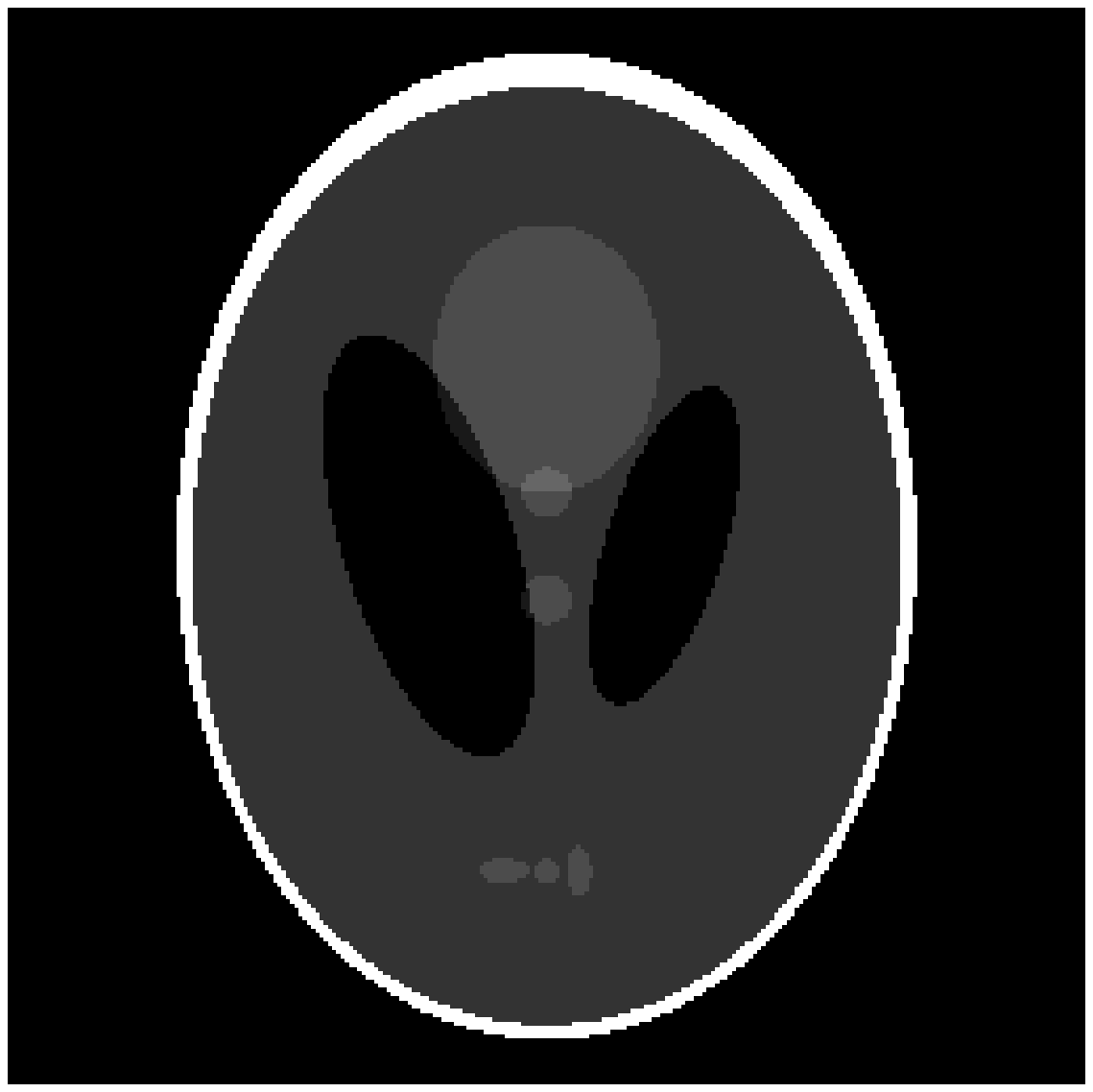} }
\subfigure[Reconstructed]{ \includegraphics[width=0.3\textwidth]{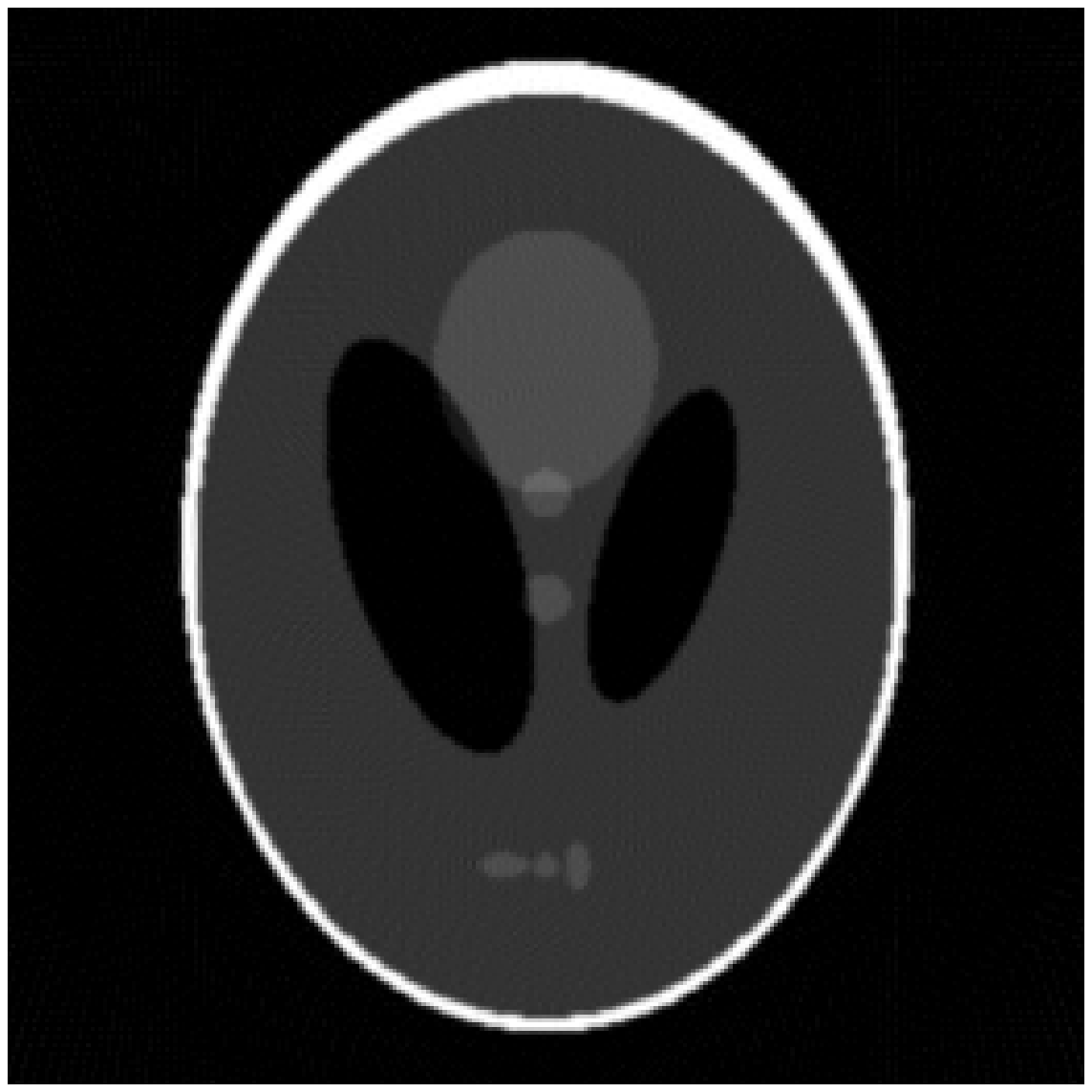} }
\subfigure[Noisy, error=61.53]{ \includegraphics[width=0.3\textwidth]{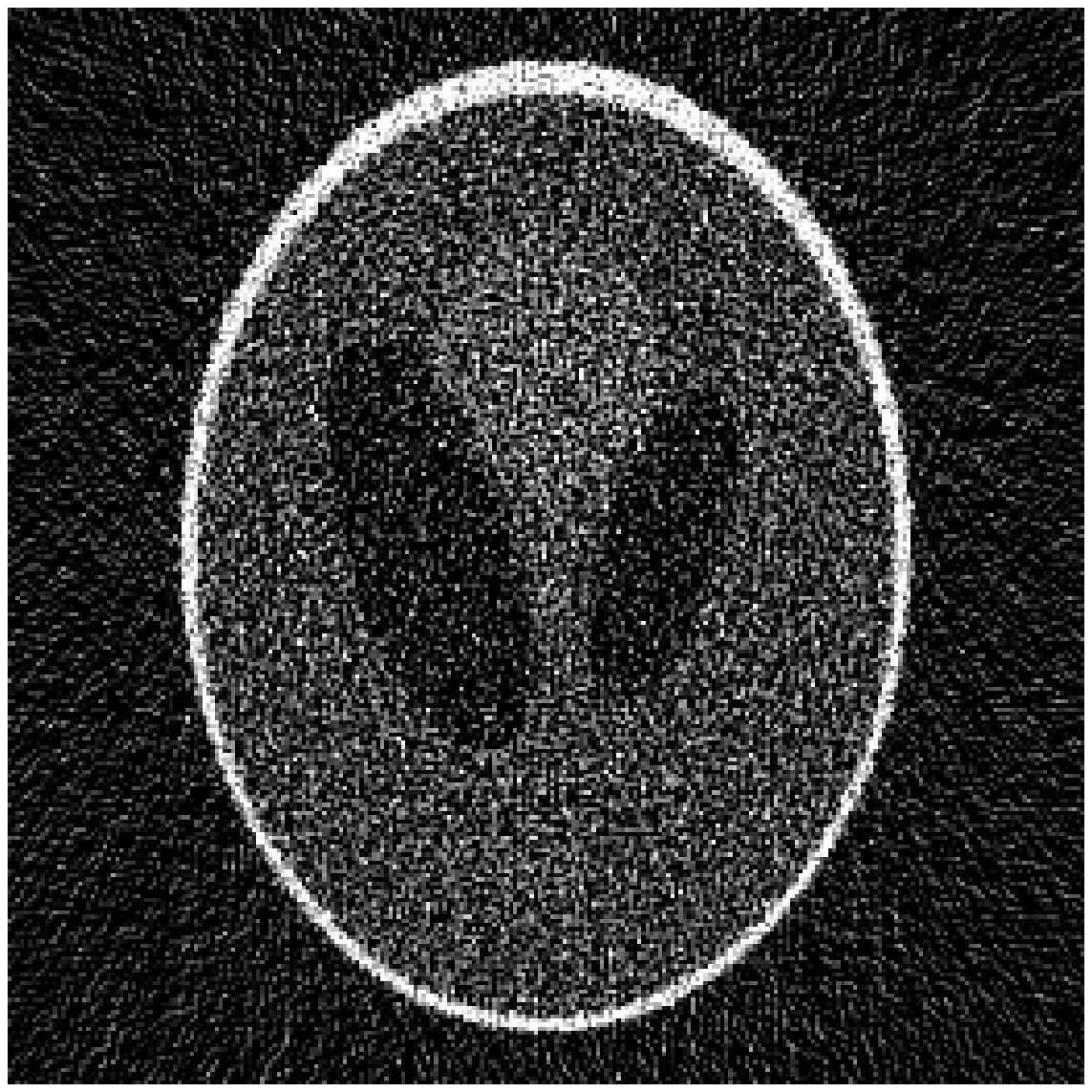} }
\subfigure[TV, error=39.12]{ \includegraphics[width=0.3\textwidth]{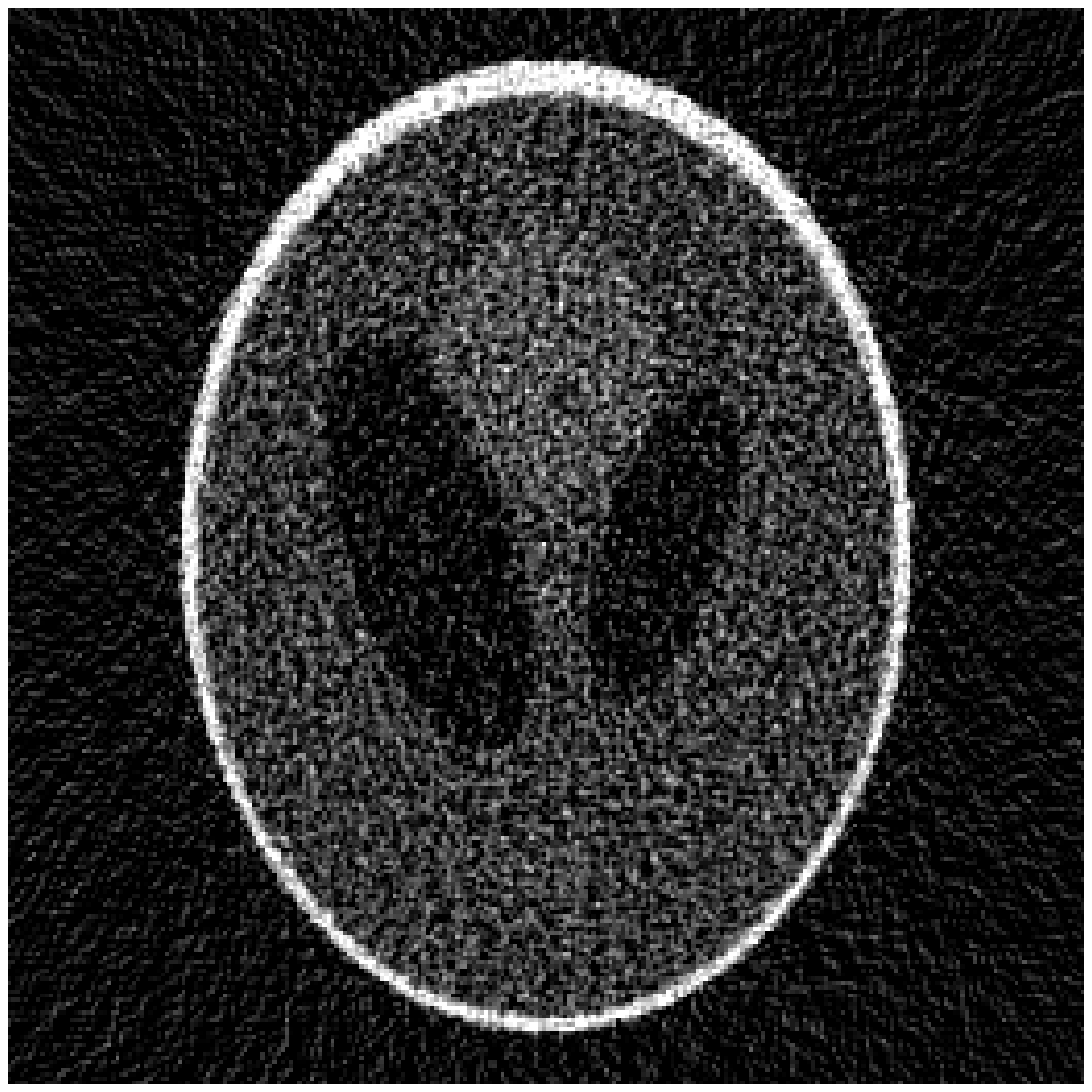} }
\subfigure[Linear, error=13.56]{ \includegraphics[width=0.3\textwidth]{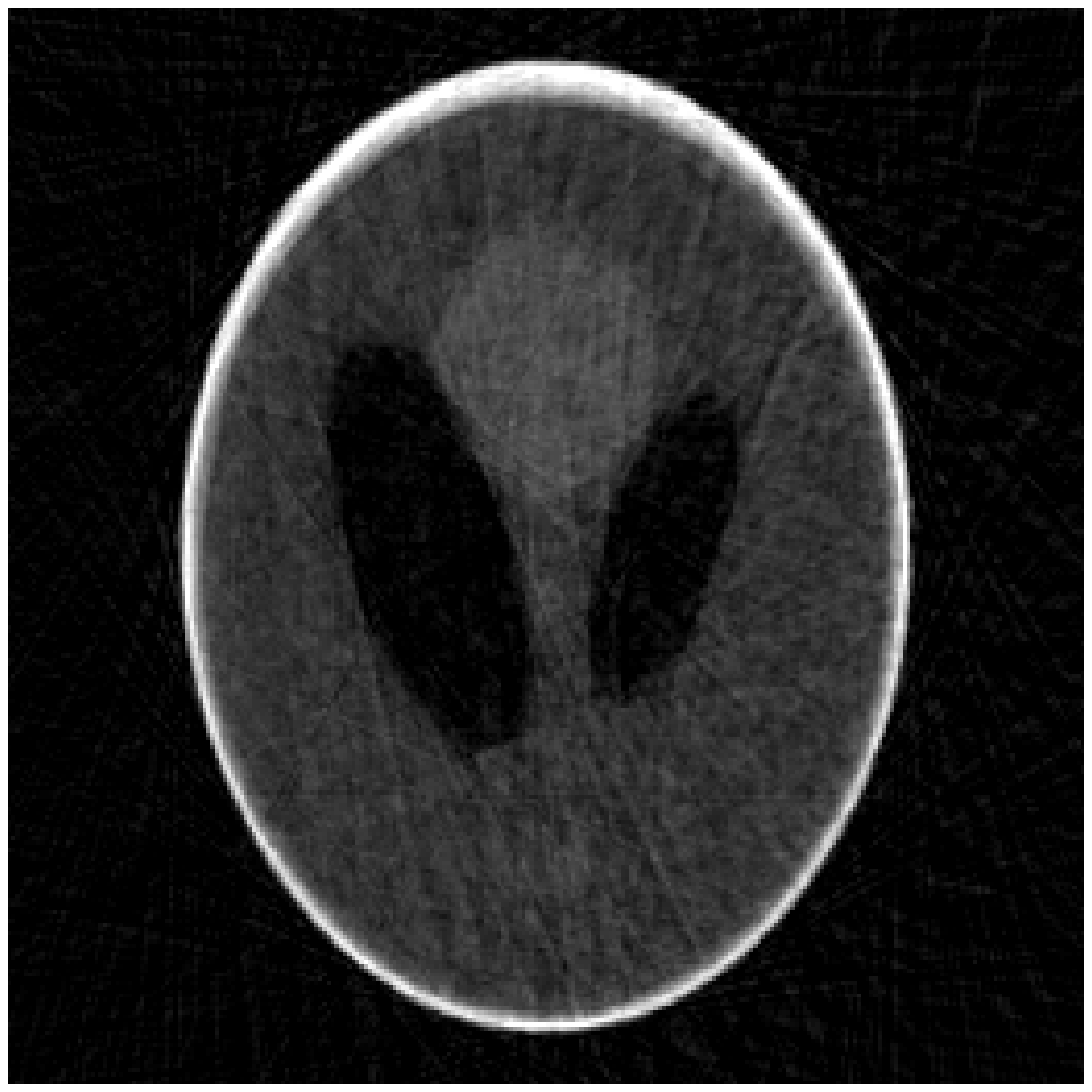} }
\subfigure[Cubic, error=11.32]{ \includegraphics[width=0.3\textwidth]{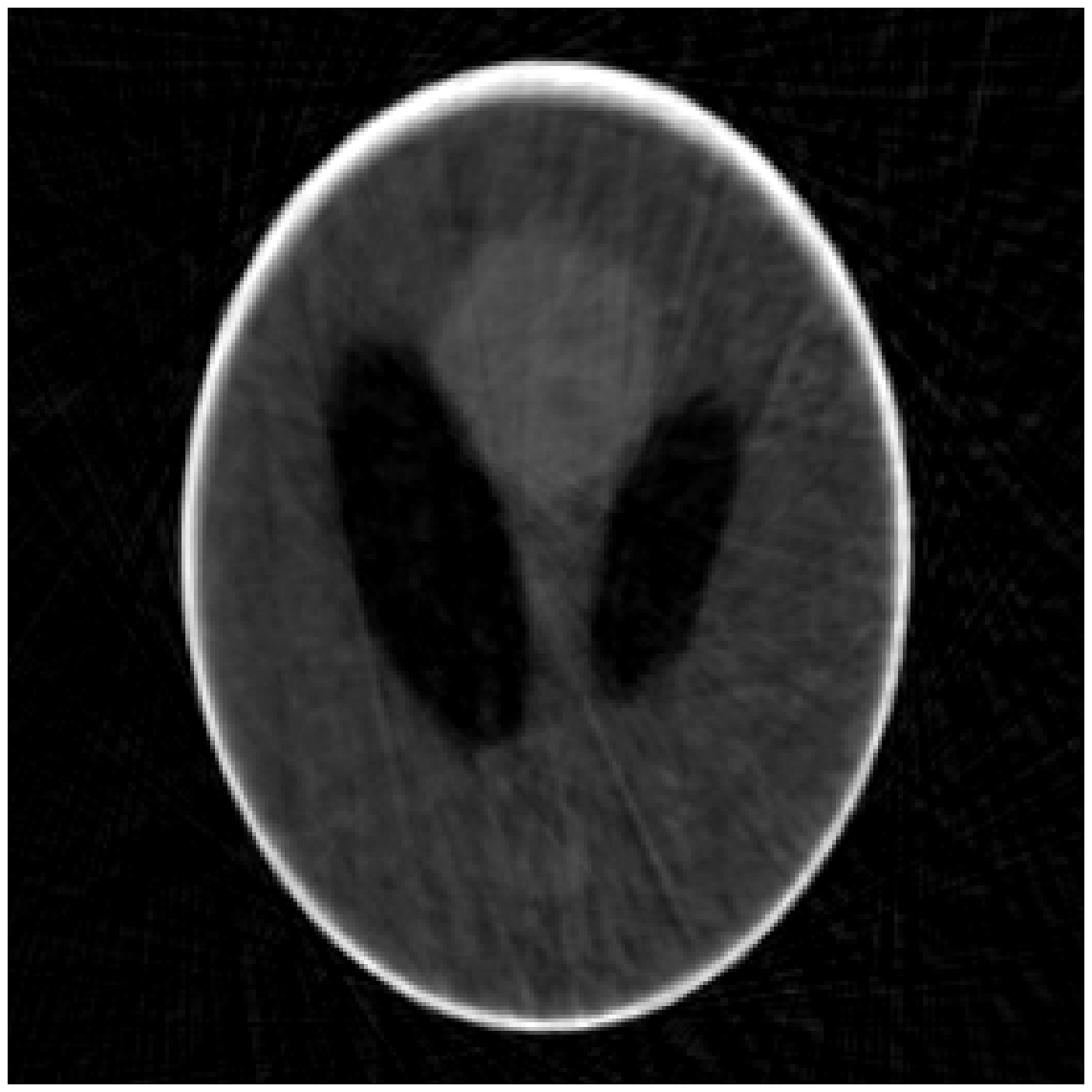} }
\caption{\small The comparison of the reconstructed Phantom using various versions of projections. The first row shows the original Phantom; the reconstructed Phantom from the clean projections; the reconstructed Phantom from the projections contaminated by Poisson noise. The second row shows the reconstructed Phantoms from denoised projections using the TV model, the piecewise linear B-spline model and the piecewise cubic B-spline model. All the reconstructions were generated by Matlab build-in function ``ifanbeam".} \label{figure2}
\end{figure}
\end{eg}

\begin{eg}
{In this example we show a numerical experiment for the 3D Shepp-Logan phantom data with size $128 \times 128 \times 128$. The projection and reconstruction are done by the 3D Cone beam CT projection  backprojection FDK Matlab code \cite{Kim13}. We obtain 84 projections with each having the size $600 \times 500$. For each projection, we add Poisson noise using Matlab built-in function ``imnoise" with a scaling factor $10^{-6}$. These projections are then processed by the TV and frame based denoising models $\eqref{eq-TVKL-model}$ and $\eqref{eq-WaveletFrameKL-model}$. The noisy projections and various denoised versions are used to reconstruct the density function of the 3D Phantom. The reconstruction is done by backprojection method with filter ``hamming". Fig. \ref{phantom128_proj} shows one randomly chosen projection denoised by the TV model, the Haar wavelet system, piecewise linear B-spline system and piecewise cubic B-spline system. The piecewise linear and piecewise cubic B-spline systems return both smoother denoised results with higher SNR values than the Haar system and TV model. Fig. \ref{phantom128_rec} shows the 64'th slice of the reconstructed density function. The SNR values of the denoised projections and the reconstruction errors demonstrate the advantage of the frame based model, especially the model with smoother framelets.}

\begin{figure}
\begin{center}
\subfigure[Original]{ \includegraphics[width=0.3\textwidth]{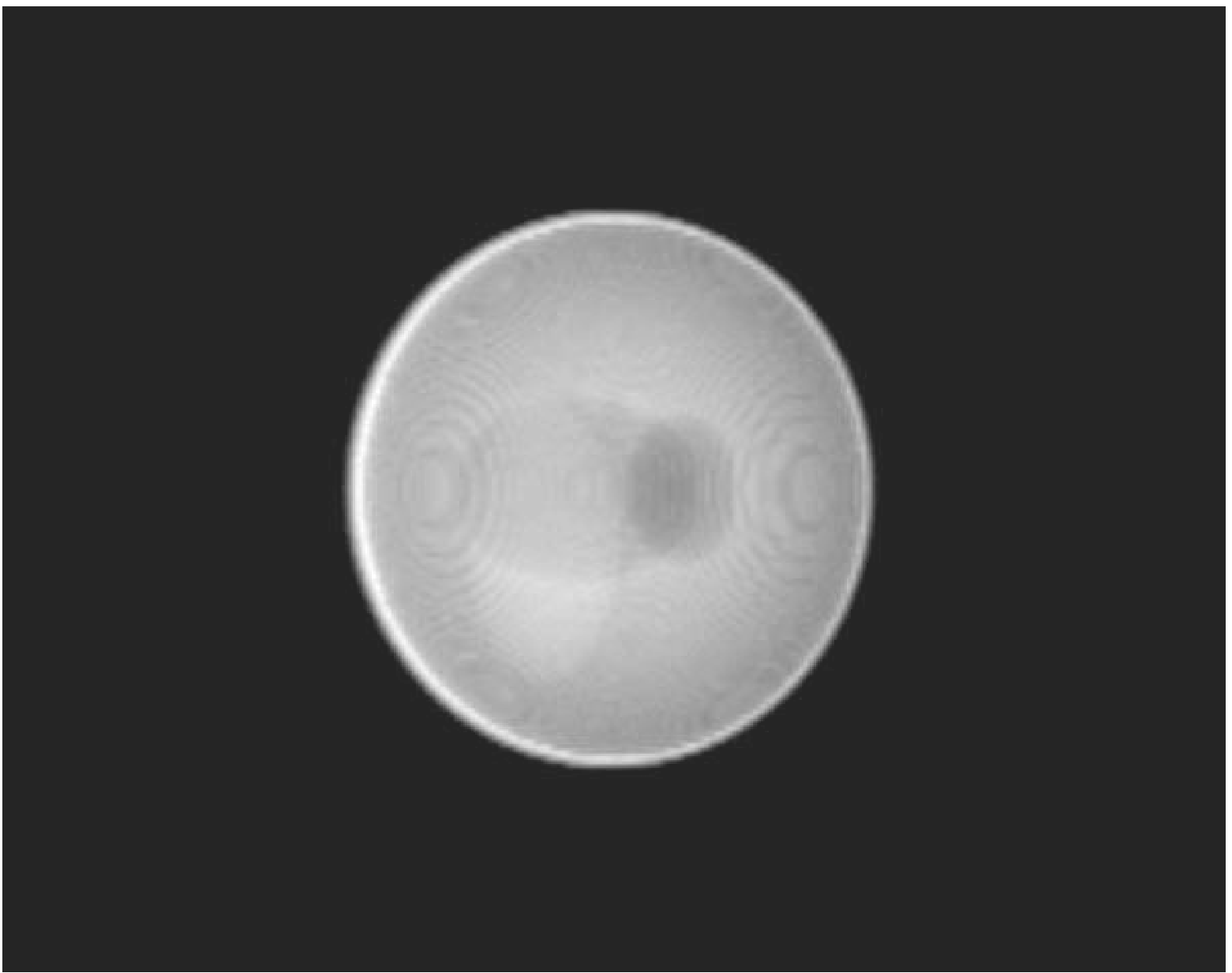} }
\subfigure[Noisy, SNR=9.25]{ \includegraphics[width=0.3\textwidth]{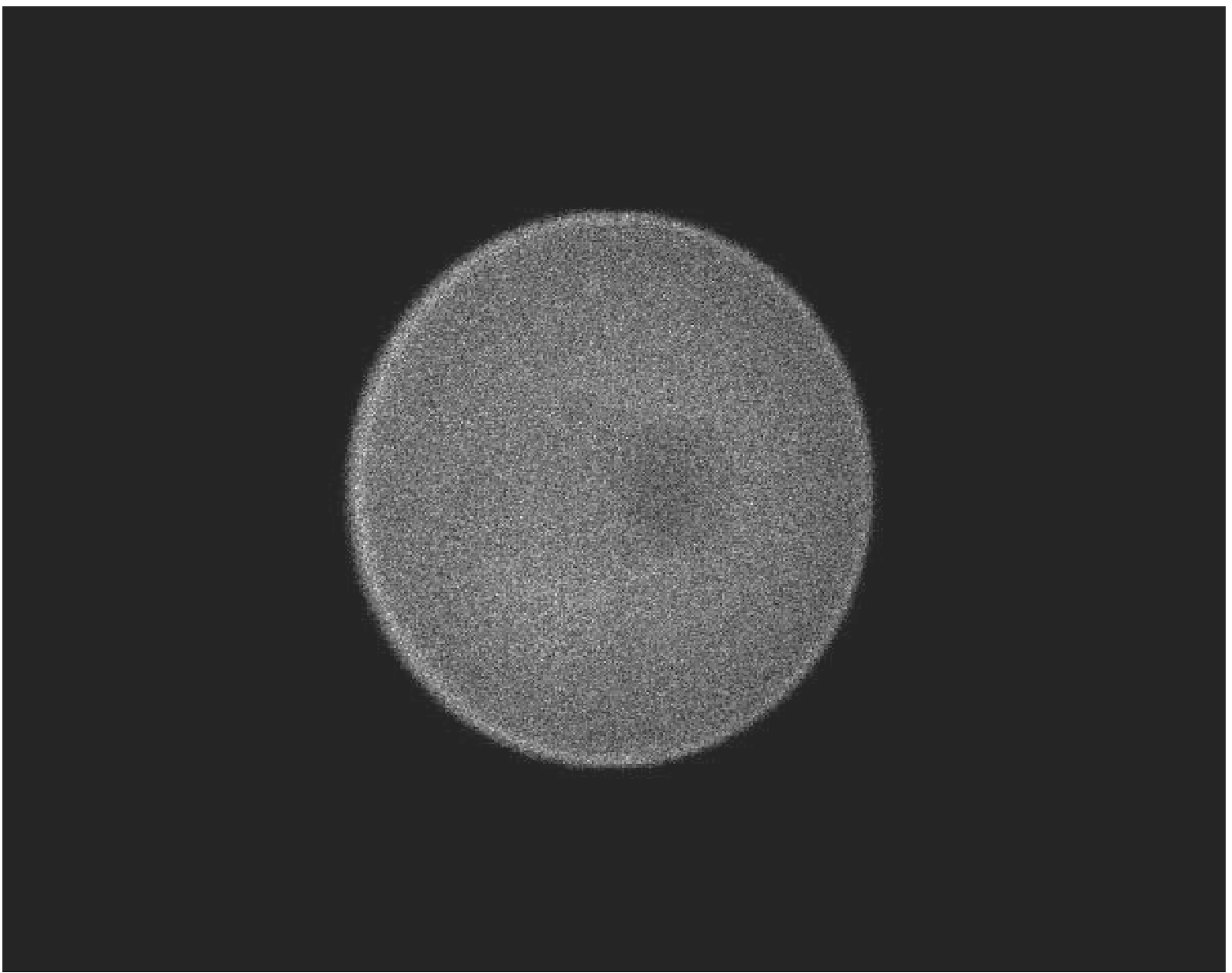} }
\subfigure[TV, SNR=24.62]{ \includegraphics[width=0.3\textwidth]{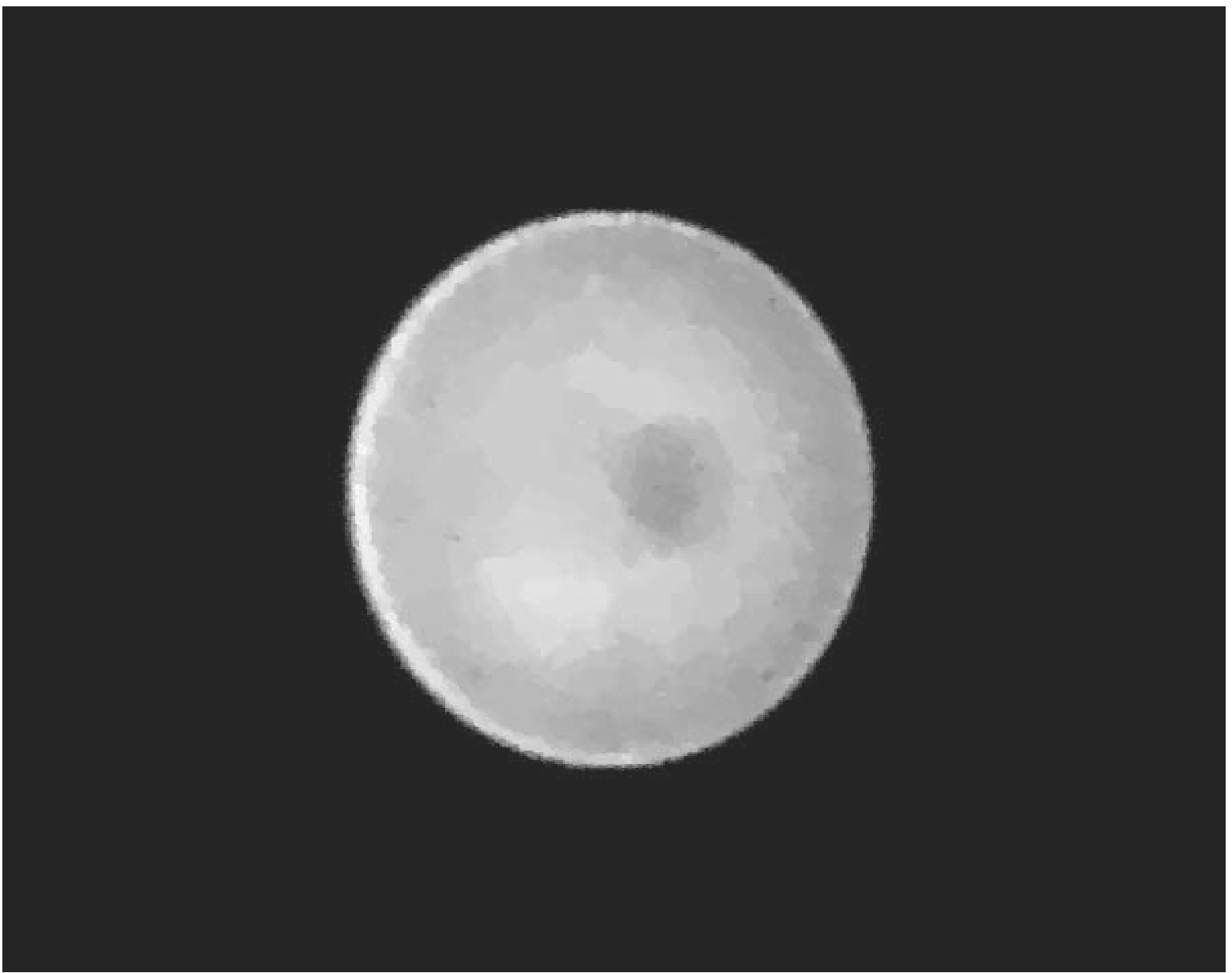} }
\subfigure[Haar, SNR=24.63]{ \includegraphics[width=0.3\textwidth]{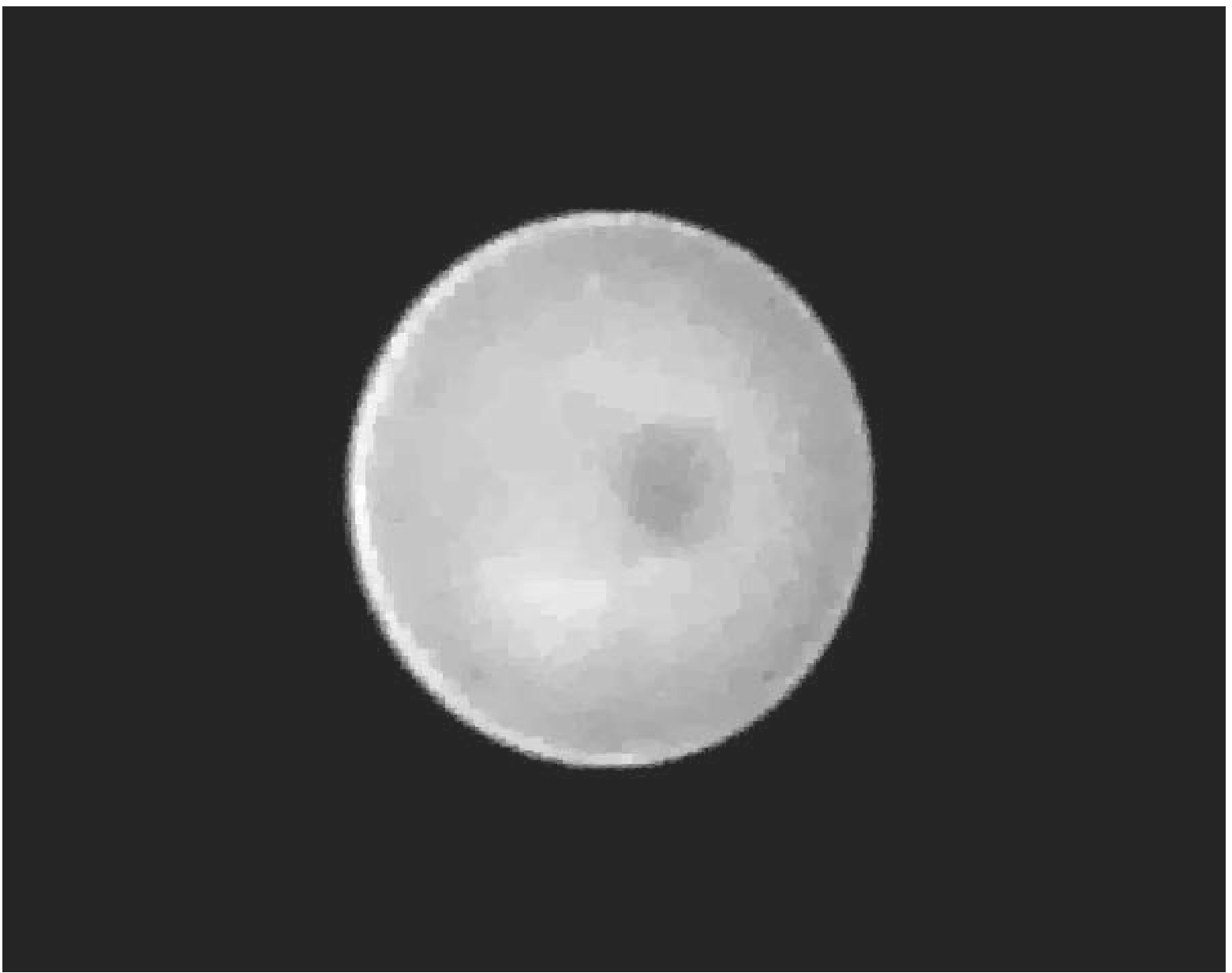} }
\subfigure[Linear, SNR=25.74]{ \includegraphics[width=0.3\textwidth]{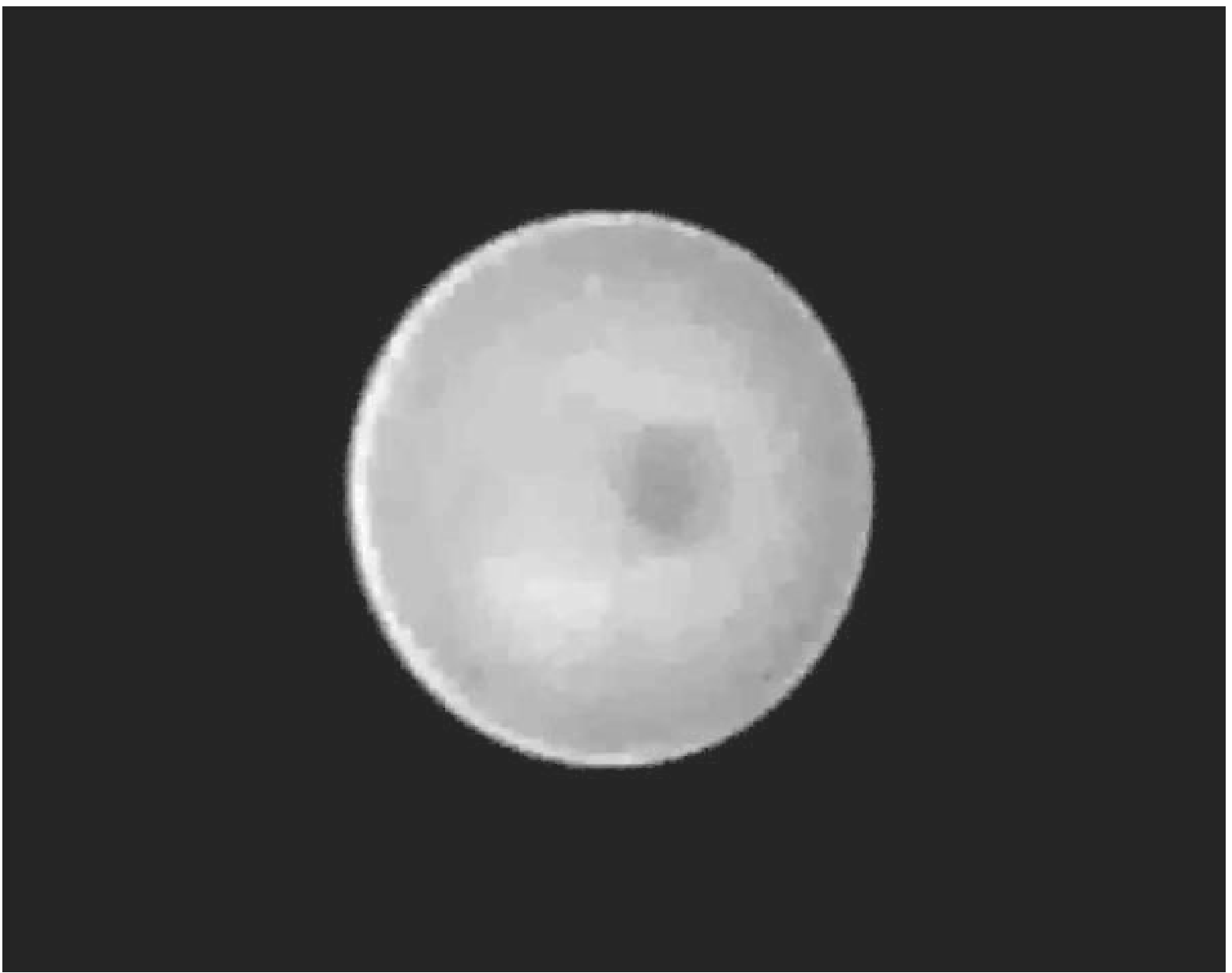} }
\subfigure[Cubic, SNR=26.50]{ \includegraphics[width=0.3\textwidth]{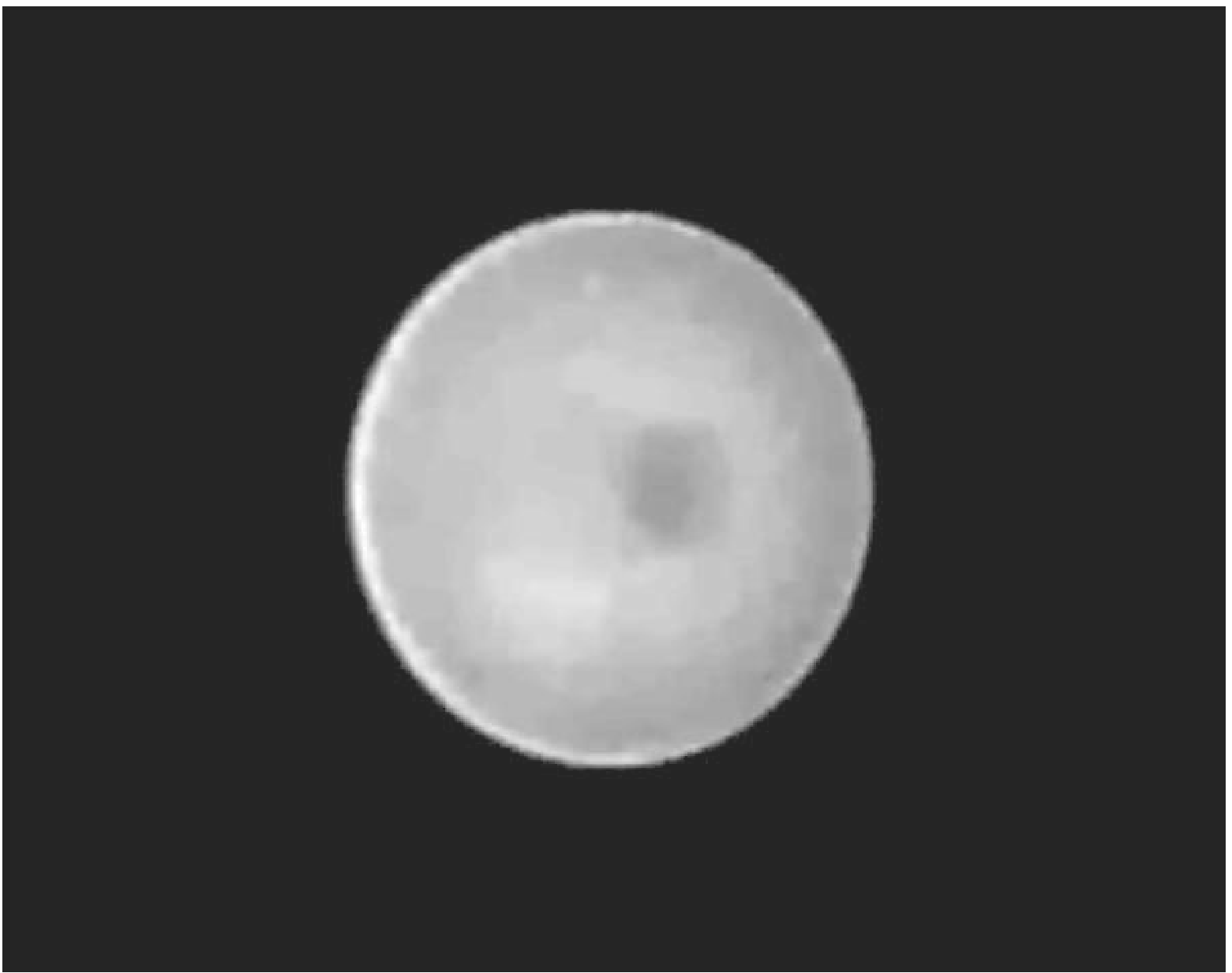} }
\end{center}
\caption{The comparison between the TV and wavelet frame based models for denoising a projection of the 3D Phantom. The first row shows the ground truth projection image, the noisy image and the denoised image by the TV model. The second row shows the denoised projection images by the Haar wavelet system, piecewise linear B-spline system and piecewise cubic B-spline system, respectively. }
\label{phantom128_proj}
\end{figure}

\begin{figure}
\begin{center}
\subfigure[Original]{ \includegraphics[width=0.3\textwidth]{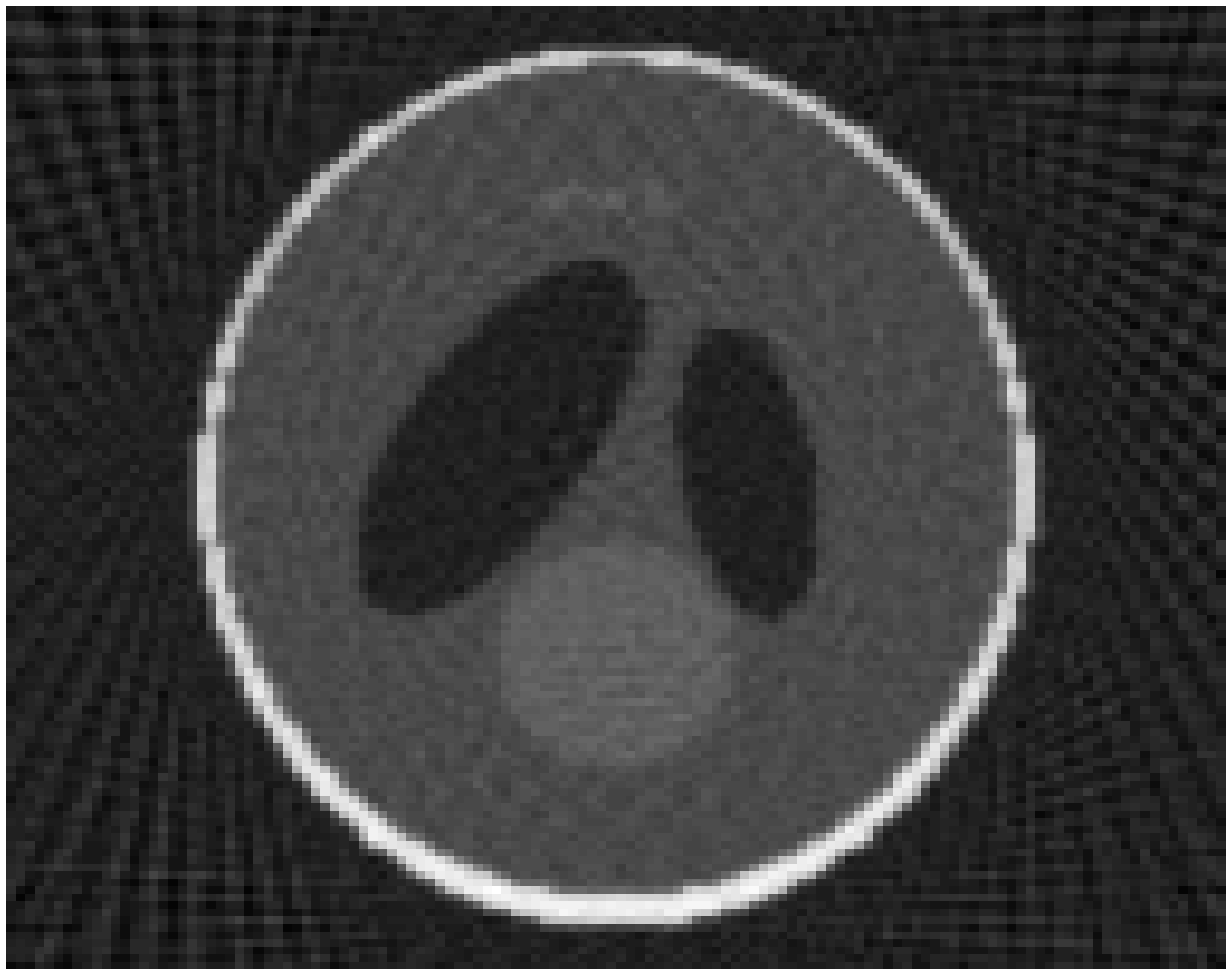} }
\subfigure[Noisy, error=346.33]{ \includegraphics[width=0.3\textwidth]{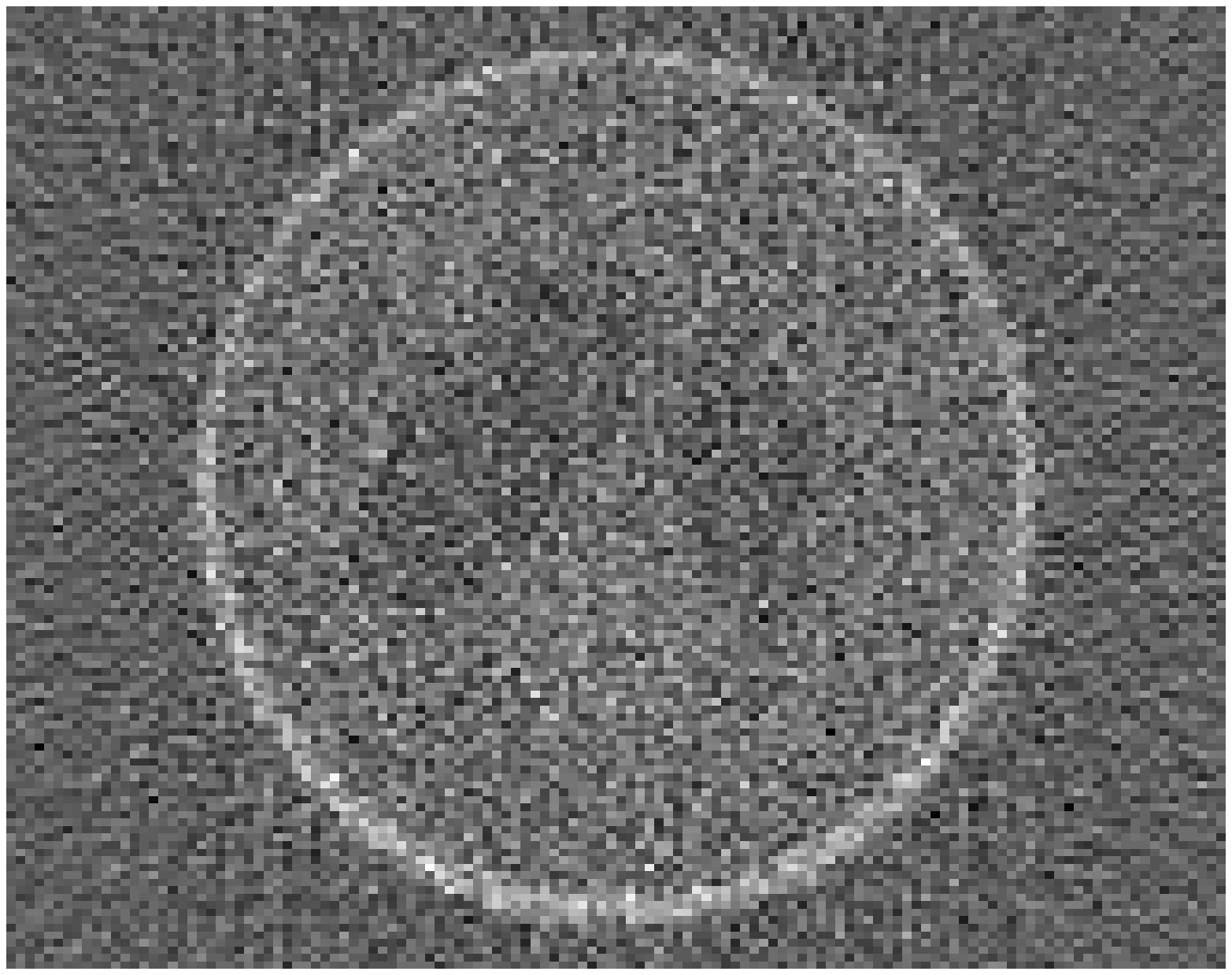} }
\subfigure[TV, error=56.98]{ \includegraphics[width=0.3\textwidth]{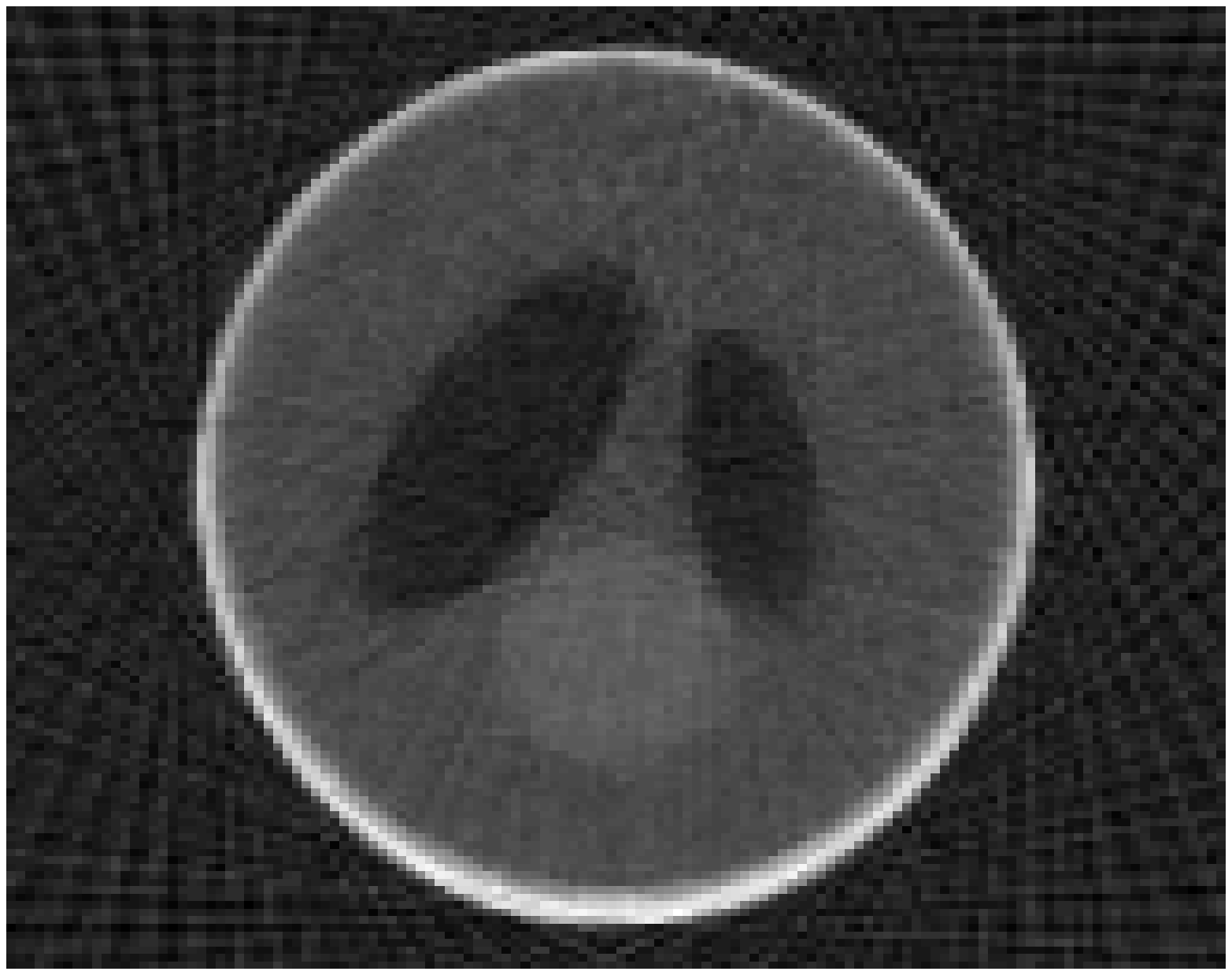} }
\subfigure[Haar, error=56.65]{ \includegraphics[width=0.3\textwidth]{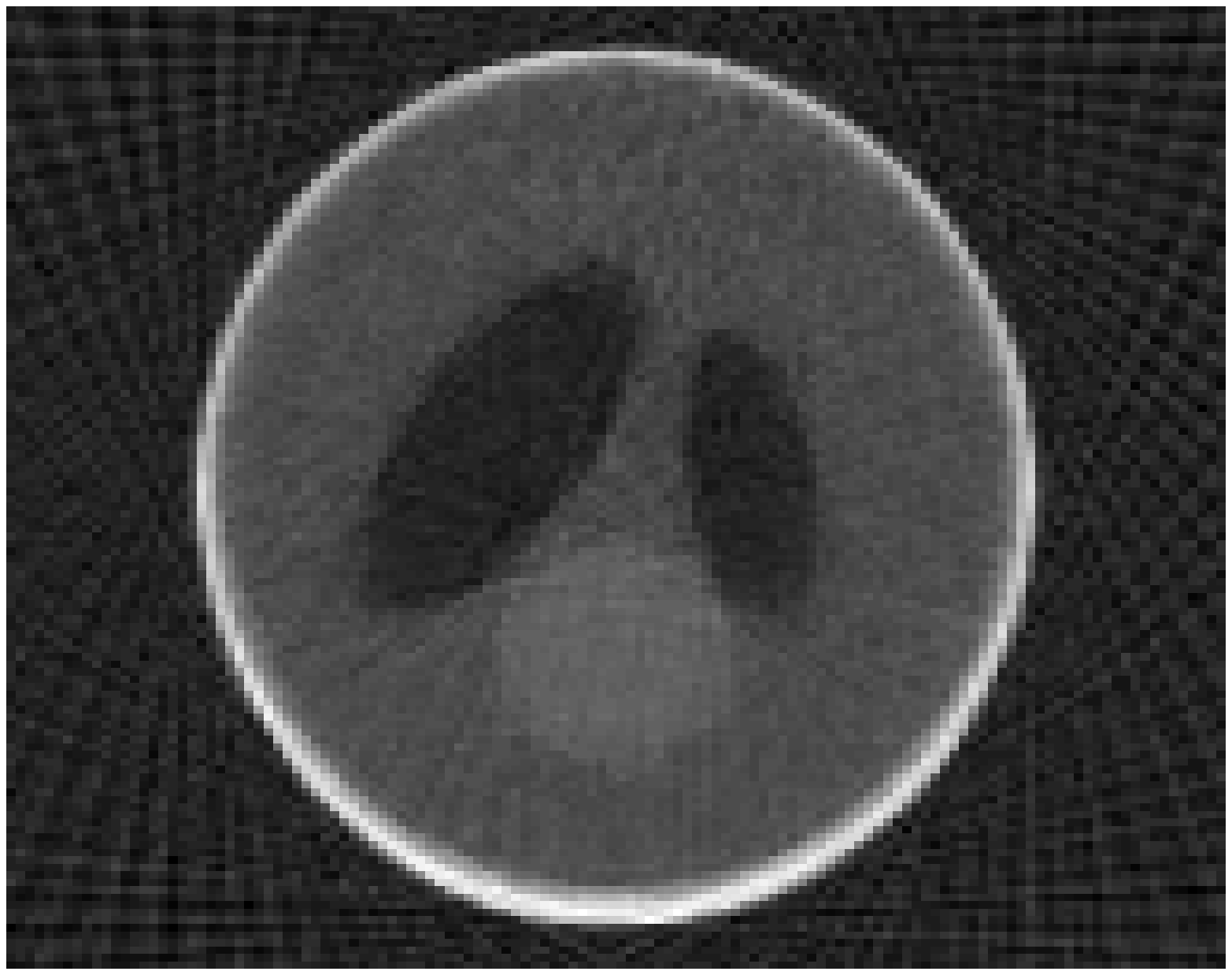} }
\subfigure[Linear, error=49.47]{ \includegraphics[width=0.3\textwidth]{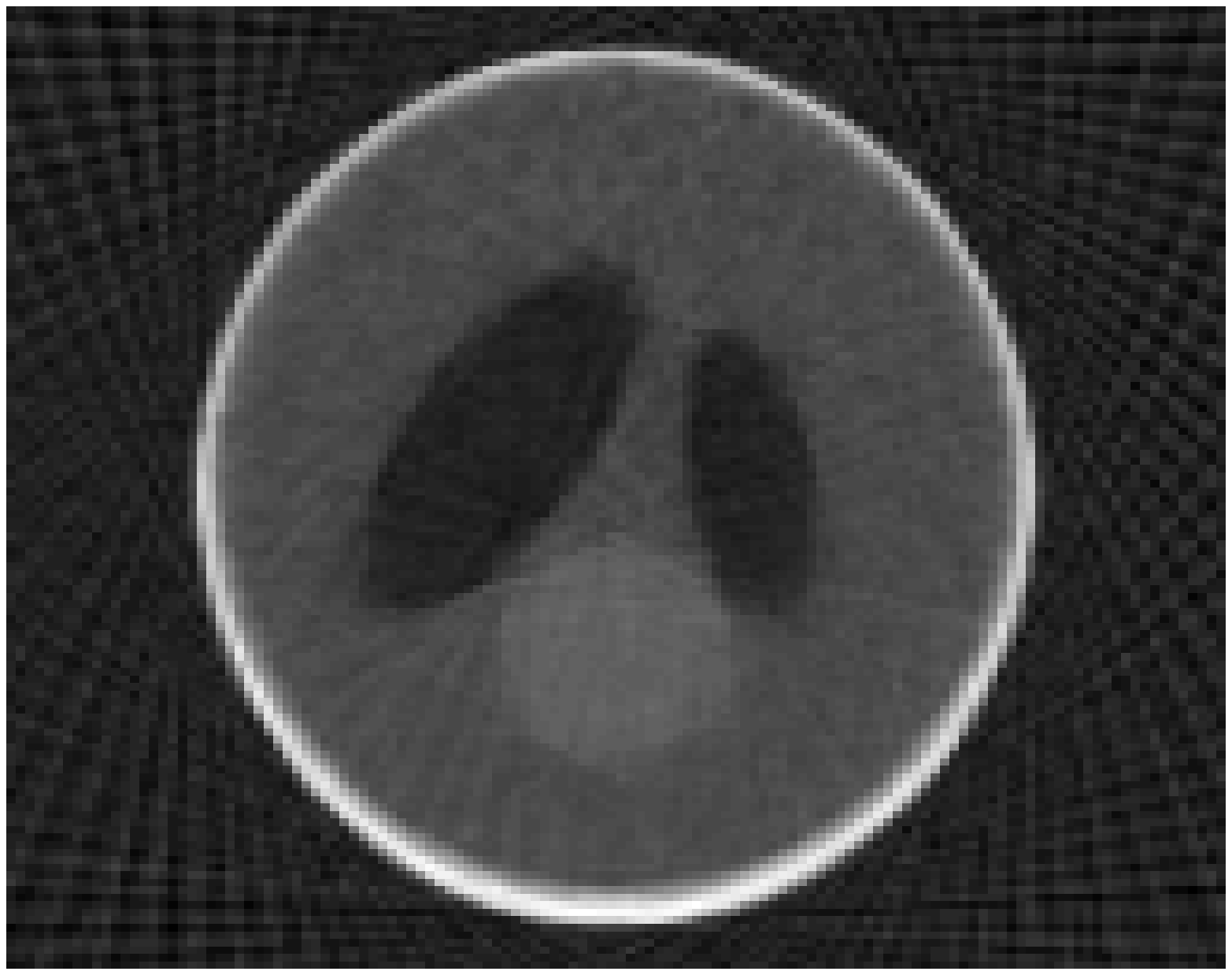} }
\subfigure[Cubic, error=44.11]{ \includegraphics[width=0.3\textwidth]{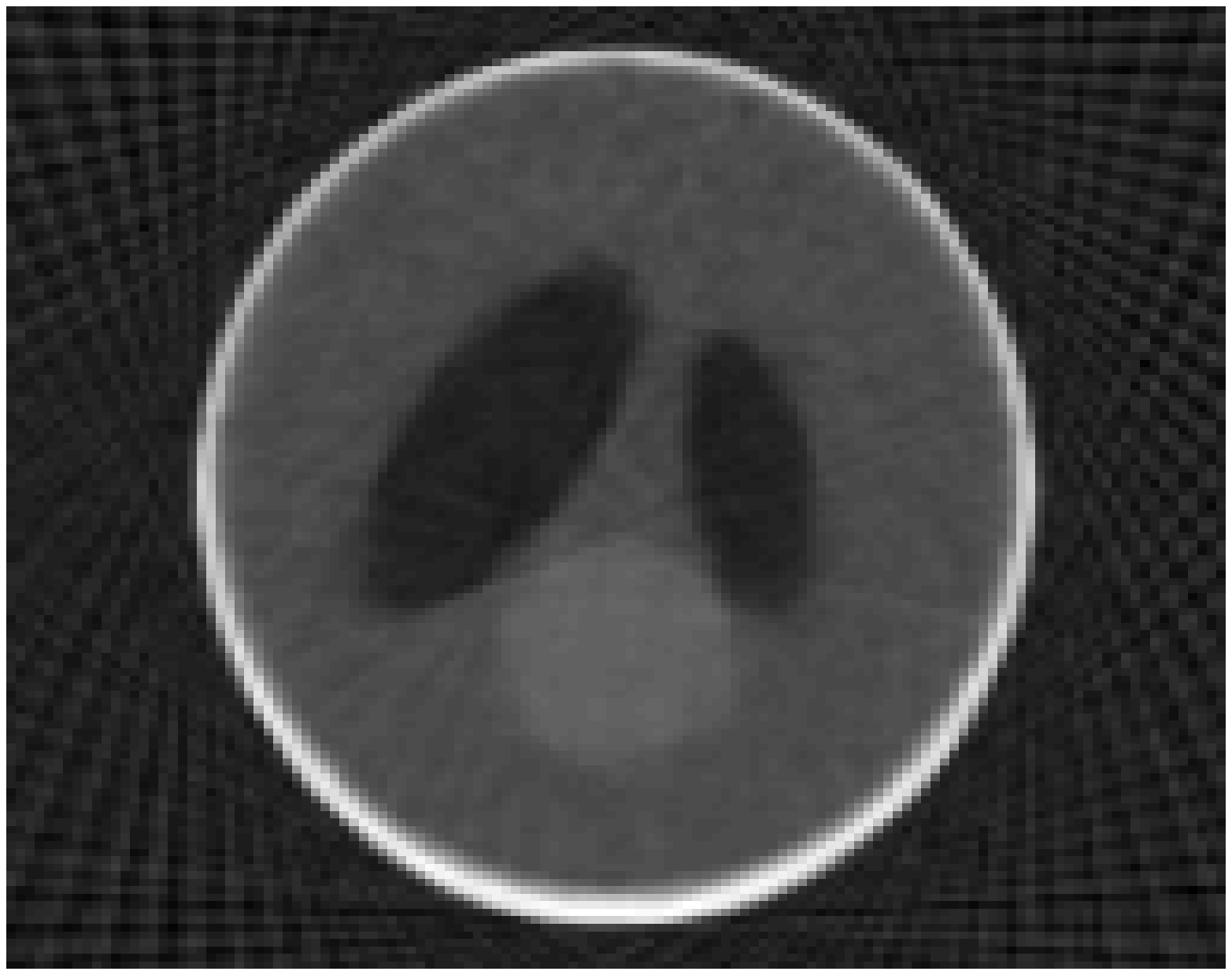} }
\end{center}
\caption{The comparison of the 64'th slice along z-direction of reconstructed 3D phantom object. The first row shows the ground truth, the objects from noisy projections and denoised projections by the TV model, respectively. The second row shows objects reconstructed with the denoised projections by the Haar wavelet system, the piecewise linear B-spline system and the piecewise cubic B-spline system, respectively.} \label{phantom128_rec}
\end{figure}
\end{eg}

\begin{eg}
{The data used for this experiment is the 3D medical data of human head with size $401 \times 401 \times 401$ from Cone beam CT. Totally there are 84 projections with size $800 \times 700$. Poisson noise is added to each projection by using Matlab ``imnoise" function with a scaling factor $5 \times 10^{-6}$. These projections are then processed by the TV and frame based denoising models \eqref{eq-TVKL-model} and \eqref{eq-WaveletFrameKL-model}. The noisy projections and various denoised versions are used to reconstruct the density function of the 3D Phantom. The reconstruction is done by backprojection method with filter ``hamming". Fig. \ref{medical251_proj} shows the 30'th projection denoised by the TV model, the Haar wavelet system, the piecewise linear and piecewise cubic B-spline system, respectively. Both the piecewise linear and piecewise cubic B-spline systems return better results than the Haar wavelet system and the TV model. Fig. \ref{medical251_rec} shows the 120'th slice of the reconstructed volume along z-direction and Fig. \ref{medical251_3d} gives the isosurface view of the 3D density function with function value 0.09. }

\begin{figure}
\begin{center}
\subfigure[Original]{ \includegraphics[width=0.3\textwidth]{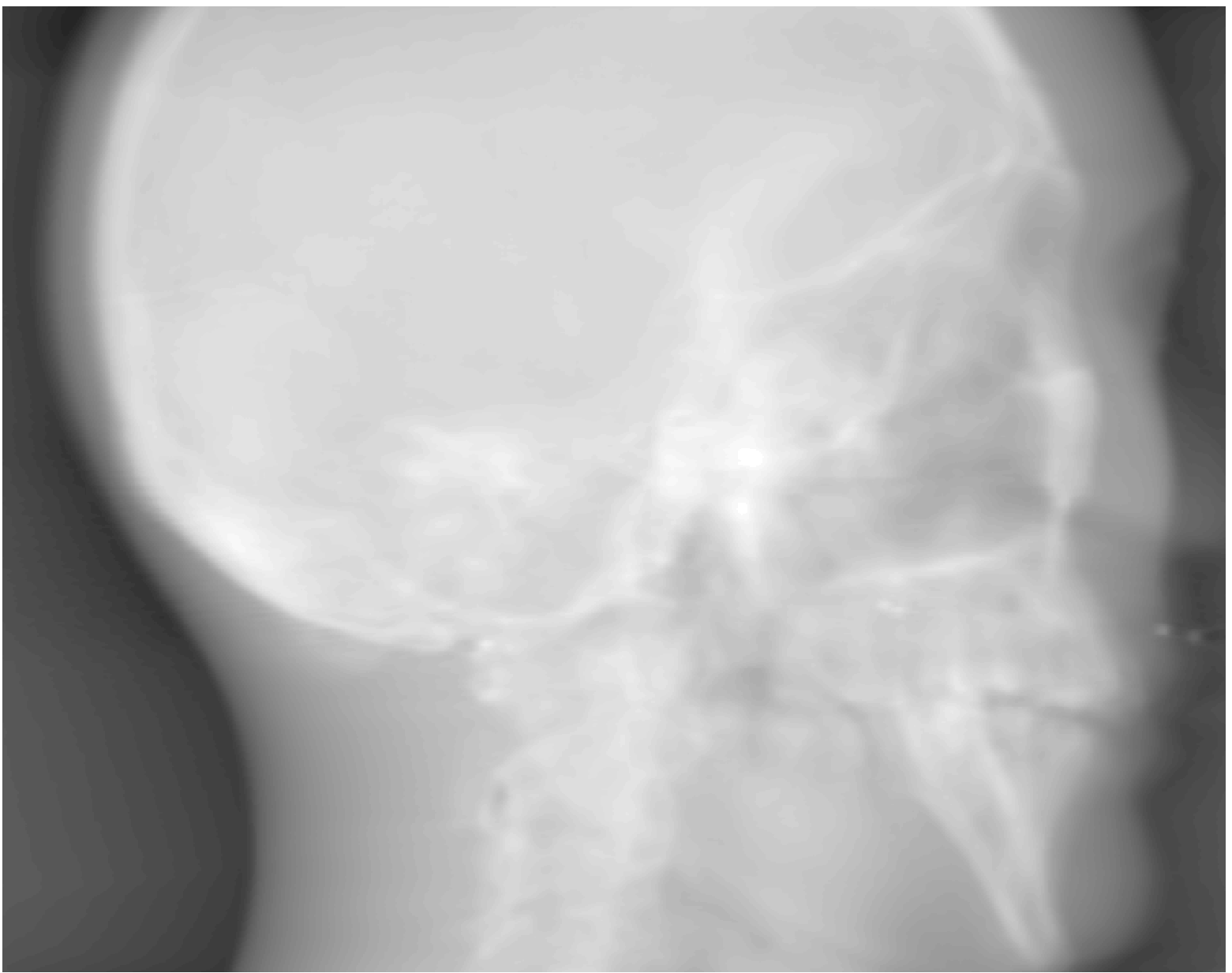} }
\subfigure[Noisy, SNR=7.77]{ \includegraphics[width=0.3\textwidth]{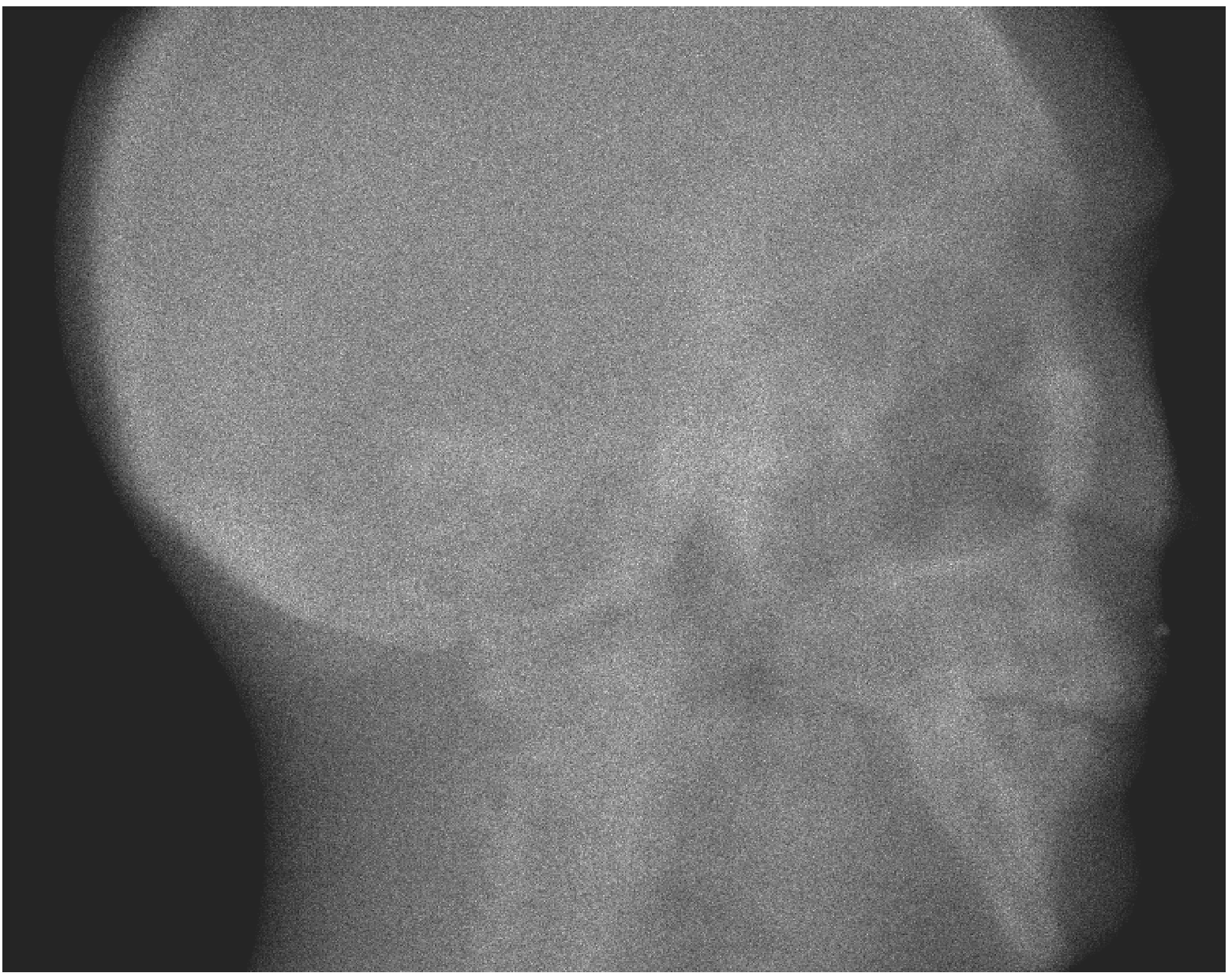} }
\subfigure[TV, SNR=13.07]{ \includegraphics[width=0.3\textwidth]{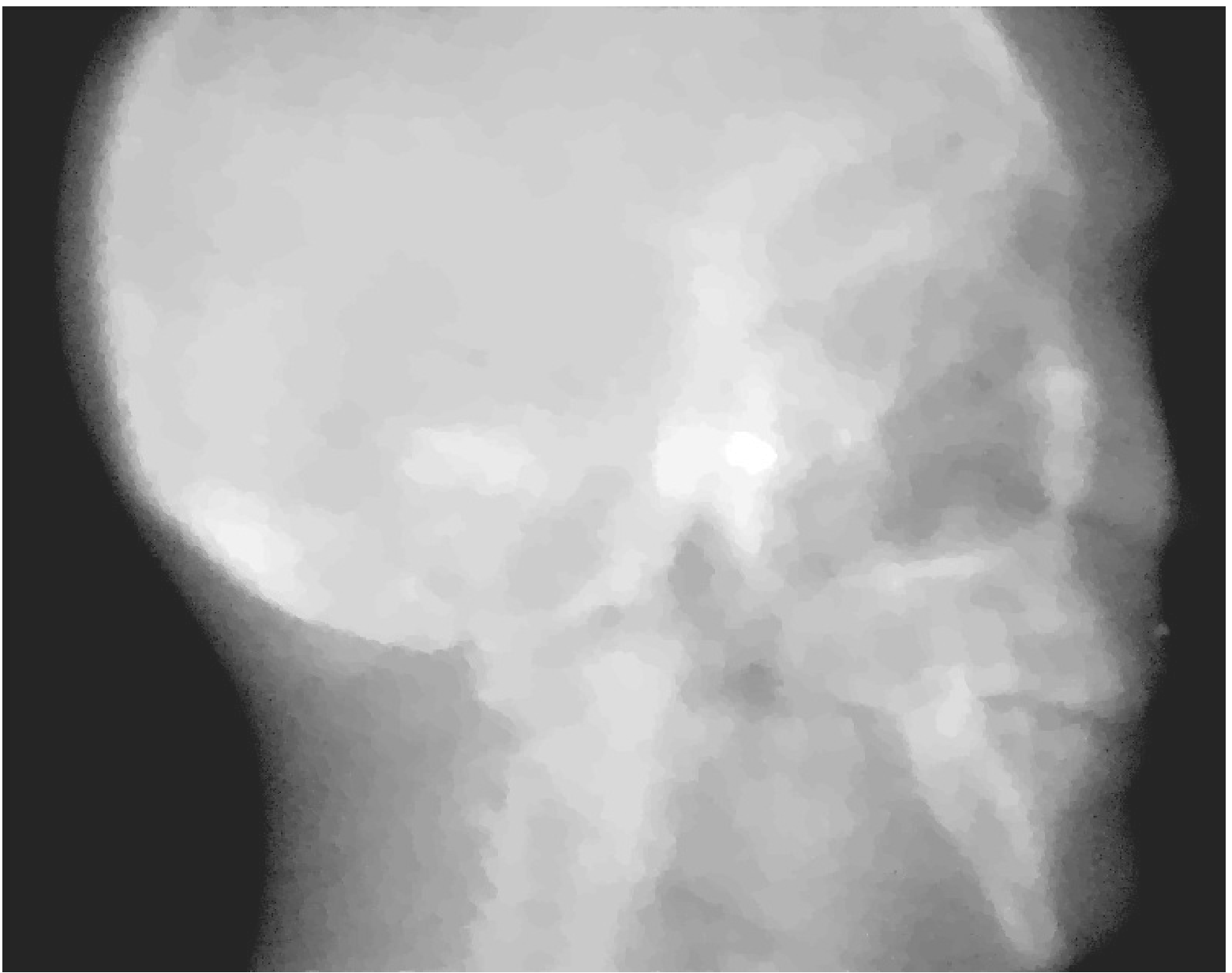} }
\subfigure[Haar, SNR=13.09]{ \includegraphics[width=0.3\textwidth]{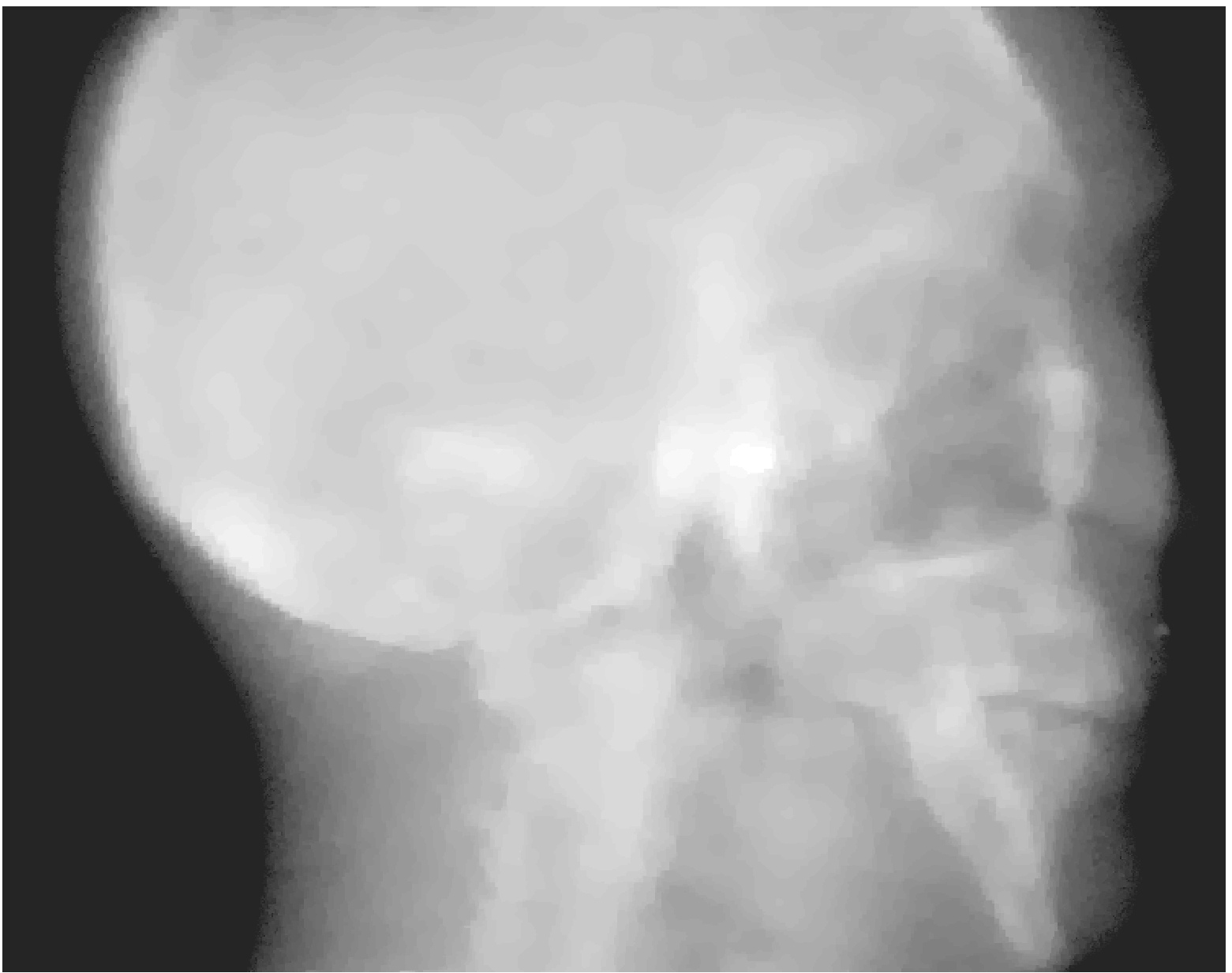} }
\subfigure[Linear, SNR=13.10]{ \includegraphics[width=0.3\textwidth]{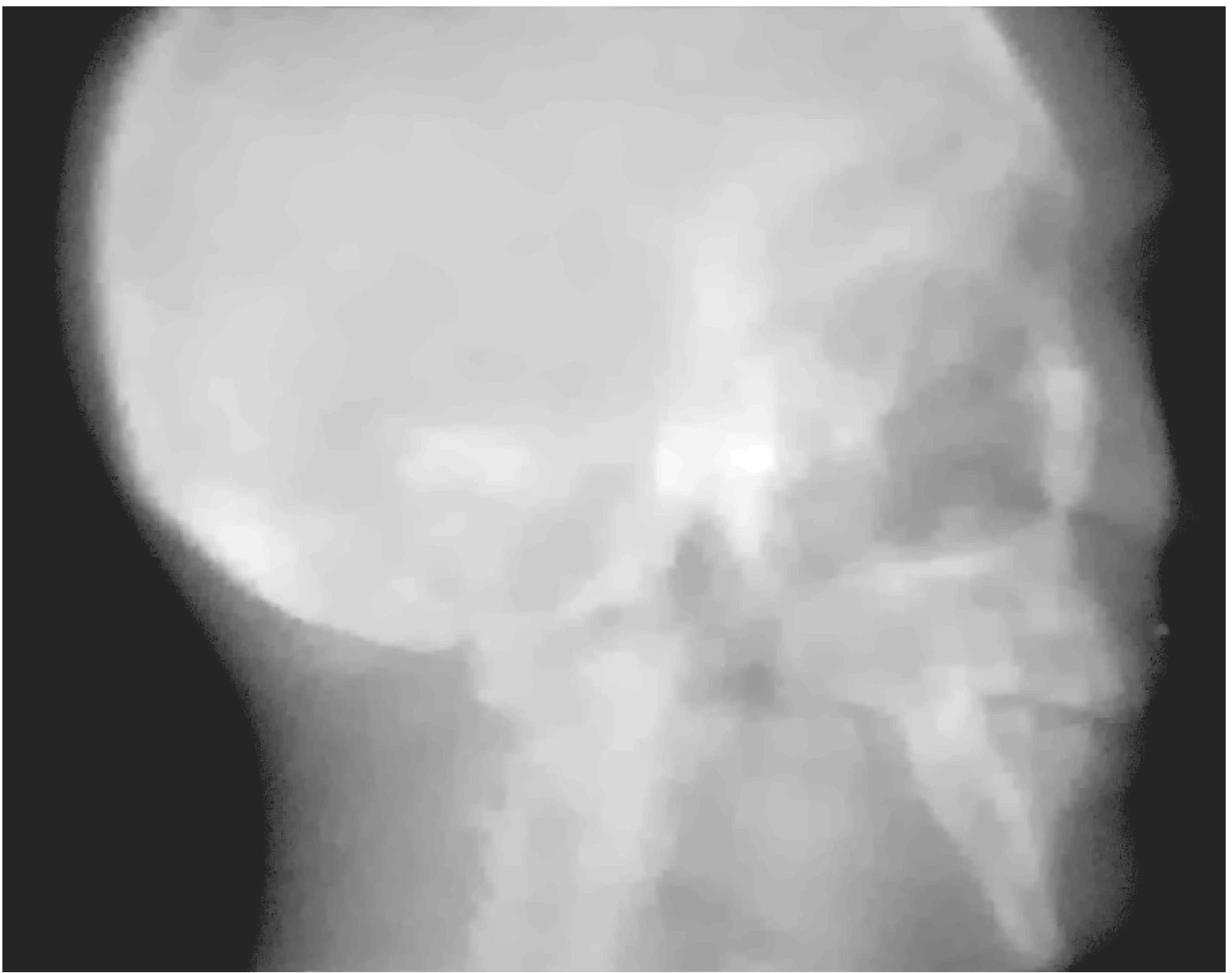} }
\subfigure[Cubic, SNR=13.11]{ \includegraphics[width=0.3\textwidth]{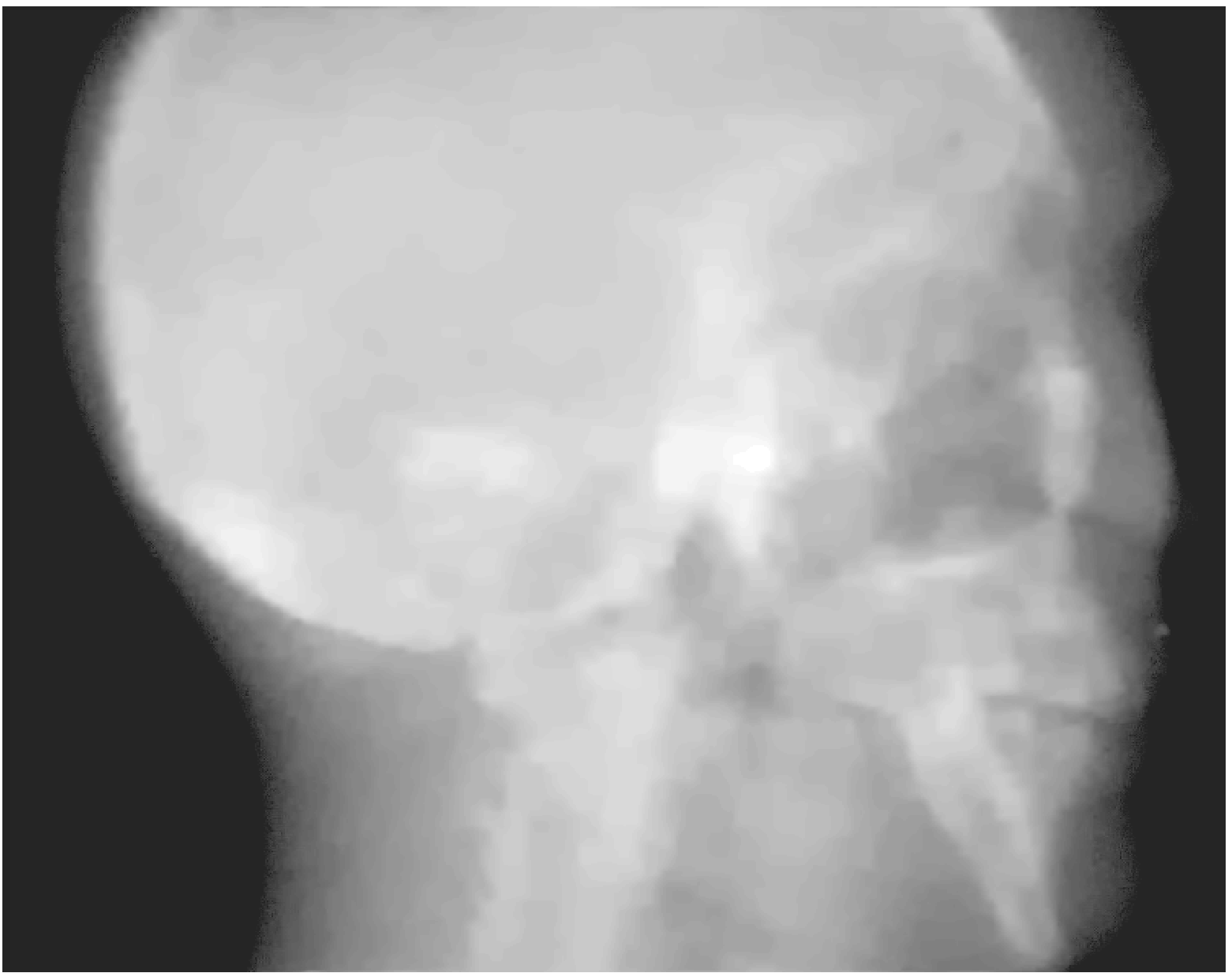} }
\end{center}
\caption{ The comparison between the TV and wavelet frame based models for denoising a projection of the 3D medical data. The first row shows the ground truth, the noisy image and the denoised result of the TV model. The second row shows the denoised results of the Haar wavelet system, the piecewise linear B-spline system and the piecewise cubic B-spline system, respectively. }
\label{medical251_proj}
\end{figure}

\begin{figure}
\begin{center}
\subfigure[Original]{ \includegraphics[width=0.3\textwidth]{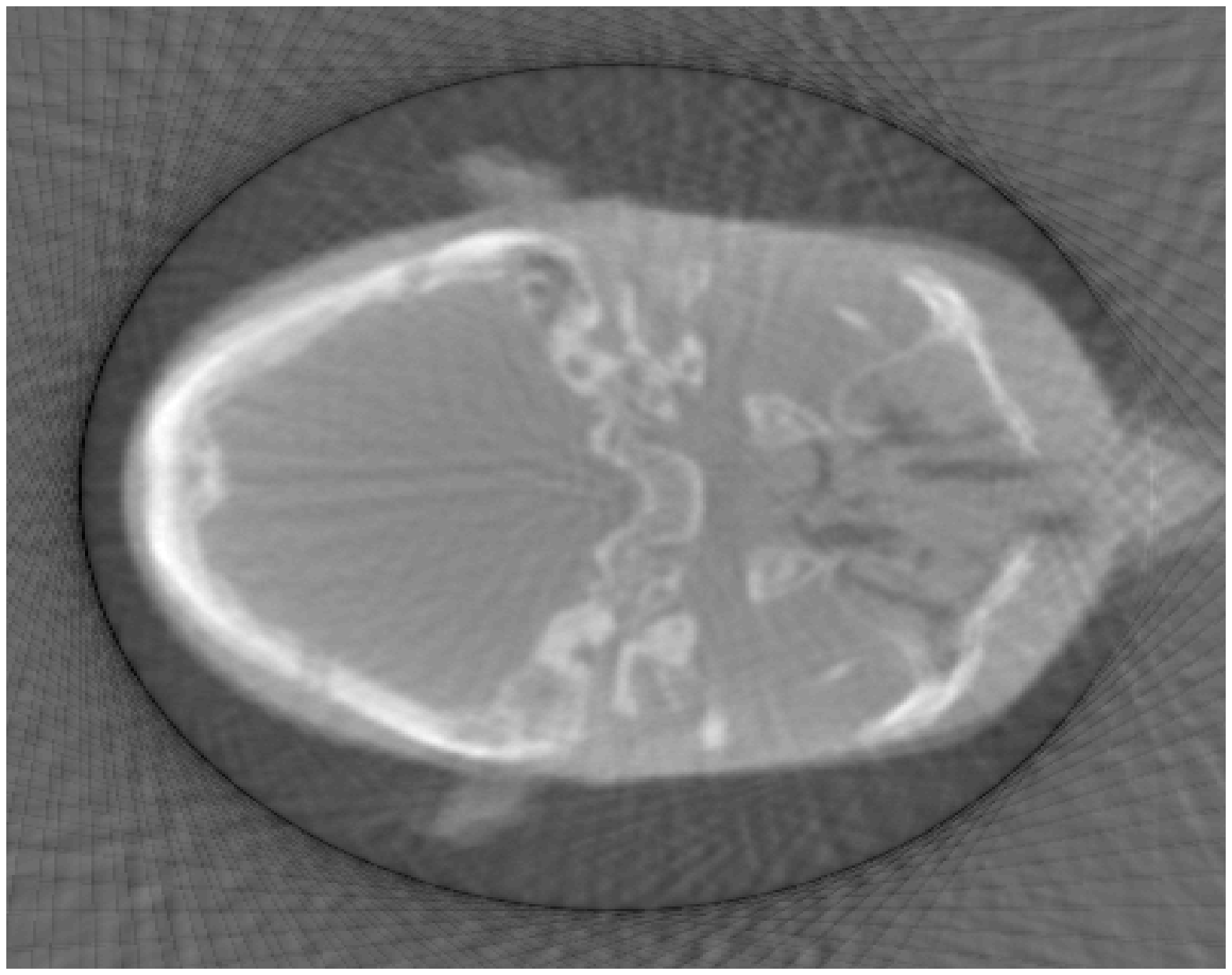} }
\subfigure[Noisy, error=782.50]{ \includegraphics[width=0.3\textwidth]{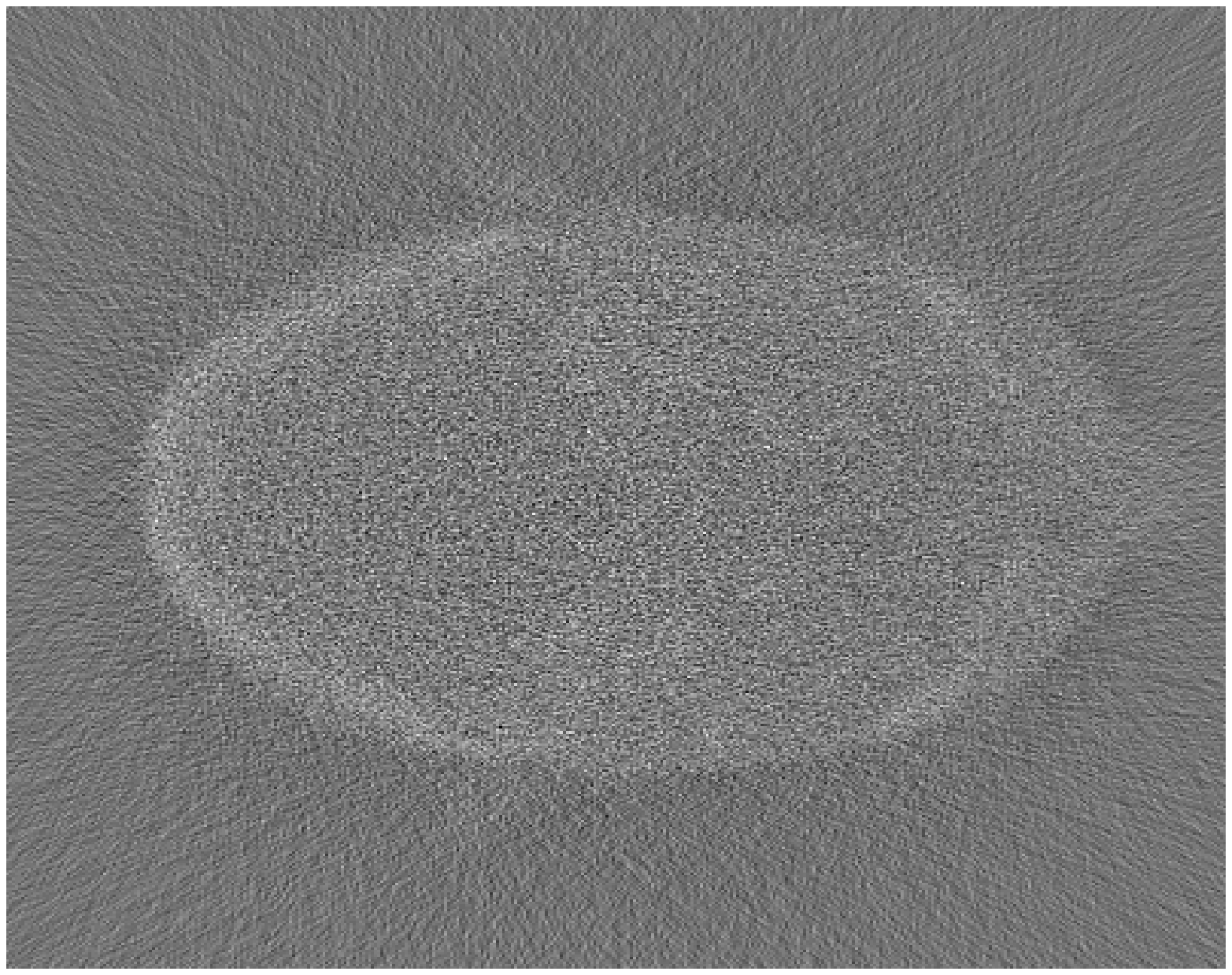} }
\subfigure[TV, error=145.18]{ \includegraphics[width=0.3\textwidth]{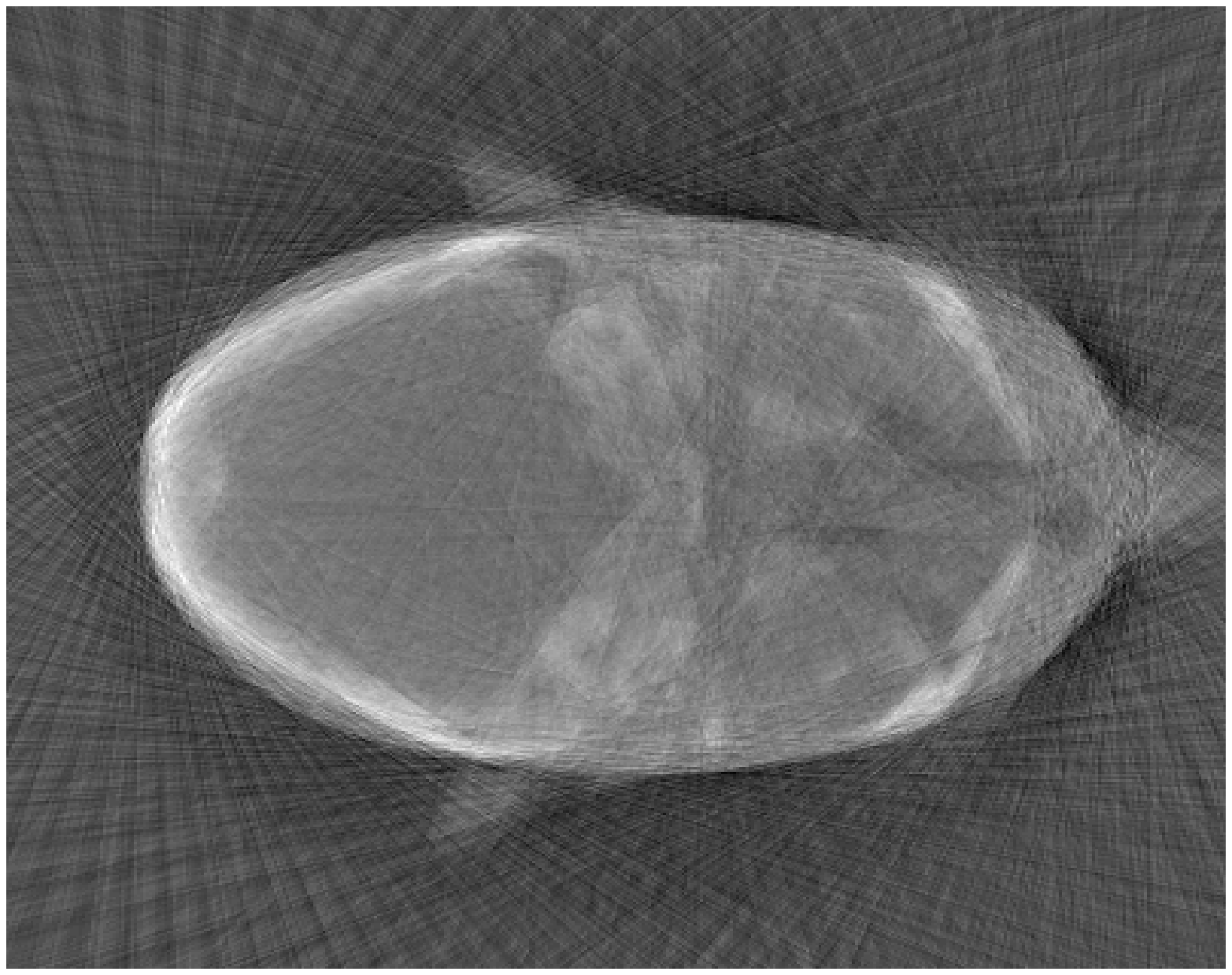} }
\subfigure[Haar, error=141.66]{ \includegraphics[width=0.3\textwidth]{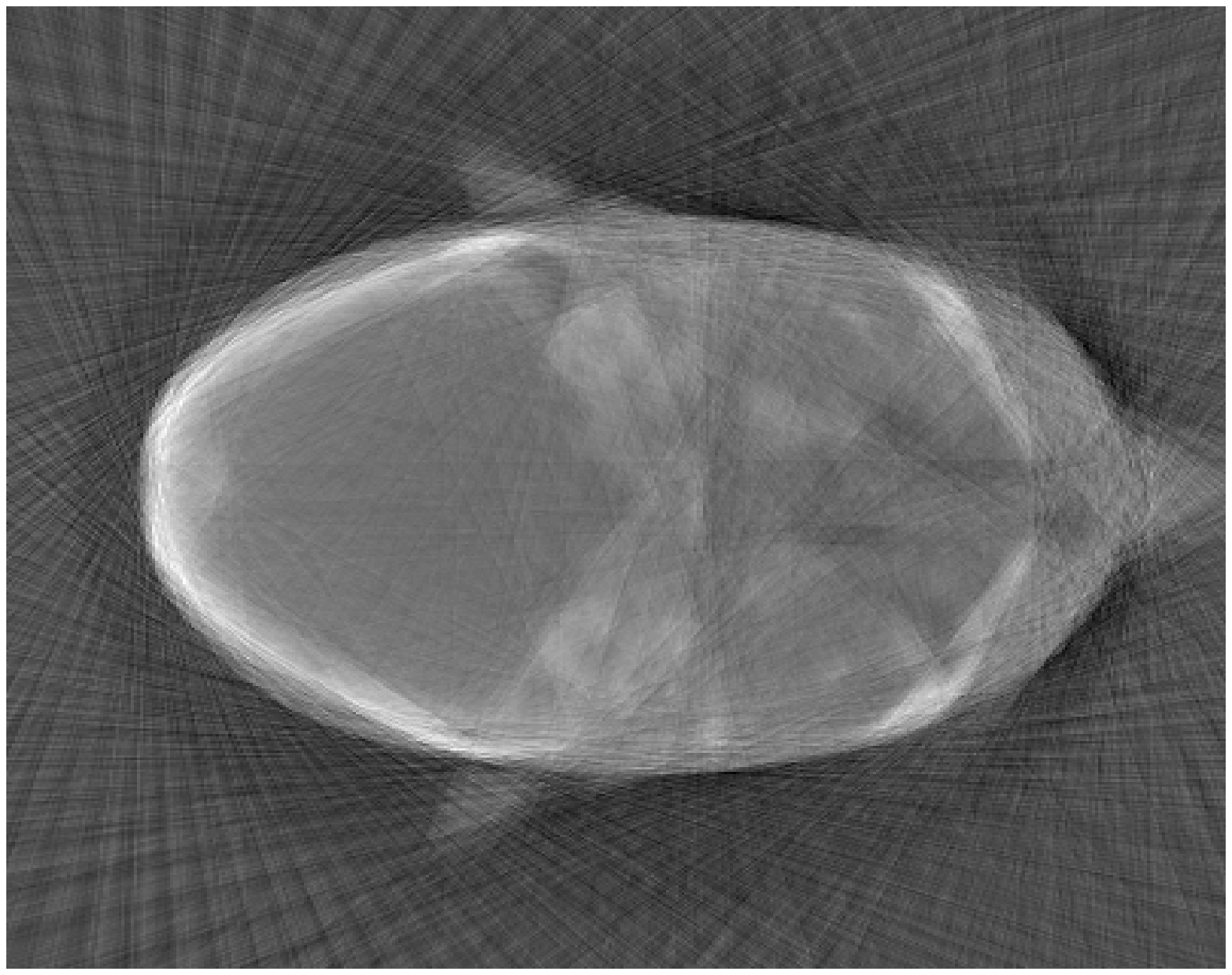} }
\subfigure[Linear, error=137.23]{ \includegraphics[width=0.3\textwidth]{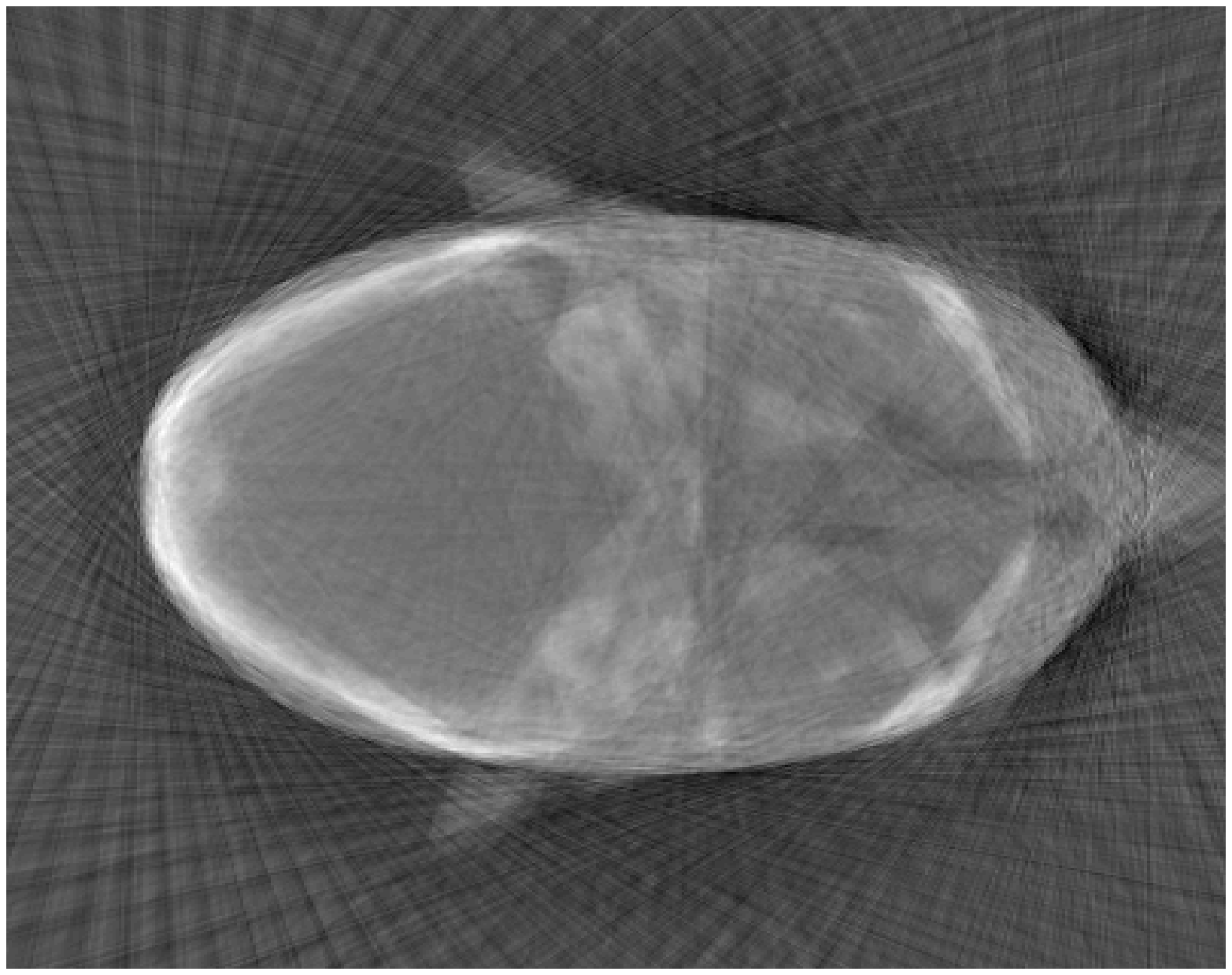} }
\subfigure[Cubic, error=134.83]{ \includegraphics[width=0.3\textwidth]{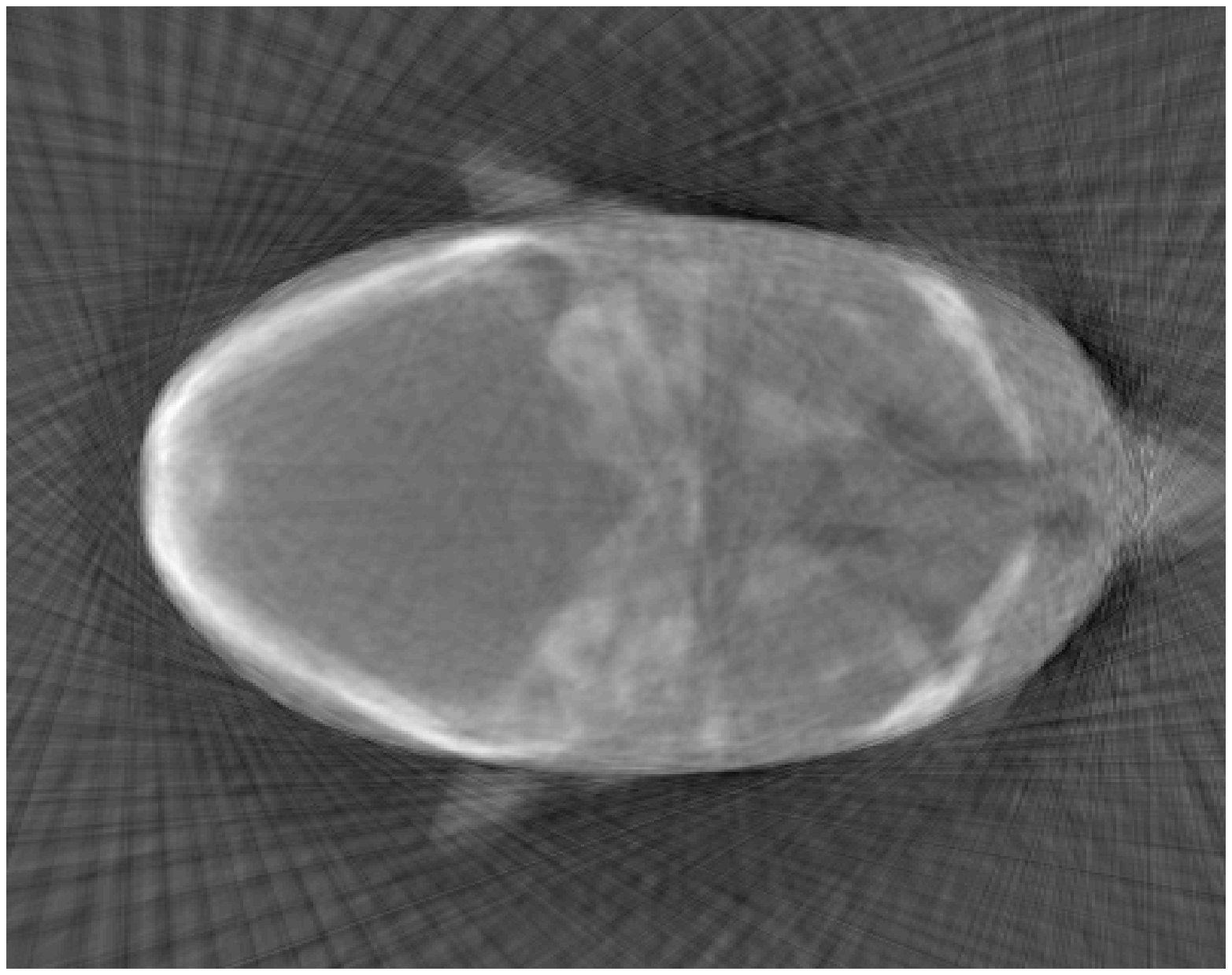} }
\end{center}
\caption{The comparison of the 120'th slice of the reconstructed 3D medical density function. The first row shows slices of the ground truth, the density functions from noisy projection images and denoised projections by the TV model, respectively. The second row shows the density functions from denoised images by the Haar wavelet system, the piecewise linear B-spline and piecewise cubic B-spline systems, respectively. } \label{medical251_rec}
\end{figure}

\begin{figure}
\begin{center}
\subfigure[Original]{ \includegraphics[width=0.23\textwidth]{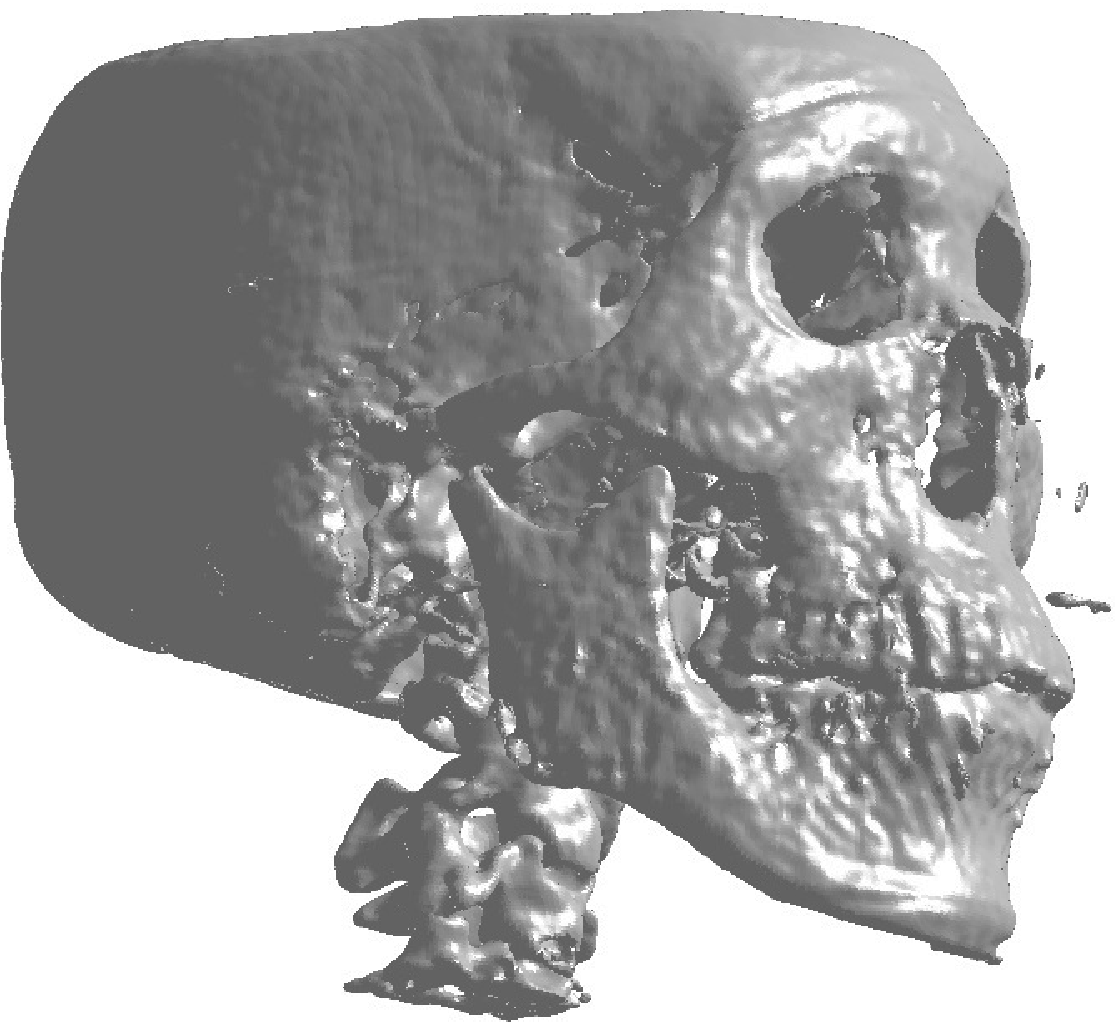}  }
\subfigure[Noisy]{ \includegraphics[width=0.35\textwidth]{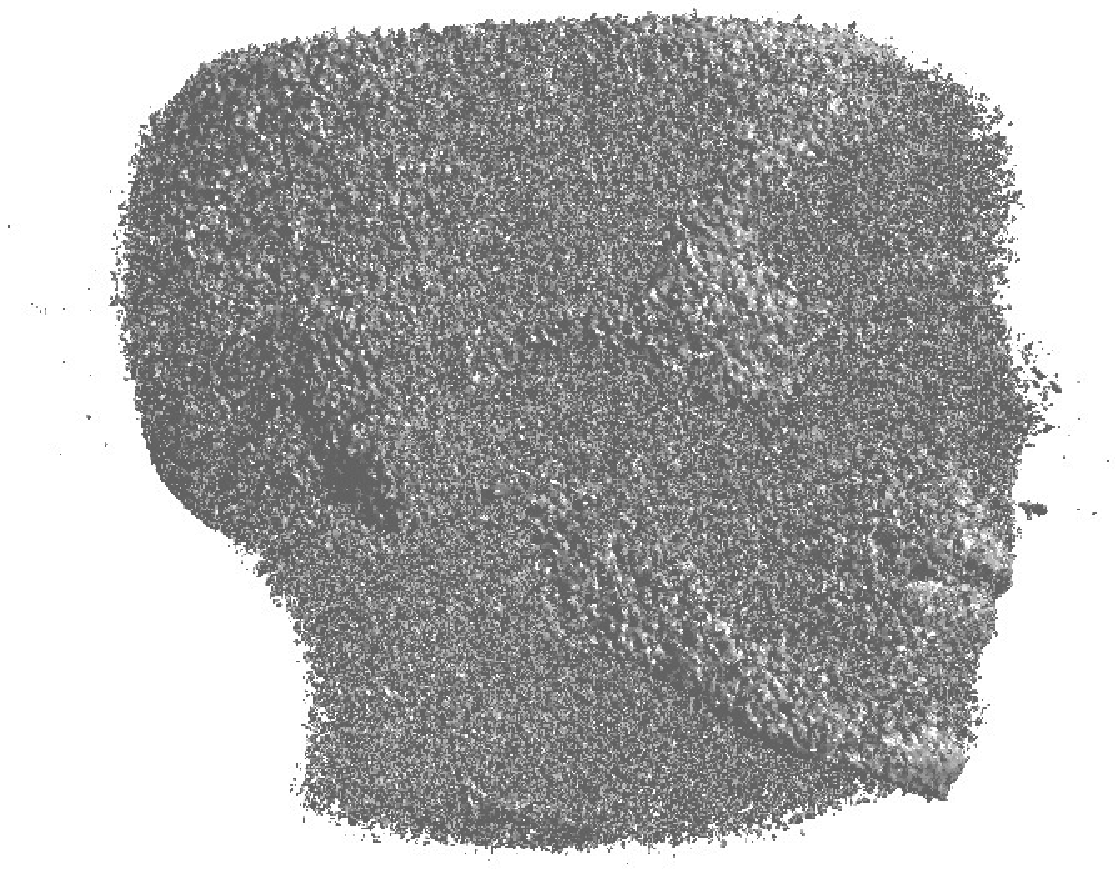} }
\subfigure[TV]{ \includegraphics[width=0.23\textwidth]{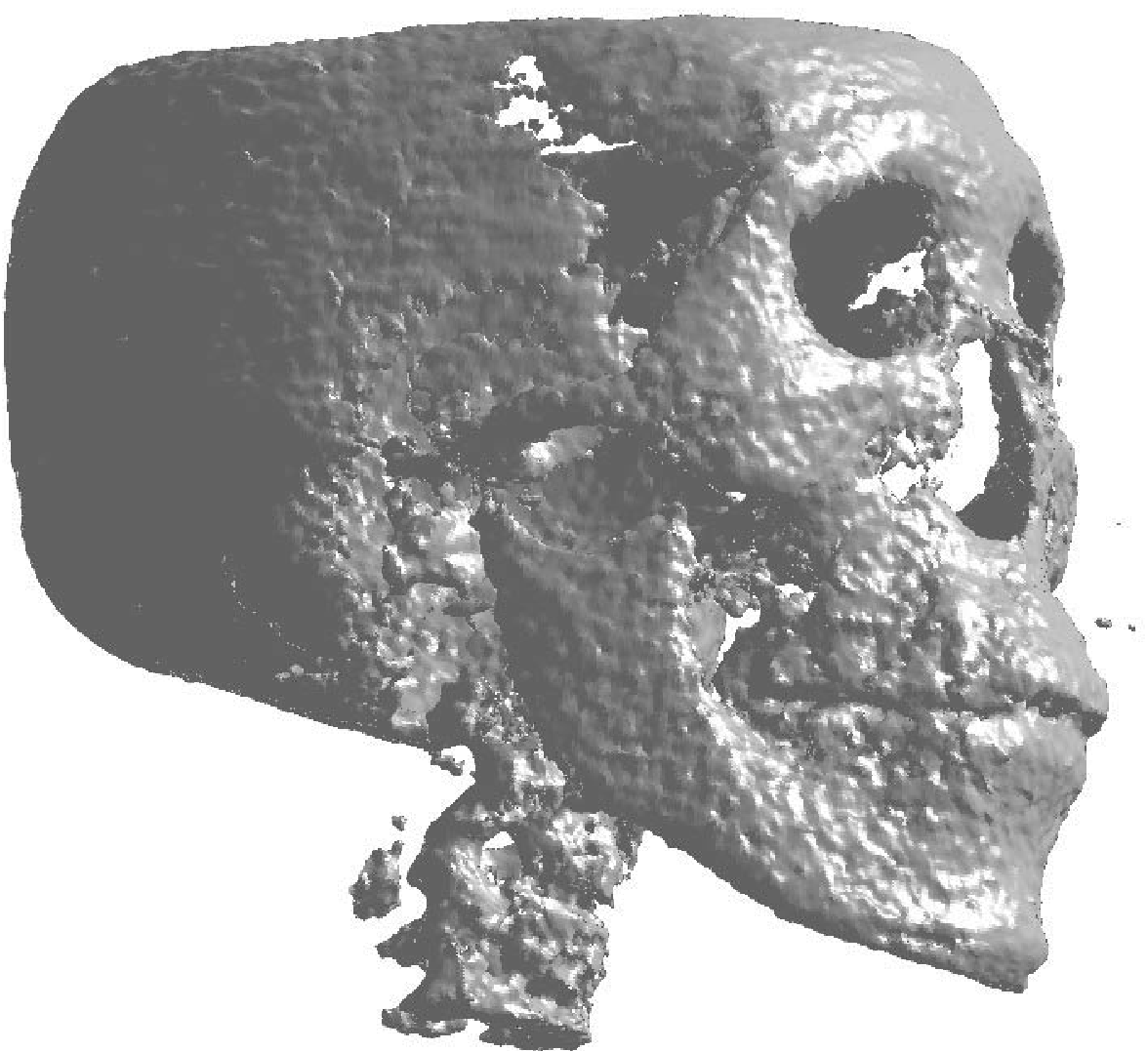} }
\subfigure[Haar]{ \includegraphics[width=0.23\textwidth]{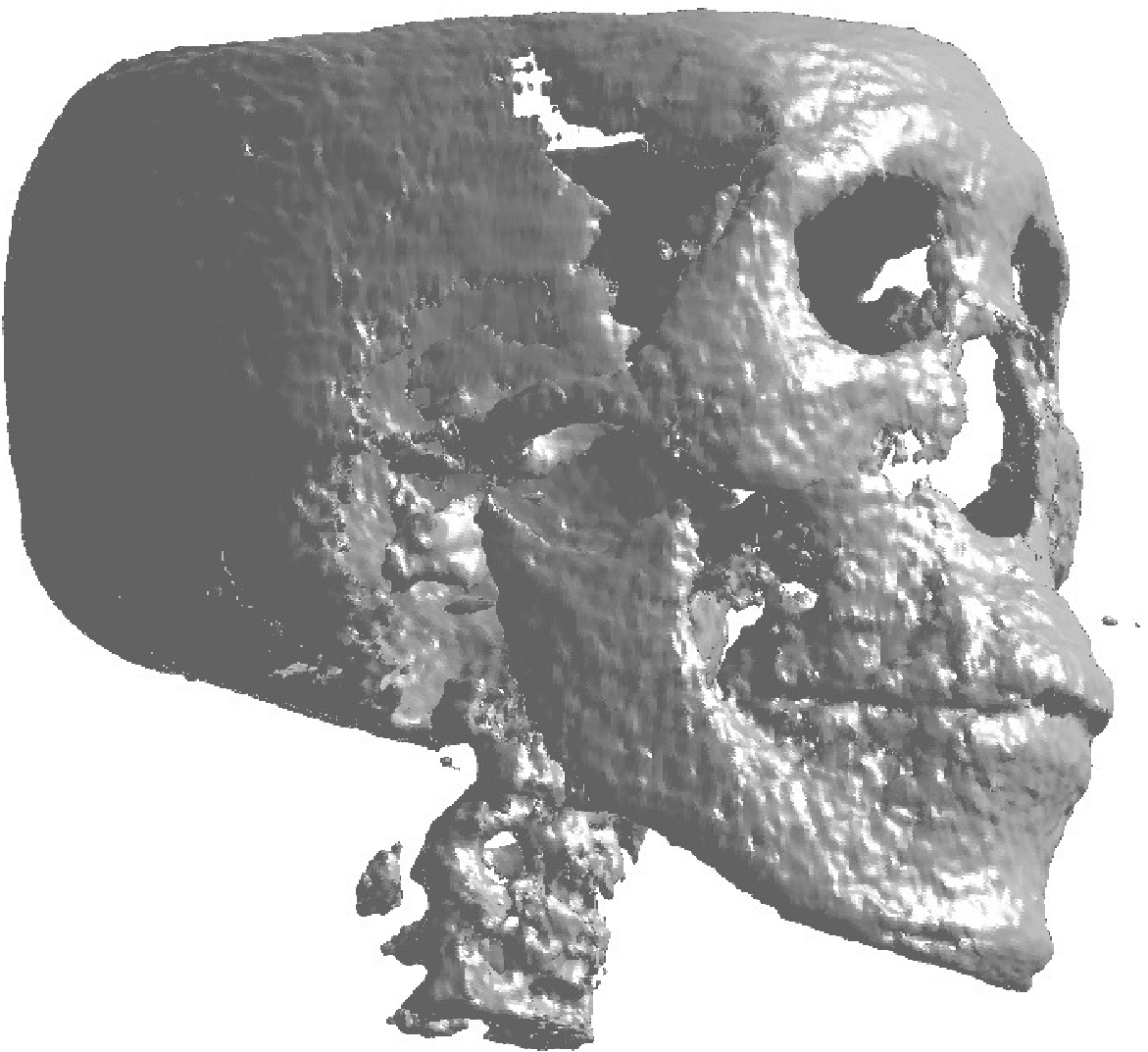} } \ \ \ \
\subfigure[Linear]{ \includegraphics[width=0.23\textwidth]{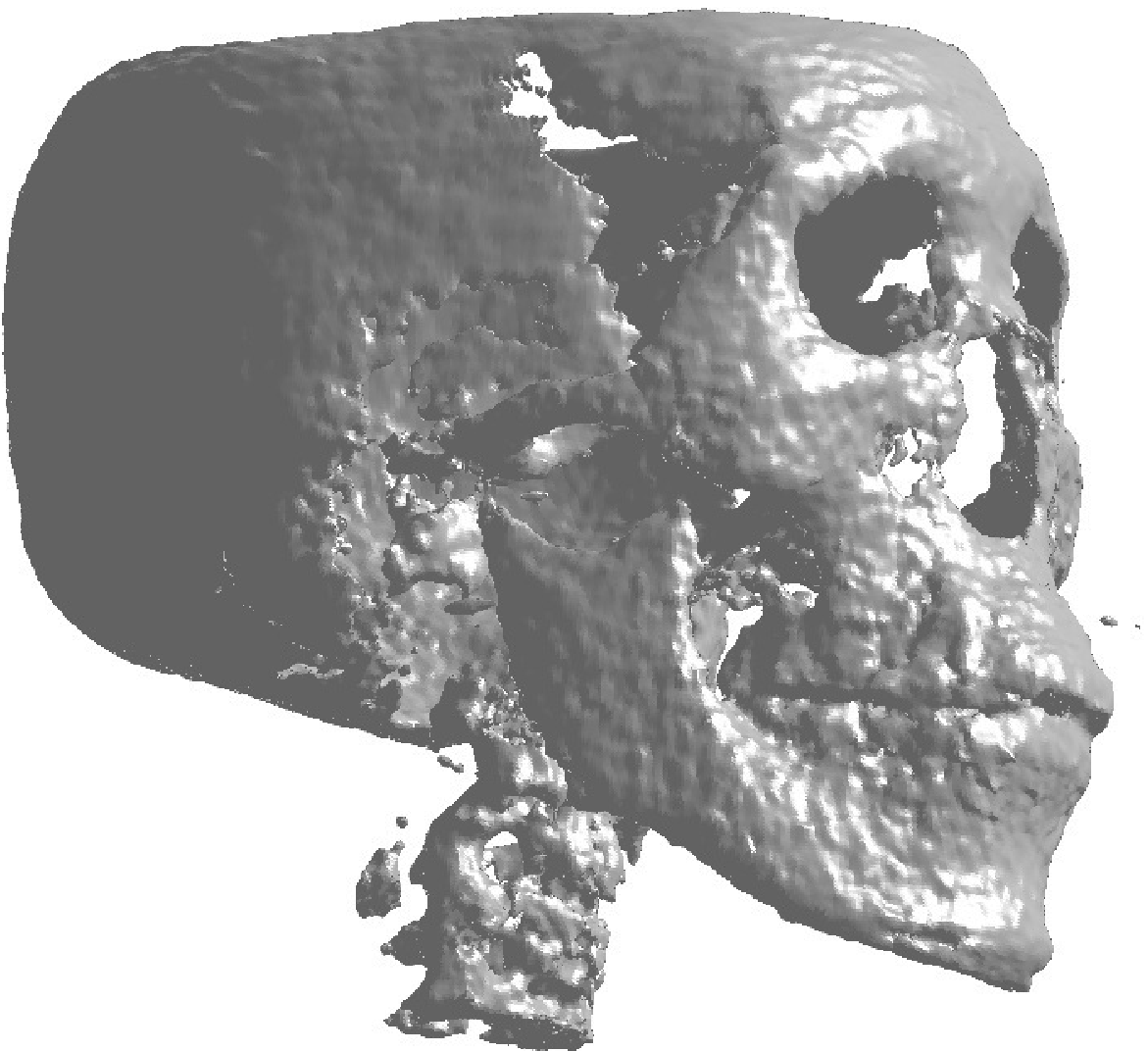} }  \ \ \ \
\subfigure[Cubic]{ \includegraphics[width=0.23\textwidth]{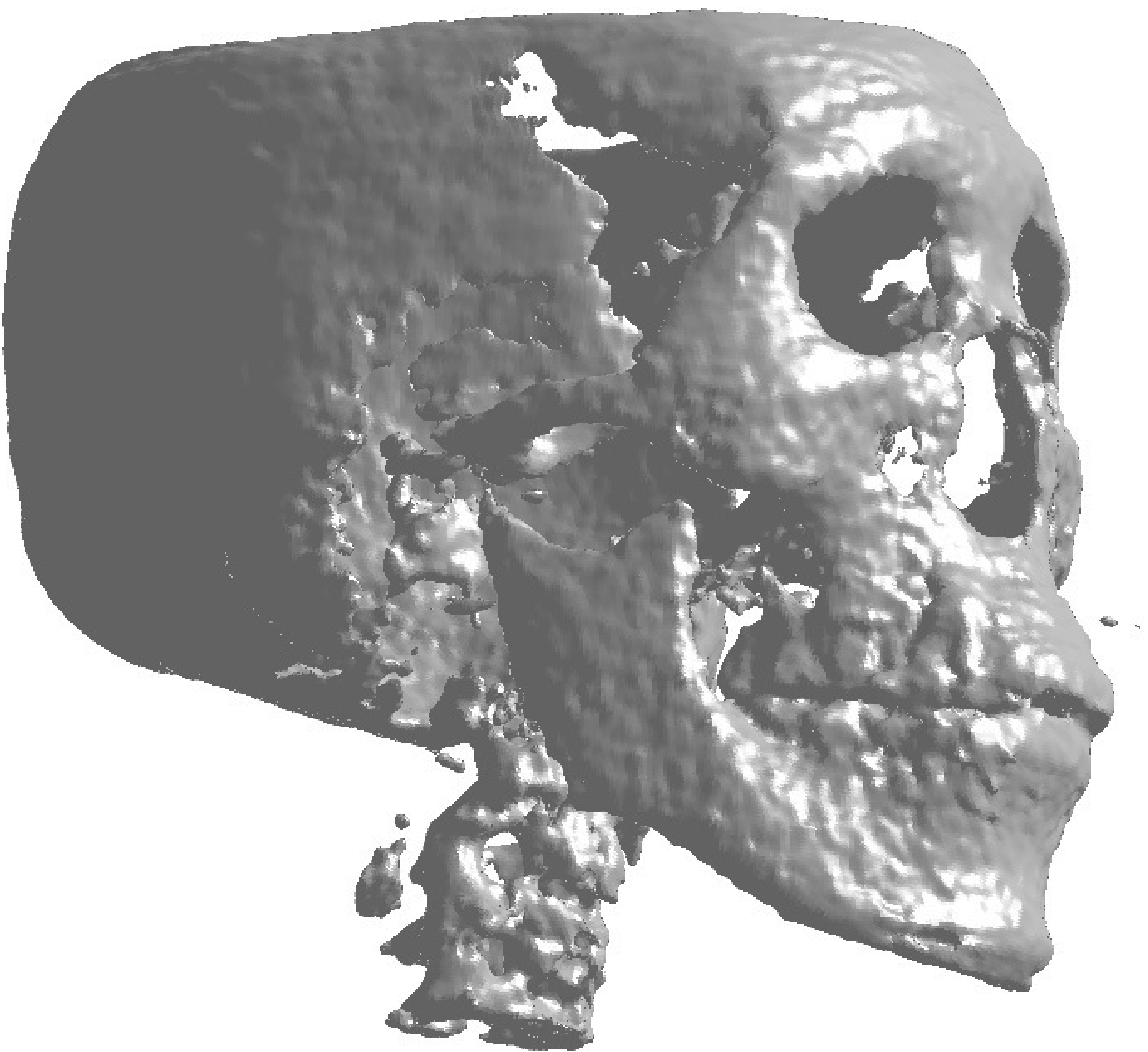} }
\end{center}
\caption{The comparison of isosurface (density function value = 0.09) of the reconstructed 3D medical density function. The first row shows the isosurfaces of the ground truth, the density functions from noisy projections and denoised projections by the TV model, respectively. The second row shows the isosurfaces of the density functions from denoised images by the Haar wavelet system, the piecewise linear B-spline system and the piecewise cubic B-spline system, respectively. } \label{medical251_3d}
\end{figure}

\emph{The SNR values of noisy and denoised versions of some selected projections for all three experiments are summarized in TABLE~\ref{proj_data}; and the Frobenius norms of the differences between reconstructed density functions and the ground truth density function are given in TABLE~\ref{rec_data}}.

\begin{table}[ht]
\centering
\begin{tabular}{|c|c|r|r|r|r|r|}  \hline
Data & Projection \# & Noisy & TV\ \ \  & Haar & Linear & \ Cubic \\
\hline
                        & 100 & 11.44 & 14.52 & - & 19.97  & 20.42   \\
Phantom 2D    & 200 & 14.14 & 16.63 & - & 23.20  & 24.90   \\
                        & 300 & 12.04 & 14.74 & - &  21.30 & 22.82  \\ \hline

                        & 20 & 9.73 & 24.64  & 24.62   & 25.74    & 26.47   \\
Phantom 3D    & 40 & 8.37 & 23.81  & 23.93   & 24.98 & 25.77     \\
                        & 60 & 9.59 & 24.61  & 24.63   & 25.69 & 26.43     \\ \hline

                       & 20 &  6.99  &  12.86 & 12.88  &  12.89   &  12.91  \\
Medical 3D     & 40 &  9.65  &  12.69 & 12.69  &  12.70   &  12.70 \\
                       & 60 &  7.10  &  12.79 & 12.81  &  12.82  &  12.83  \\ \hline
\end{tabular}
\caption{The comparison of the SNR for some selected projections.}
\label{proj_data}
\end{table}

\begin{table}[ht]
\centering
\begin{tabular}{|c|c|r|r|r|r|} \hline
Data & Noisy & TV & \ \ Haar & Linear & \ Cubic \\ \hline
Phantom 2D    & 61.53       &  39.12     & -              &  13.56      & 11.32      \\
Phantom 3D    & 346.33     &  56.98     &  56.65   &   49.47 &  44.11     \\
Medical 3D      & 782.50     &  145.18   &  141.66    &
137.23     &  134.83     \\ \hline
\end{tabular}
\caption{The comparison of the Frobenius norm of the difference between reconstructed object and the ground truth object.}
\label{rec_data}
\end{table}

\end{eg}

\section{Conclusion}
Transmission imaging is widely applied in astronomy and biomedical sciences for macro and micro scale objects, whose physical mechanisms and mathematical models are quite different from reflection imaging frequently used in our everyday life for common scale objects. In this paper, we improved the existing continuity analysis of images generated by transmission imaging in two aspects. First, we consider both parallel and divergent beam geometries, i.e., two basic and important geometries in transmission imaging, while existing analysis applies to only the parallel beam geometry. Second, we prove the continuity property of images generated by transmission under much weaker conditions which admit almost all cases in applications, while previous analysis excludes some common cases such as cylindrical objects. Our analysis shows that images by transmission imaging with both parallel and divergent beam geometries are almost surely continuous functions, even if the density functions of the imaged objects are discontinuous (discontinuous density functions are very common). This is quite different from reflection imaging where images are usually modeled as discontinuous functions. Although the central topic in transmission imaging is the reconstruction of density functions of objects, processing of the image (projection) data before reconstruction is also sometimes important due to the fact that these data usually involve degradations such as Poisson noise. There are many methods in commercial dealing with the transmission images first before performing the reconstruction, and understanding the structures of transmission images is important because better denoised transmission images gives better reconstruction results. Our theoretical analysis may help us to understand the structures of images generated by transmission imaging and provide some information for choosing, designing and testing image processing techniques for transmission imaging. Taking into accounts our analysis, we compared two popular image denoising methods for Poisson noise removal. Numerical experiments are provided to verify our theoretical analysis.




\begin{thebibliography}{1}
\bibitem{Aubert06} {\sc G. Aubert and P. Kornprobst}, {\em Mathematical Problems in Image Processing: Partial Differential Equations and the Calculus of Variations (second edition)}, Series of Applied Mathematical Sciences, 147, Springer-Verlag, 2006.


\bibitem{Borsdorf08} {\sc A. Borsdorf, R. Raupach, T. Flohr, and J. Hornegger}, {\em Wavelet based noise reduction in CT images using correlation analysis}, IEEE Transactions on Medical Imaging, 27 (2008), pp.~1685--1703.

\bibitem{Brune09} C. Brune, A. Sawatzky, and M. Burger, \emph{Bregman-EM-TV methods with application to optical nanoscopy}, Lecture Notes in Computer Science, 5567 (2009), pp.~235--246.

\bibitem{Buades05} A. Buades, B. Coll, and J.M. Morel, \emph{A review of image denoising algorithms, with a new one}, Multiscale Modeling \& Simulation, 4 (2005), pp.~490--530.


\bibitem{Cai09} J. Cai, S. Osher, and Z. Shen, \emph{Split Bregman methods and frame based image restoration}, Multiscale Modeling \& Simulation, 8 (2009), pp.~337--369.

\bibitem{CDOS12} J.F. Cai, B. Dong, S. Osher, and Z. Shen, \emph{Image restoration: total variation, wavelet frames, and beyond}, Journal of the American Mathematical Society, 25 (2012), pp.~1033-1089.

\bibitem{Carter01} J.L. Carter, {\em Dual Methods for Total Variation-Based Image Restoration}, Ph.D. thesis, University of California, Los Angeles, 2001.


\bibitem{Chambolle04} A. Chambolle, \emph{An algorithm for total variation minimization and applications}, Journal of Mathematical Imaging and Vision, 20 (2004), pp.~89--97.

\bibitem{Chan99} T.F. Chan, G.H. Golub, and P. Mulet, \emph{A nonlinear primal-dual method for total variation-based image restoration}, SIAM Journal on Scientific Computing, 20 (1999), pp.~1964--1977.



\bibitem{Chan05} T. Chan, S. Esedoglu, F. Park, and A. Yip, \emph{Recent developments in total variation image restoration} in Handbook of Mathematical Models in Computer Vision. Springer Verlag, 2005.

\bibitem{Chan07} R.H. Chan, and K. Chen, \emph{Multilevel algorithms for a Poisson noise removal model with total variation regularization}, International Journal of Computer Mathematics, 84 (2007), pp.~1183--1198.

\bibitem{Dong10} B. Dong, and Z.W. Shen, \emph{MRA-based Wavelet Frames and Applications}, IAS Lectures Notes Series, Summer Program on ``The Mathematics of Image Processing",  Park City Mathematics Institute, 2010.

\bibitem{Donoho95} D.L. Donoho, \emph{Denoising by soft-thresholding}, IEEE Transactions on Information Theory, 41 (1995), pp.~613--627.

\bibitem{Esser09} E. Esser, \emph{Applications of Lagrangian-Based Alternating Direction Methods and Connections to Split Bregman}, UCLA CAM Report, 09-31, 2009.

\bibitem{Frank96} J. Frank, \emph{Three-Dimensional Electron Microscopy of Macromolecular Assemblies}, Academic Press, 1996.

\bibitem{Figueiredo09} M.A.T. Figueiredo, and J.M. Bioucas-Dias, \emph{Deconvolution of Poissonian images using variable splitting and augmented Lagrangian optimization}, IEEE Workshop on Statistical Signal Processin, 2009.

\bibitem{Glowinski89} R. Glowinski and P. Le Tallec, \emph{Augmented Lagrangians and Operator-Splitting Methods in Nonlinear Mechanics}, SIAM, Philadelphia, 1989.

\bibitem{Goldstein08} T. Goldstein and S. Osher, \emph{The split Bregman method for L1 regularized problems}, SIAM Journal on Imaging Sciences, 2 (2009), pp.~323--343.

\bibitem{Helgason99} S. Helgason, \emph{The Radon Transform}, Birkhauser Boston, 1999.



\bibitem{JS1}H. Ji, Z. Shen, and Y. Xu, \emph{Wavelet frame based image restoration with missing/damaged pixels}, East Asia Journal on Applied Mathematics, 1 (2011), pp.~108--131.

\bibitem{Kak88} A.C. Kak and M. Slaney, \emph{Principles of Computerized Tomographic Imaging}, IEEE Press, 1988.

\bibitem{Kim13} K.S. Kim,  http://www.mathworks.com/matlabcentral/fileexchange/35548-3d-cone-beam-ct-cbct-projection-backprojection-fdk-mlem-reconstruction-matlab-codes-for-students, 2013

\bibitem{Riviere06} P.J. La Riviere, J. Bian, and P. A. Vargas, \emph{Penalized-likelihood sinogram restoration for computed tomography}, IEEE Transactions on Medical Imaging, 25 (2006), pp.~1022-1036.




\bibitem{Le07} T. Le, R. Chartrand, and T.J. Asaki, \emph{A variational approach to reconstructing images corrupted by Poisson noise}, Journal of Mathematical Imaging and Vision, 27 (2007), pp.~257--263.

\bibitem{Li09} Y. Li, and S. Osher, \emph{A new median formula with applications to PDE based denoising}, Communications in Mathematical Sciences, 7 (2009), pp.~741--753.


\bibitem{milnor} John Milnor, {\em Topology from the Differentiable Viewpoint}, Princeton Landmarks in Mathematics. Princeton, NJ: Princeton University Press, 1997.

\bibitem{Mumford89} D. Mumford and J. Shah, \emph{Optimal approximations by piecewise smooth functions and associated variational problems}, Communications on Pure and Applied Mathematics, 42 (1989), pp.~577--685.

\bibitem{Natterer01} F. Natterer, \emph{The Mathematics of Computerized Tomography}, in Classics in Applied Mathematics, SIAM, 2001.

\bibitem{Osher10} S. Osher, Y. Mao, B. Dong, and W. Yin, \emph{Fast linearized Bregman iteration for compressive sensing and sparse denoising}, Communications in Mathematical Sciences, 8 (2010), pp.~93--111.

\bibitem{Pan09} Xiaochuan Pan, Emil Y Sidky, and Michael Vannier, \emph{Why do commercial CT scanners still employ traditional, filtered back-projection for image reconstruction?} Inverse Problems, 25 (2009), p.~123009,  2009.

\bibitem{Panin99} V.Y. Panin, G.L. Zeng, and G.T. Gullberg, \emph{Total variation regulated EM algorithm [SPECT reconstruction]}, IEEE Transactions on Nuclear Science, 46 (1999), pp.~2202--2210.

\bibitem{Perona90} P. Perona and J. Malik, \emph{Scale-space and edge detection using anisotropic diffusion}, IEEE Transactions on Pattern Analysis and Machine Intelligence, 12 (1990),  pp.~629--639.

\bibitem{Radon86} J. Radon, \emph{On the determination of functions from their integral values along certain manifolds}, IEEE Transactions on Medical Imaging, 5 (1986), pp.~170--176.

\bibitem{RS1}A. Ron and Z. Shen, \emph{Affine systems in $L_2(R^d)$: the analysis of the analysis operator}, Journal of Functional Analysis, 148 (1997), pp.~408--447.

\bibitem{Rudin92} L. Rudin, S. Osher, and E. Fatemi, \emph{Nonlinear total variation based noise removal algorithms}, Physica D, 60 (1992), pp.~259--268.

\bibitem{Schaap08} M. Schaap, A.M. Schilham, K. .I. Zuiderveld, et al, \emph{Fast noise reduction in computed tomography for improved 3-D visualization}, IEEE Transactions on Medical Imaging, 27 (2008), pp.~1120-1129.



\bibitem{Setzer09} S. Setzer, \emph{Split Bregman algorithm, Douglas-Rachford splitting and frame Shrinkage},  Lecture Notes in Computer Science, 5567 (2009), pp.~464--476.

\bibitem{Setzer09b} S. Setzer, G. Steidl, and T. Teuber, \emph{Deblurring Poissonian images by split Bregman techniques}, Journal of Visual Communication and Image Representation, 21 (2010), pp.~193--199.

\bibitem{Shen10} Z. Shen, \emph{Wavelets frames and image restorations}, in Proceedings of the International Congress of Mathematicians, Hyderabad, India, 2010.

\bibitem{Shtok11} J. Shtok, M. Elad, and M. Zibulevsky, \emph{Sparsity-based sinogram denoising for low-dose computed tomography}, in Acoustics, Speech and Signal Processing (ICASSP), 2011 IEEE International Conference on, pp.~569--572, 2011.

\bibitem{Sidky06} E.Y. Sidky, C.M. Kao, and X. Pan, \emph{Accurate image reconstruction fron few-views and limited-angle data in divergent-beam CT}, Journal of X-Ray Science and Technology, 14 (2006), pp.~119-139.

\bibitem{Sidky12} E.Y. Sidky, J.H. J¿rgensen, and X. Pan \emph{Convex optimization problem prototyping for image reconstruction in computed tomography with the ChambolleÐPock algorithm}, Physics in Medicine and Biology, 57 (2012), p.~3065.

\bibitem{SWBMW} G. Steidl, J. Welckert, T. Brox, P. Mrazek, and M. Welk, \emph{On the equivalence of soft wavelet shrinkage, total variation diffusion, total variation regularization, and SIDEs}, SIAM Journal on Numerical Analysis, 42 (2004), pp.~686-713.


\bibitem{Tai09} X.-C. Tai and C. Wu, \emph{Augmented Lagrangian method, dual methods and split Bregman iteration for ROF model}, Lecture Notes in Computer Science, 5567 (2009), pp.~502-513.

\bibitem{Wang2006} J. Wang, T. Li, H. Lu, and Z. Liang, \emph{Penalized weighted least-square approach to sinogram noise reduction and image reconstruction for low dose X-ray computed tomography}, IEEE Transactions on Medical Imaging, 25 (2006), pp.~1272-1283.

\bibitem{Wang08} Y. Wang, J. Yang, W. Yin, and Y. Zhang, \emph{A vew alternating minimization algorithm for total variation image reconstruction}, SIAM Journal on Imaging Sciences, 1 (2008), 248--272.

\bibitem{Weickert98} J. Weickert, \emph{Anisotropic Diffusion in Image Processing}, Tuebner Stuttgart, 1998.

\bibitem{Willemink13} M.J. Willemink, P.A. de Jong, T. Leiner, L.M. de Heer, R.A. Nievelstein, R.P. Budde, and A.M. Schilham, \emph{Iterative reconstruction techniques for computed tomography part 1: technical principles}, European radiology, 23 (2013), pp.~1623-1631.

\bibitem{Wu13} C. Wu, \emph{On the continuity of images by transmission imaging}, Communications in Mathematical Sciences, 11 (2013), pp.~573--595.

\bibitem{Wu10} C. Wu, and X.C. Tai, \emph{Augmented Lagrangian method, dual methods, and split Bregman iteration for ROF, vectorial TV, and high order models}, SIAM Journal on Imaging Sciences, 3 (2010), pp.~300--339.

\bibitem{Wu11} C. Wu, J. Zhang, X.C. Tai, \emph{Augmented Lagrangian method for total variation restoration with non-quadratic fidelity}, Inverse Problems and Imaging, 5 (2011), pp.~237--261.

\bibitem{Yan13} M. Yan, A. Bui, J. Cong, and L.A. Vese, \emph{General convergent expectation maximization (EM)-type algorithms for image reconstruction}, Inverse Problems and Imaging, 7 (2013), pp.~1007--1029.


\bibitem{Yin08} W.T. Yin, S. Osher, D. Goldfarb and J. Darbon, \emph{Bregman iterative algorithms for compressed sensing and related problems}, SIAM Journal on Imaging Sciences, 1 (2008), pp.~143--168.

\bibitem{Yin09} W.T. Yin, \emph{Analysis and generalizations of the linearized Bregman method}, SIAM Journal on Imaging Sciences, 3 (2010), pp.~856--877.

\bibitem{Zanella09} R. Zanella, P. Boccacci, L. Zanni, and M. Bertero, \emph{Efficient gradient projection methods for edge-preserving removal of Poisson noise}, Inverse Problems, 25 (2009), p.~045010.

\bibitem{Zhu08} M. Zhu, S.J. Wright, T.F. Chan, \emph{Duality-based algorithms for total-variation-regularized image restoration}, Computational Optimization and Applications, 47 (2010), pp.~377-400.
\end{thebibliography}
\end{document}